\documentclass[a4paper, 12pt]{article}
\usepackage{mathtext}
\usepackage{amsmath}
\usepackage{amsfonts}
\usepackage{amssymb}
\usepackage{mathrsfs}
\usepackage{amsthm}
\usepackage{titling}

\usepackage{float}
\usepackage{graphicx}
\PassOptionsToPackage{usenames,dvipsnames}{color}
\usepackage{color}
\usepackage[colorlinks=true, linkcolor=blue, citecolor=red, unicode=true, pageanchor=false,pdfstartview=FitH]{hyperref}

\DeclareMathOperator{\sign}{sign}
\DeclareMathOperator{\Ker}{Ker}
\newcommand{\df}{\buildrel\mathrm{def}\over=}
\newcommand{\Bell}{\boldsymbol{B}_{\eps}}
\newcommand{\BMO}{\mathrm{BMO}}
\newcommand{\eps}{\varepsilon}
\newcommand{\W}{\mathfrak{W}_{\eps_0}}
\newcommand{\Wone}{\mathfrak{W}_{1}}
\newcommand{\Drt}{D_{\mathrm{R}}}
\newcommand{\Dlt}{D_{\mathrm{L}}}

\newcommand{\Cii}[1]{_{{}_{\scriptstyle #1}}}
\newcommand{\cii}[1]{_{{}_{#1}}}
\newcommand{\BNorm}[1]{\|#1\|_{{}_{\BMO(I)}}}
\newcommand{\Av}[2]{\langle {#1}\rangle_{{}_{\scriptstyle #2}}}
\newcommand{\av}[2]{\langle {#1}\rangle_{{}_{#2}}}
\newcommand{\Ch}{\Omega_{\mathrm{ch}}}
\newcommand{\Bch}{B^{\mathrm{ch}}}
\newcommand{\Alv}{\Omega_{\mathrm{cup}}}

\newcommand{\Balv}{B^{\mathrm{cup}}}
\newcommand{\Rt}{\Omega_{\mathrm{R}}}
\newcommand{\Lt}{\Omega_{\mathrm{L}}}
\newcommand{\LtRt}{\Omega_{\mathrm{LR}}}
\newcommand{\RtLt}{\Omega_{\mathrm{RL}}}
\newcommand{\Brt}{B^{\mathrm{R}}}
\newcommand{\Blt}{B^{\mathrm{L}}}
\newcommand{\Bltrt}{B^{\mathrm{LR}}}
\newcommand{\Brtlt}{B^{\mathrm{RL}}}
\newcommand{\Ang}{\Omega_{\mathrm{ang}}}
\newcommand{\Bang}{B^{\mathrm{ang}}}
\newcommand{\RTroll}{\Omega_{\mathrm{tr},\mathrm{R}}}
\newcommand{\LTroll}{\Omega_{\mathrm{tr},\mathrm{L}}}
\newcommand{\RtRt}{\Omega_{\mathrm{RR}}}
\newcommand{\LtLt}{\Omega_{\mathrm{LL}}}
\newcommand{\Brtrt}{B^{\mathrm{RR}}}
\newcommand{\Bltlt}{B^{\mathrm{LL}}}
\newcommand{\Brtroll}{B^{\mathrm{tr},\mathrm{R}}}
\newcommand{\Bltroll}{B^{\mathrm{tr},\mathrm{L}}}
\newcommand{\ex}{\varphi}
\newcommand{\dc}{\gamma}
\newcommand{\mrt}{m\cii{{\mathrm{R}}}}
\newcommand{\mlt}{m\cii{{\mathrm{L}}}}
\theoremstyle{plain}
\theoremstyle{plain}\newtheorem{Le}{Lemma}[section]
\theoremstyle{definition}\newtheorem{Def}[Le]{Definition}
\theoremstyle{plain}\newtheorem{St}[Le]{Statement}
\theoremstyle{plain}\newtheorem{Th}[Le]{Theorem}
\theoremstyle{plain}\newtheorem{Cor}[Le]{Corollary}
\theoremstyle{plain}\newtheorem{Prop}[Le]{Proposition}
\theoremstyle{remark}\newtheorem{Rem}[Le]{Remark}

\numberwithin{equation}{section}

\author{Paata~Ivanisvili\thanks{Supported by Chebyshev Laboratory (SPbU), RF Government grant no. 11.G34.31.0026.}
\and Nikolay~N.~Osipov\thanksmark{1}\thanksgap{1ex}\thanks{Supported by RFBR, grant no. 11-01-00526.}\thanksgap{1ex}\thanks{Supported by Rokhlin grant.}
\and Dmitriy~M.~Stolyarov\thanksmark{1}\thanksgap{1ex}\thanksmark{2}
\and Vasily~I.~Vasyunin\thanks{Supported by RFBR, grant no. 11-01-00584.} 
\and Pavel~B.~Zatitskiy\thanksmark{1}\thanksgap{1ex}\thanksmark{3}}
\title{\vskip-2cm Bellman function for extremal problems in $\BMO$}

\begin{document}
\maketitle
{\small
\begin{center}
\vskip-1cm
St.~Petersburg Department 
of Steklov Mathematical Institute RAS,\\Fontanka 27, St.~Petersburg, 
Russia
\end{center}
\begin{center}
Chebyshev Laboratory (SPbU), 14th Line 29B, Vasilyevsky Island, St.~Petersburg, Russia
\end{center}
\begin{center}
Saint Petersburg State University, Universitetsky prospekt 28,\\Peterhof, St.~Petersburg, Russia.
\end{center}
}
\begin{center}
{\em
ivanishvili.paata@gmail.com\\
nicknick@pdmi.ras.ru\\
dms239@mail.ru\\
vasyunin@pdmi.ras.ru\\
paxa239@yandex.ru 
} 
\end{center}
\begin{abstract}
In this paper we develop the method of finding sharp estimates by using 
a Bellman function. In such a form the method appears in the proofs of
the classical John--Nirenberg inequality and $L^p$ estimations of
$\BMO$ functions. In the present paper we elaborate a method of
solving the boundary value problem for the homogeneous Monge--Amp\`ere
equation in a parabolic strip for sufficiently smooth boundary conditions.
In such a way, we have obtained an algorithm for constructing an exact Bellman
function for a large class of integral functionals on the $\BMO$ space.
\end{abstract}
\newpage
\tableofcontents

\section{History of the problem and description of our results}
\subsection{History and formulation of the problem}
	We consider extremal problems for integral functionals on the $\BMO$ space that is defined on some interval $I\subset\mathbb{R}$.
	First, we introduce some notation. By $I$ and $J$ we always denote intervals on~$\mathbb{R}$. By $\Av{\varphi}{J}$ we denote
	the average of a function~$\varphi$ over an interval~$J$:
  $$
  	\Av{\varphi}{J} \df \frac{1}{|J|}\int\limits_{J}\varphi,
  $$
	where $|J|$ is the length of the interval. We consider the $\BMO$ space endowed with the quadratic 
	norm\footnote{We call the expression~$\BNorm{\varphi}$ a norm, although we must
	factorize by the constant functions in order to obtain a normed space.}:
	$$
		\BMO(I) \df \Big\{\varphi \in L^1(I)\mid\; 
		\BNorm{\varphi}^2 \!\!\df\; \sup_{J\subset I}\Av{|\varphi-\Av{\varphi}{J}|^2}{J} < \infty\Big\}.
	$$	
	Details on $\BMO$ can be found in~\cite{Koosis} or~\cite{Stein}.	
	By $\BMO_\eps(I)$ we denote the ball of radius~$\eps$ in this space:
	$$
		\BMO_{\eps}(I) \df \big\{\varphi \in \BMO(I) \mid\; \BNorm{\varphi} \le \eps\big\}.
	$$

	Now we consider several well-known inequalities for functions in $\BMO(I)$.
	First, there is a double estimate claiming the equivalence of any $p$-norm (${0\!<\!p\!<\!\infty}$) and the initial quadratic norm:
	\begin{equation}\label{e13}
		c_p\BNorm{\varphi} \le \sup_{J\in I}\Av{|\varphi-\Av{\varphi}{J}|^p}{J}^{1/p} \le C_p\BNorm{\varphi}. 
	\end{equation}
	Second, the weak-form John--Nirenberg inequality claims that the measure of the set where 
	some function ${\varphi \in \BMO(I)}$ deviates from its average by more than a certain value $\lambda >0$, decreases exponentially in $\lambda$:
	\begin{equation}\label{e14}
		\frac{1}{|I|}\big|\big\{t\in I\mid\; |\varphi(t)-\Av{\varphi}{I}| \geq \lambda\big\}\big|\le 
		c_1 e^{-c_2\lambda/\BNorm{\varphi}}.
	\end{equation}
	And the third inequality can be obtained from the previous one by integration. It is called the integral 
	John--Nirenberg inequality and may be treated as the reverse Jensen inequality for functions 
	in $\BMO_\eps(I)$ and the exponent. 
	Namely, there exist a number $\eps_0 > 0$ and a positive function $C(\eps)$,~${0<\eps<\eps_0}$, such that 
	\begin{equation}\label{e15}
		\Av{e^{\varphi}}{I} \le C(\eps) e^{\av{\varphi}{I}}
	\end{equation}
	for all $\varphi \in \BMO_{\eps}(I)$.
	
	There exist various proofs of these inequalities. For example, in Koosis' book~\cite{Koosis}, Garnett's martingale proof is presented.
	In Stein's book~\cite{Stein}, the proof, based on the duality of $\BMO$ and $H^1$, can be found. We are interested in 	sharp constants in
	inequalities of this kind. One of the methods that are employed to obtain sharp constants, is called the Bellman 
	function method.	
	The history of this method can be found, e.g. in~\cite{NTV}.
	
	Now, we consider the following Bellman function:	
	\begin{equation}\label{e11}
		\Bell(x_1,x_2;\,f) \;\df\!\! 
		\sup_{\varphi \in \BMO_{\eps}(I)}\big\{\Av{f\circ\varphi}{I} \mid\; \Av{\varphi}{I} = x_1,\Av{\varphi^2}{I} = x_2\big\},
	\end{equation}
	where $f$ is some function on $\mathbb{R}$ (we postpone the discussion of the class that~$f$ may belong to). 
	We often omit~$f$ in the notation and merge two variables into one, i.e. we write
	$\Bell(x_1,x_2)$, or $\Bell(x ;\,f)$, or simply $\Bell(x)$, where $x = (x_1,x_2)$.
	
	There are two points worth noting. First, $\Bell$ does not depend on the interval~$I$
	participating in the definition above. 
	Second, if we replace supremum by infimum in~(\ref{e11}), we will obtain 
	the function $-\Bell(x_1,x_2;\,-f)$. In the beginning of Section~\ref{s21}, all this will be discussed in detail.
	
	If we set $f(u)=|u|^p$, then after obtaining analytical expressions for $\Bell(x;\,f)$ and 
	$-\Bell(x;\,-f)$,
	we will get estimate~(\ref{e13}) with the sharp constants $c_p$ and $C_p$ as a corollary. All this was  done 
	in~\cite{SlVa2}.	
	Setting $f(u) = \chi\cii{(-\infty,-\lambda]\cup[\lambda,\infty)}(u)$, we obtain the Bellman function that gives us
	the sharp constants for the weak John--Nirenberg inequality (see~(\ref{e14})). This function was found in~\cite{Va}.	
	Finally, setting $f(u) = e^u$, we obtain the Bellman function for the integral John--Nirenberg inequality 
	(see~(\ref{e15})). The analytical expression for this function was found in~\cite{Va2} and~\cite{SlVa};
	the sharp constants $\eps_0=1$ and $C(\eps) = e^{-\eps}(1-\eps)^{-1}$
	were obtained as a corollary.
	
	In this paper, we construct the function $\Bell(x_1,x_2;\,f)$ not for a function $f$ fixed, but for 
	some wide class of functions, which is described in the next section.

\subsection{Description of our results}
	We will see later that in the formulas for~$\Bell$ the integrals of the following expressions participate:  
	$$
		f^{(r)}(t)e^{\pm {t}/{\eps}},\quad r=0,1,2,3.
	$$
	Therefore, the following space is required:
	$$
		\W \df C^2(\mathbb{R}) \cap W_3^1(\mathbb{R},w_{\eps_0}),
	$$
	where $\eps_0>0$ and $w_{\eps_0}(t) \df e^{-{|t|}/{\eps_0}}$.
	The space on the right of the intersection sign is a weighted Sobolev space. Functions in this space, together with their first 
	three derivatives, are 
	integrable with the weight~$w_{\eps_0}$. We note that~$\W$ is defined as an intersection of a set of functions and a set of equivalence 
	classes. But this definition becomes reasonable if we read it left to right: 
	if a function~$f$ belongs to~$\W$, then it is twice continuously differentiable and $f\in W_3^1(\mathbb{R},w_{\eps_0})$. 
	
	Also, we will see that the behavior of~$\Bell$ depends strongly on the sign of $f'''$.
	We introduce a subset	$\W^N\subset\W$ of functions
we deal with. Any function of this class has  $2N+1$ points  
$$
-\infty\le c_0<v_1<c_1<v_2<\ldots<v_N<c_N\le+\infty
$$
on the extended real line such that
\begin{enumerate}
	\item[\textup{1)}] $f'''>0$ a.e. on $(v_k,c_k)$ and on $(-\infty,c_{0})$. Also,  
$f'''<0$  a.e. on $(c_k,v_{k+1})$ and on $(c_{N},\infty)$;
	\item[\textup{2)}] $|c_k-v_j| \ge 2\eps_{0}$.
\end{enumerate}
We build the function $\Bell(x;\,f)$ for $f \in \W^N$ and $\eps<\eps_0$.
		
	It is worth mentioning that not all the functions listed in the previous section belong to $\W^{N}$ or even to $\W$. For example, the 
	function $|\cdot|^p$ with $p<2$ 
	and the function $\chi\cii{(-\infty,-\lambda]\cup[\lambda,\infty)}(\cdot)$ are not smooth enough (although, if $p>2$, the function $|\cdot|^p$ 
	satisfies all the conditions required). Moreover, by the first 
	point of our assumptions, $f'''\ne 0$ a.e., so $\W^{N}$ does not contain
	functions quadratic on intervals of positive measure. 
	All the restrictions imposed on $f$ are technical,
	and we will lift most of them in future papers (see Chapter~\ref{s7}). 
		
	Next, consider the parabolic strip (see Figure~\ref{fig:pp}):
	\begin{equation}\label{e12}
		\Omega_\eps \df \big\{(x_1, x_2) \in \mathbb{R}^2\mid\; x_1^2 \le x_2 \le x_1^2+\eps^2\big\}.
	\end{equation}

\begin{figure}[ht]
\begin{center}
\includegraphics{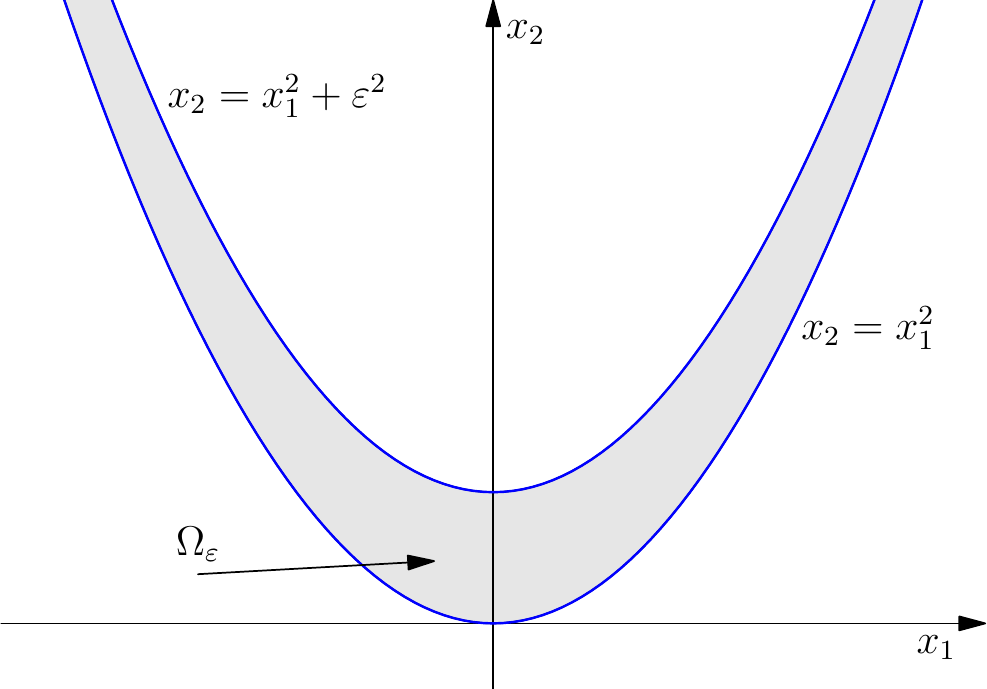}
\caption{The parabolic strip $\Omega_{\varepsilon}$.}
\label{fig:pp}
\end{center}
\end{figure}

	It is easy to prove (and this will be done in the beginning of the next chapter) that $\Omega_\eps$ is the domain of~$\Bell$ (in the sense that
	$\Omega_\eps$ consists of all the points $(x_1,x_2)$ such that the supremum in~(\ref{e11}) is taken over a nonempty set for them), and that
	$\Bell$ satisfies the boundary condition $\Bell(x_1,x_1^2) = f(x_1)$ on the lower parabola. 	
	
	We will also see that $\Bell$ is \emph{locally concave}, i.e. concave on every convex subset of 
	$\Omega_\eps$.
	We give the definition of the local concavity in another form that is more suitable for our purposes.
	\begin{Def}\label{D1}
		A function $G$, defined on some set ${\Omega \subset \mathbb{R}^n}$, is called locally concave in $\Omega$ 
		if the inequality $$G(\alpha_{-}x^- + \alpha_{+}x^+) \ge \alpha_-G(x^-) + \alpha_+G(x^+)$$ is fulfilled 
		for every straight-line 
		segment $[x^-,x^+]\subset\Omega$
		and every pair of numbers ${\alpha_-,\alpha_+ \ge 0}$ such that
		$\alpha_- + \alpha_+ = 1$.
	\end{Def}			 
	By $\Lambda_{\eps,f}$ we denote the class of continuous functions that are locally concave in $\Omega_\eps$ and 
	satisfy the boundary condition mentioned above:
	$$
		\Lambda_{\eps,f} \df \big\{G\in C(\Omega_\eps)\mid\; \mbox{$G$ is locally concave};\;
		G(u,u^2)=f(u)\;\forall\, u\in\mathbb{R}\big\}.
	$$		
	Now we are ready to describe our results. 
	
	{Suppose $f\in \W^{N}$, where $\eps_0>0$ and $N\in\mathbb{Z}_+$.}
	\begin{enumerate}
	\item[\textup{(a)}]
				\textit{For $0<\eps<\eps_0$\textup, the function $\Bell(x;\,f)$ belongs to $\Lambda_{\eps,f}$. Moreover\textup,}
				$$
					\Bell(x;\,f) = \inf\limits_{G\in\Lambda_{\eps,f}}G(x);
				$$
	\item[\textup{(b)}]
				\textit{For each $\eps$\textup{,} $0<\eps<\eps_0$\textup, we construct 
				an expression for $\Bell$ in terms of~$f$.}
		\end{enumerate}

	Statement~(a) means that the problem of finding the Bellman function~$\Bell$ can be reformulated in geometric terms: it is equivalent to the problem 
	of finding the minimal locally concave function in $\Omega_\eps$ that  satisfies a certain boundary condition. 
	We believe that this remains true in a 
	more general setting without 
	any assumptions about the sign of $f'''$. Unfortunately, we do not know how to prove the local concavity 
	of~$\Bell$ directly.
	The fact that~$\Bell$ is locally concave will follow from
	an explicit expression for this function (and the restrictions on $f'''$ are required in order to find 
	this expression). In the next chapter, we will discuss this problem in more detail.
	
	Concerning~(b), by \emph{an expression in terms of~$f$} we mean a rather complicated construction, 
	which consists of various integral and differential transformations of~$f$. 
	Roots of some equations that cannot be solved in elementary functions also participate.
	We are going to find such an expression employing the vast theory, which is continued to be developed in this 
	article. 
	Using this theory, we will solve the homogeneous Monge--Amp\`ere
	equation $B_{x_1x_1}B_{x_2x_2} - B_{x_1x_2}^2 = 0$ (this identity, which means that the Hessian $\frac{d^2\!B}{dx^2}$ 
	is degenerate, allows
	us to consider~$B$ as a Bellman function candidate). After that we will find
	functions on which the supremum in the definition of the Bellman function is attained (after we get such 
	functions, we will be able to prove that the candidate~$B$ coincides with the true Bellman function). 
	The development of these ideas is the main purpose of
	the paper; statements~(a) and~(b) are corollaries of our results.  
	
	\section{General principles}
	Throughout this chapter, we assume that $0<\eps<\eps_0$, $f \in \W$, and $\Bell$ is the Bellman function defined by~(\ref{e11}). 
	In Section~\ref{s21}, we will prove that 
	the domain of~$\Bell$ is~$\Omega_\eps$ and obtain the boundary condition for~$\Bell$ on the lower parabola $x_2= x_1^2$. 
	We will also explain why the 
	assumption of the 
	local concavity of~$\Bell$ is reasonable (but the fact that $\Bell$ is, indeed, locally concave will become 
	clear only after we find an explicit expression for~$\Bell$; for
	this, we will employ the additional restriction $f\in\W^{N}$).
	In Section~\ref{s22}, we will prove that every locally concave function with the same domain and boundary condition as $\Bell$, is pointwise greater 
	than $\Bell$. Thus, we must find a minimal locally concave function on~$\Omega_\eps$ satisfying some boundary condition.
	In Section~\ref{s23}, we will describe a general method of finding such functions that is based on solving a homogeneous Monge--Amp\`ere equation. 
	A solution of such an equation may be considered as a Bellman function candidate. 	
	But to ensure that this candidate is, indeed, the Bellman function required, we must build for each 
	point $x\in\Omega_\eps$ \emph{an optimizer}. An optimizer is a function $\ex\in\BMO_\eps(I)$ such that the supremum from
	definition~(\ref{e11}) of $\Bell$ is attained on it.
	In Section~\ref{s24}, we will present some general considerations on how to build optimizers. 
	
	\subsection{Main properties of  \texorpdfstring{function~$\Bell$}{Bellman function}}\label{s21}
	\paragraph{Preliminaries.}
	First, using linear transformation of one interval into another, we obtain the following fact.
	\begin{Rem}
		The function~$\Bell$ does not depend on the interval~$I$ participating in its definition.
	\end{Rem}	
	Second, if we need the lower estimate, we can replace supremum by infimum in~(\ref{e11}). Instead of this, we can 
	solve the supremum problem for the boundary function~$-f$.
	\begin{Rem}\label{rem1}
		The function~$-\Bell(x_1,x_2;\,-f)$ coincides with 
		\begin{equation*}
			\Bell^{\min}(x_1,x_2;\,f) \; \df\!\! 
			\inf_{\varphi \in \BMO_{\eps}(I)}\big\{\Av{f\circ\varphi}{I} \mid\; \Av{\varphi}{I} = x_1,\Av{\varphi^2}{I} = x_2\big\}.
		\end{equation*}
	\end{Rem}	
	Third, if two extremal problems correspond to boundary functions~$f$ such that their difference is a quadratic polynomial, 
	then this problems are, in fact, equivalent.
	Namely, the definition of the Bellman function and the linearity of averages imply the following fact.
	\begin{Rem}\label{rem2}
   For any real numbers $a$, $b$, $c$, and $d$, we have
   $$
   	\Bell(x_{1},x_{2};\,af(t)+bt^{2}+ct+d) \\
   	= |a|\Bell(x_{1},x_{2};\,(\sign{a})f(t))+bx_{2}+cx_{1}+d.
   $$
  \end{Rem}  
  Thus, if we know the Bellman function for $(\sign{a}) f(t)$, then we can easily construct it for
  $af(t)+bt^{2}+ct+d$. Similarly, we can make a linear change of variable in the boundary condition.
	\begin{Rem}\label{remm2}
	For any real numbers $\alpha$ and $\beta$, we have
	$$
	\boldsymbol{B}_{\eps}(x_{1},x_{2};\,f(\alpha t +\beta))=\boldsymbol{B}_{|\alpha|\eps}(\alpha x_{1}+\beta, \alpha^{2}x_{2}+2\alpha \beta   x_{1}+\beta^{2}; \, f(t)).
	$$
	\end{Rem}
		
	\paragraph{The domain of $\Bell$ and the boundary condition.}
	\begin{St}\label{S1}
		The function $\Bell(x;\,f)$ has the following properties\textup{:}
		\begin{enumerate}
			\item[\textup{(i)}]
				its domain is the parabolic strip~$\Omega_\eps$ defined by the formula~\textup{(\ref{e12}),} 
				i.e.\textup, the set over which the supremum in~\textup{(\ref{e11})} is taken\textup, is non-empty for those
				and only those $x=(x_1,x_2)$ that lie in $\Omega_\eps$				
				\textup{(}however\textup, the function~$\Bell$ can take the value $+\infty$ there\textup{);}
			\item[\textup{(ii)}]
				the boundary condition $\Bell(x_1,x_1^2) = f(x_1)$ is satisfied\textup{.}
		\end{enumerate}
	\end{St}
	\begin{proof}
		Consider statement~(i). It is easy to see that the estimate ${x_1^2 \le x_2}$ is fulfilled due 
		to the Cauchy--Schwarz inequality and the estimate ${x_2 \le x_1^2+\eps^2}$ follows from the requirement
		${\varphi \in \BMO_{\eps}(I)}$. Therefore, $\Omega_{\eps}$ contains the domain of $\Bell$. On the other hand, 
		if the estimate ${x_1^2 \le x_2 \le x_1^2+\eps^2}$ is fulfilled, we can easily construct a function 
		${\varphi \in \BMO_{\eps}(I)}$ whose average equals $x_1$ and square deviation equals $\sqrt{x_2 - x_1^2}$.
		For example, we may take the function
		$$
			\varphi(t)=\begin{cases}
					\textstyle x_1+\sqrt{x_2 - x_1^2}, &t\in I_-;\\
					\textstyle x_1-\sqrt{x_2 - x_1^2}, &t\in I_+,
				\end{cases} 
		$$
		where $I_-$ and $I_+$ are the left and right halves of $I$, respectively. This means that 
		$\Omega_{\eps}$ is contained in the domain of $\Bell$.
	
		Statement~(ii) is trivial. Indeed, the identity $x_2 = x_1^2$ means that all the functions $\varphi$ over which
		the supremum is taken, do not deviate from
		their average $x_1$. Therefore, the set of such functions consists of a single element $\varphi(t)\equiv x_1$. This
		implies the condition required.
	\end{proof}
	
	\paragraph{Local concavity.}
	Now we discuss the concavity of $\Bell$. Let $x^\pm \in \Omega_\eps$, and let ${\alpha_{\pm}}$ be numbers such that  
	$\alpha_{\pm}\ge 0$, $\alpha_- + \alpha_+ = 1$,
	and the point $\alpha_-x^- + \alpha_+x^+$ gets into $\Omega_\eps$. 
	We split~$I$ into two subintervals~$I_-$ and~$I_+$ such that $|I_{\pm}| = \alpha_{\pm}|I|$.
	Further, we choose two functions $\varphi\Cii{\pm} \in \BMO_{\eps}(I_{\pm})$ such that 
	${\big(\Av{\varphi\Cii{\pm}}{I_{\pm}}, \Av{\varphi\Cii{\pm}^2}{I_{\pm}}\big) = x^{\pm}}$ and these functions
	almost realize the supremum on the corresponding intervals. The latter means that 
	$$
		\langle{f(\varphi\Cii\pm)}\rangle\Cii{I_{\pm}} \ge \Bell(x^{\pm}) - \eta,
	$$
	where $\eta > 0$ is a small value. 
	Consider the function
	$$
		\varphi(t)=
			\begin{cases}
				\varphi\Cii{-}(t),&t\in I_-;\\
				\varphi\Cii{+}(t),&t\in I_+.
			\end{cases}
	$$
	First, since the point 
	${\big(\Av{\varphi}{I}, \Av{\varphi^2}{I}\big)} = {\alpha_-x^- + \alpha_+x^+}$ gets into $\Omega_\eps$,
	we have $\Av{\varphi^2}{I}-\Av{\varphi}{I}^2 \le \eps^2$.
	Second, it is clear that $\varphi\in \BMO_\eps(I_-)\cap\BMO_\eps(I_+)$. 
	However, these conditions are not sufficient for the function $\varphi$ to get into $\BMO_\eps(I)$
	(it is worth mentioning that this problem does not arise in the case of the dyadic $\BMO$ space; see~\cite{SlVa}).	
	If we could for every $\eta$ choose functions $\varphi\Cii{\pm}$
	such that their concatenation~$\varphi$ gets into $\BMO_\eps(I)$, then the following inequalities
	would be fulfilled:
	\begin{multline*}
		\Bell(\alpha_-x^- + \alpha_+x^+) 
		\ge \Av{f(\varphi)}{I} = 
		\alpha_-\Av{f(\varphi\Cii{-})}{I_{-}}+\alpha_+\Av{f(\varphi\Cii{+})}{I_{+}} \\\ge
		\alpha_-\Bell(x^-)+\alpha_+\Bell(x^+)-\eta.
	\end{multline*}
	Letting $\eta \rightarrow 0$, we would get the concavity of $\Bell$. 
	But in the continuous case the method described above does not work, because $\varphi$ may lay outside $\BMO_\eps(I)$. 
	It turns out that the function~$\Bell$ is only locally concave. But this will be clear only after we
	construct an explicit expression for~$\Bell$. Nevertheless, the heuristic method that we will use to build a Bellman 
	candidate,
	is based on the fact that the local concavity condition is satisfied:
	 \begin{enumerate}
		\item[\textup{(iii)}]
			\emph{the function $\Bell$ is locally concave in the parabolic strip $\Omega_\eps$.}
	\end{enumerate}
	
\subsection{Locally concave majorants}\label{s22}	
	In this section, we prove that every function in $C(\Omega_\eps)$ with properties (ii) and (iii) (we recall
	that the set of such functions is denoted 
	by~$\Lambda_{\eps,f}$) majorizes~$\Bell$. Namely, we verify the following statement.
	\begin{St}\label{S2}
		Suppose $0<\eps<\eps_0$\textup, $f\in\W$\textup, and $G\in\Lambda_{\eps,f}$. 
		Then $\Bell(x;\,f) \le G(x)$ for all $x \in \Omega_\eps$.
	\end{St}
	In order to prove this statement, we need some preparation.

	\paragraph{Auxiliary lemmas.}
	First, we need the following geometric lemma, which was proved both in~\cite{Va2} and~\cite{SlVa}. 
	\begin{Le}\label{L1}
		Suppose $\eps_1 > \eps$. Then for any interval $I\subset\mathbb{R}$ and any function 
		$\varphi \in \BMO_\eps(I)$ there exists a partition $I={I_{-}\cup I_{+}}$ such that the line segment 
		with the endpoints $x^{\pm} = \big(\Av{\varphi}{I_{\pm}},\Av{\varphi^2}{I_{\pm}}\big)$ lies in
		$\Omega_{\eps_1}$ entirely. Moreover\textup, the parameters $\alpha_{\pm} = |I_\pm|/|I|$ can be chosen to be 
		separated from $0$ and $1$ uniformly in $I$ and $\varphi$.
	\end{Le}
	\hskip-1.8pt We also need the following statement about truncations of functions in $\BMO(I)$.
	\begin{Le}\label{L2}
		Let $\varphi \in \BMO(I)$\textup, $c,d \in \mathbb{R}$\textup, and $c<d$.
		Let $\varphi_{c,d}$ be the truncation of $\varphi$\textup{:}
		$$
			\varphi_{c,d}(t) \df 
			\begin{cases}
				d, &\varphi(t) > d;\\
				\varphi(t), &c \le \varphi(t) \le d;\\
				c, &\varphi(t) < c.
			\end{cases}  
		$$
		Then $\Av{\varphi_{c,d}^2}{J} - \Av{\varphi_{c,d}}{J}^2 \le \Av{\varphi^2}{J} - \Av{\varphi}{J}^2$ for every
		interval $J \subset I$.
	\end{Le}
	A proof can be found in~\cite{SlVa2}, it is also contained implicitly in~\cite{SlVa}. This lemma immediately implies the following fact.
	\begin{Cor}
		If ${\varphi \in \BMO_\eps(I)}$\textup{,} then $\varphi_{c,d} \in \BMO_\eps(I)$.
	\end{Cor} 
		
	Now we discuss how a function $f \in \W$ and its first two derivatives behave at infinity.
	\begin{Le}\label{L3}
		If $f \in \W$\textup, then the following limit relations are fulfilled\textup{:}
		\begin{equation}\label{e25}
			f^{(r)}(u)e^{-|u|/\eps_0} \to 0\quad\mbox{as}\quad u\to\pm\infty\quad\mbox{for}\quad r = 0,1,2.
		\end{equation}
	\end{Le}	
	\begin{proof}
		\begin{align*}
			f''(u)e^{-{|u|}/{\eps_0}} - f''(0) &= 
			\int\limits_{0}^{u}\left(f''(t)e^{-|t|/\varepsilon_{0}}\right)'\,dt\\ 
			&=
			\int\limits_{0}^{u}f'''(t)e^{-|t|/\eps_0}\,dt - \eps_0^{-1}\sign{u}\int\limits_{0}^{u}f''(t)e^{-|t|/\eps_0}\,dt.
		\end{align*}
		Since $f \in \W$, we have the existence of the limits
	  $$
			\lim_{u \to \pm\infty}f''(u)e^{-|u|/\eps_0}.
		$$
		But if such limits exist, they must be equal to zero (because  $f''(u)e^{-|u|/\eps_0}$ is integrable). 
		Similar reasoning for $f'$ and $f$ gives (\ref{e25}).		
	\end{proof}
	
	We are ready to prove Statement~\ref{S2}. It is worth noting that statements of this kind are a commonplace of 
	the theory and they can be found in almost 
	every article on the Bellman function method in analysis (a classical example is the paper~\cite{NaTr}). 
	
\paragraph{Proof of Statement~\ref{S2}.}
	Let $0 < \tau < 1$. Consider the function 
	$$
			G_\tau(x_1,x_2) \df G(\tau x_1,\tau^2x_2).
	$$
	We also define $f_\tau(x_1) \df f(\tau x_1)$. It is easily seen that~$G_\tau$ is continuous and locally concave in 
	$\Omega_{\eps/\tau}$. 
	This function also satisfies the boundary condition
	$$
			G_\tau(x_1,x_1^2) = f_\tau(x_1).
	$$
	
	Next, consider a point $x\in\Omega_\eps$. Fix a function $\varphi \in \BMO_\eps(I)$ such that
	$x = \big(\Av{\varphi}{I},\Av{\varphi^2}{I}\big)$. 
	By $x^\sigma$ we denote the Bellman point generated by the same function~$\varphi$ and 
	a subinterval $\sigma \subset I$, i.e.
	$x^\sigma \df \big(\Av{\varphi}{\sigma}, \Av{\varphi^2}{\sigma}\big)$. 
	By Lemma~\ref{L1}, there exists a partition $I = I_- \cup I_+$ such that the segment with the endpoints
	$x^{I_-}$ and $x^{I_+}$ lies in $\Omega_{\eps/\tau}$ entirely. 
	Note that $x = x^I = \alpha_-x^{I_-} + \alpha_+x^{I_+}$, where $\alpha_{\pm} = |I_\pm|/|I|$. Using the local 
	concavity of $G_{\tau}$, we get the inequality
	\begin{equation}\label{e21}
			|I|\,G_{\tau}(x) \ge |I_-|\,G_{\tau}\big(x^{I_-}\big) + |I_+|\,G_{\tau}\big(x^{I_+}\big).
	\end{equation}
	We repeat the procedure described above for each subinterval $I_{\pm}$ (treating $\varphi$ as a function on the corresponding subinterval), after 
	that we repeat it again for each of
	four subintervals obtained in the previous step, and so on. 	
	After $n$ steps we have a collection~$D_n$ of $2^n$ subintervals that divide $I$. Using the
	local concavity of $G_{\tau}$ in each step, we get the estimate
	$$
			|I|\,G_{\tau}(x) \ge \sum_{\sigma \in D_n} |\sigma|\,G_{\tau}(x^\sigma) 
			= \int\limits_I G_{\tau}\big(x^n(t)\big)\,dt,
	$$
	where $x^n(t)$ is the step function taking the value $x^\sigma$ on each interval ${\sigma \in D_n}$\footnote{The 
	procedure just described is often 
	called \emph{the Bellman induction}.}. 		
	Since $\alpha_{\pm}$ can be chosen to be separated from $0$ and $1$ uniformly,
	the lengths of the intervals tend to zero as $n$ tends to infinity:
	${\max_{\,\sigma\in D_n} |\sigma| \to 0}$ as ${n \to \infty}$.
	By the Lebesgue differentiation theorem, this implies that 
	$$
		x^n(t)\to\big(\varphi(t),\varphi^2(t)\big)
	$$ 
	for almost all $t\in I$. 
	Suppose for a while that $\varphi \in L^\infty(I)$. 
	Then the values of the functions $x^n(t)$ lie in some compact subset of
	$\Omega_\eps$. Therefore, since $G_\tau(x)$ is continuous, the sequence of 
	functions $G_{\tau}\big(x^n(t)\big)$ is uniformly bounded. 
	Passing to the limit
	and using the boundary condition, we get
	$$
			|I|\,G_{\tau}(x) \ge \int\limits_I G_{\tau}\big(\varphi(t),\varphi^2(t)\big)\,dt = 
			\int\limits_I f_\tau\big(\varphi(t)\big)\,dt.
	$$
	Now we lift the boundedness of $\varphi$ and pass to the limit in $\tau$. Consider the truncations $\varphi_m(t) \df \varphi_{-m,m}(t)$ (the definition see in Lemma~\ref{L5}).
	By Lemma~\ref{L2}, they lie in the same ball $\BMO_\eps(I)$ as~$\varphi$. Thus, since the functions $\varphi_{m}$ are bounded, the estimate proved earlier is true for them:
	$$
			|I|\,G_{\tau}\big(\Av{\varphi_m}{I},\Av{\varphi_m^2}{I}\big) \ge 
			\int\limits_I f_\tau\big(\varphi_m(t)\big)\,dt.
	$$
	Since~$G$ is continuous, the left part tends to $G(x)$ as ${m\to\infty}$ and $\tau \to 1-$. 
	Thus, it remains to pass to the corresponding limits in the right part of the inequality.
	The continuity of $f$ implies that the integrands converge to $f\big(\varphi(t)\big)$ pointwise. 
	Therefore, in order to establish the convergence of 
	the integrals, it remains to find an integrable majorant. Due to relation~(\ref{e25}) for
	$r=0$ and the continuity of $f$, the estimate 
	$|f(s)| \leq C e^{{|s|}/{\eps_0}}$ is fulfilled.
	Then we have
	$$
			\big|f_\tau\big(\varphi_m(t)\big)\big| \leq C\exp{\tfrac{\tau|\varphi_m(t)|}{\eps_0}} 
			\leq C\exp{\tfrac{|\varphi(t)|}{\eps_0}} \leq C \left(\exp{\tfrac{\varphi(t)}{\eps_0}} + \exp{\tfrac{-\varphi(t)}{\eps_0}} \right).
	$$
	The last expression is integrable by the integral John--Nirenberg inequality (see \cite{Va2} or \cite{SlVa}), 
	because $\eps < \eps_0$, and both  
	$\varphi$ and $-\varphi$ are in $\BMO_{\varepsilon}(I)$. Passing to the limits, we finally get ${G(x) \ge \Bell(x)}$.
	
	\subsection{Monge--Amp\`ere equation}\label{s23}
	Let $B$ be the minimal function in $\Lambda_{\eps,f}$.
	Properties (i) and (ii), together with property (iii) being assumed and Statement~\ref{S2}, imply 
	that we may treat $B$ as a candidate for 
	the Bellman function~$\Bell$.	
	In this section, we present some reasoning (not intended to be rigorous)
	that allows us to reduce the problem of finding such a function to solving a certain 
	partial differential equation (homogeneous Monge--Amp\`ere equation). 

	As we will see later, for each point $x \in \Omega_\eps$, there exists a function ${\ex \in \BMO_{\eps}(I)}$ that 
	realizes 
	the supremum for the point~$x$ in the Bellman function definition (see (\ref{e11})), i.e. 
	$x = \big(\Av{\ex}{I},\Av{\ex^2}{I}\big)$ and $\Bell(x) = \Av{f(\ex)}{I}$. 
	If the functions~$G_{\tau}$ from Statement~\ref{S2} approximate $\Bell$, then for the optimizer~$\ex$
	there exists a partition of $I$ such that in~(\ref{e21}) the equality is almost attained. 
	In view of the local concavity of $G_{\tau}$, this means that this function is almost linear on the segment 
	$\big[x^{I_-},x^{I_+}\big] \subset \Omega_{\eps/\tau}$. 
	This yields that our candidate~$B$ must be linear along some vector~$\Theta_x$, i.e. its second derivative along~$\Theta_x$ vanishes 
	at~$x$:
	\begin{equation}\label{e22}
		\frac{\partial^2\!B}{\partial\Theta_x^2} = \bigg(\frac{d^2\!B}{dx^2}\,\Theta_x, \Theta_x\bigg) = 0,
	\end{equation}
	where
	$$
		\frac{d^2\!B}{dx^2} = 
		\begin{pmatrix}
			B_{x_1x_1} & B_{x_1x_2}\\
			B_{x_2x_1} & B_{x_2x_2}
		\end{pmatrix}
	$$
	(here all the functions are evaluated at $x$). 
	On the other hand, since the function~$B$ is locally concave, it follows that the matrix of its second derivatives is negative semidefinite:
	$$
		\frac{d^2\!B}{dx^2} \le 0.
	$$
	Next, by virtue of (\ref{e22}), it cannot be \emph{strictly} negative definite, so
	\begin{equation}\label{e23}
		\det{\bigg(\frac{d^2\!B}{dx^2}\bigg)} = B_{x_1x_1}B_{x_2x_2} - B_{x_1x_2}^2 = 0.
	\end{equation}
	This is the homogeneous Monge--Amp\`ere equation for $B$. Besides (\ref{e23}), the boundary condition
	${B(x_1,x_1^2) = f(x_1)}$ and the inequalities 
	${B_{x_1x_1} \le 0}$, ${B_{x_2x_2} \le 0}$ must be fulfilled.

	In order to solve equation~(\ref{e23}), we will use the following consideration: the integral curves of the vector field~$\Theta_x$ are straight
	lines and, what is more, all the partial derivatives of $B$ are constant along them.
	We formulate this principle in the following statement, which has been proved, for example, in~\cite{VaVo}. 
	\begin{St}\label{S3}
		Suppose $\Omega$ is a domain in $\mathbb{R}^2$\textup, and $G \in C^2(\Omega)$ is 
		a function satisfying the homogeneous Monge--Amp\`ere equation 
		on $\Omega$\textup{:}
		$$
			G_{x_1x_1}G_{x_2x_2} - G_{x_1x_2}^2 = 0.
		$$
		Let 
		$$
			t_1 = G_{x_1},\quad t_2 = G_{x_2},\quad \mbox{and}\quad t_0 = G - t_1x_1 - t_2x_2.
		$$
		Suppose $G_{x_1x_1}\ne 0$ or $G_{x_2x_2} \ne0$ at every point of~$\Omega$.
		Then the functions $t_1$\textup{,} $t_2$\textup, and $t_0$ are constant along the integral curves of the 
		vector field that annihilates 
		the quadratic form $\frac{d^2\!G}{dx^2}$ on $\Omega$. 
		The integral curves mentioned above \textup(the extremals\textup) are segments of the straight lines defined by the equation
		\begin{equation}\label{e24}
			x_1dt_1+x_2dt_2+dt_0 = 0.
		\end{equation}
	\end{St}
	Graphs of solutions of the homogeneous Monge--Amp\`ere equation are called \emph{developable} surfaces. 
 All the properties of such solutions can be formulated in geometric terms. For example, 
	the theorem presented above states that 
	a developable surface is \emph{ruled}. Concerning geometric interpretation, see, e.g.~\cite{Pog}. 

	In view of Statement~\ref{S3}, we can assume that our domain $\Omega_\eps$ can be split into subdomains of two kinds:
	domains where $\frac{d^2\!G}{dx^2}=0$  ($B$ is a linear function there) and 
  domains where $\dim\Ker\frac{d^2\!G}{dx^2}=1$. Latter domains are foliated 
  by straight-line segments such that the partial derivatives of~$B$ are 
  constant along them. 
	We will look for our Bellman function among the functions~$B$ corresponding to such foliations.
	The following definition fixes the notion of \emph{a Bellman candidate}.
	\begin{Def}\label{D2}
		Consider a subdomain $\widetilde \Omega \subset \Omega_\eps$ and a finite collection of pairwise disjoint 
		subdomains\footnote{It is worth noting that here the notion of a domain has a wider meaning than usually: 
		a domain is the union of a connected open set and any part of its boundary.} 
		$\widetilde\Omega^1,\dots,\widetilde\Omega^m \subset \widetilde \Omega$ whose union 
		is~$\widetilde \Omega$. Consider some function $B \in C(\widetilde \Omega)$ that is 
		locally concave in $\widetilde \Omega$ and 
		satisfies the boundary condition $B(x_1,x_1^2) = f(x_1)$. 
		Suppose $B\in C^1(\widetilde\Omega^i)$, $i=1\dots m$, and those subdomains~$\widetilde\Omega^i$
		where~$B$ is not linear, are foliated by non-intersecting straight-line segments such that the partial derivatives 
		of~$B$ are constant along them. Then we say that $B$ is a Bellman candidate in $\widetilde \Omega$.	
	\end{Def}
	From the above, it does not follow that such a function~$B$ solves the Monge--Amp\`ere equation. 
	However, all the Bellman candidates constructed below are $C^2$-smooth in each of the corresponding domains 
	$\widetilde\Omega^{1},\dots,\widetilde\Omega^{m}$. Thus, since~$B$ is linear along the extremals, 
	Monge--Amp\`ere equation~(\ref{e23}) is fulfilled in each domain~$\widetilde\Omega^{i}$ for such a candidate.
	
	Another useful observation, helping us to construct Bellman candidates, is that the extremals, intersecting the upper boundary of~$\Omega_\eps$,
	must be tangents to it (see \emph{Principle 2} on the page 8 of~\cite{SlVa2}).
	
	All of the above allows us to believe that our Bellman function can be found among the functions described 
	in Definition~\ref{D2}.
	If we find some Bellman candidate~$B$ on the whole domain~$\Omega_\eps$, the inequality $\Bell \le B$ will follow 
	immediately from 
	Statement~\ref{S2}. 
	In order to verify the converse estimate $\Bell \ge B$, we will construct, for each point $x \in \Omega_\eps$, a 
	function
	$\ex \in \BMO_\eps(I)$ such that $x = \big(\Av{\ex}{I},\Av{\ex^2}{I}\big)$ and $B(x) = \Av{f(\ex)}{I}$.
	Such functions are called \emph{optimizers}. General considerations on the construction of optimizers 
	are stated in the next section.
	
	\subsection{Optimizers}\label{s24}
	First, we fix the notion an optimizer.
	\begin{Def}\label{D3}
		Let $B$ be a Bellman candidate in the whole domain~$\Omega_\eps$. 
		A function~$\ex$ defined on some interval~$I$ is called an optimizer for a point 
		$x\in \Omega_\eps$ if the following conditions are satisfied:
		\begin{enumerate}
			\item[(1)] $\ex \in \BMO_\eps(I)$;
			\item[(2)] $\big(\Av{\ex}{I},\Av{\ex^2}{I}\big) = x$;
			\item[(3)] $\Av{f(\ex)}{I} = B(x)$.
		\end{enumerate}
	\end{Def}
	The first two properties mean that $\ex$ is one of the functions over which the supremum is taken in 
	definition~(\ref{e11}) of $\Bell$
	(we will call such functions \emph{test functions}). In view of Statement~\ref{S2}, the third property guarantees 
	that the test function~$\ex$ realizes this supremum. 
	Therefore, a Bellman candidate for which an optimizer can be constructed in each point $x\in\Omega_\eps$, coincides
	with $\Bell$.
	
	We notice that it suffices to consider
	only non-decreasing optimizers. Indeed, if we replace a function by its \emph{increasing rearrangement}, 
	the $\BMO$-norm does not increase (an increasing rearrangement of a function~$\varphi$ is 
	a non-decreasing function~$\varphi^{*}$ such that the measure of the set 
	$\{t \in I \mid\;\varphi(t) > \lambda\}$ is equal to the measure of the set
	$\{t \in I \mid\; \varphi^{*}(t) > \lambda\}$ for any $\lambda \in \mathbb{R}$).
	This statement was proved in~\cite{Ivo}. In~\cite{Kor}, it was employed for the calculation of the
	sharp constant~$c_2$ in John--Nirenberg inequality~(\ref{e14}). It is also clear that averages of the form
	$\av{h(\varphi)}{I}$ does not change when $\varphi$ is replaced by its increasing rearrangement.
	All this implies that the supremum in~(\ref{e11}) may be taken over the set of the non-decreasing functions
	satisfying the same conditions.
	
	We will construct optimizers using the notion of \emph{delivery curves}. The following reasoning, which is not 
	intended to be rigorous, will lead us to the corresponding definition. Consider a non-decreasing 
	optimizer~$\ex$. For it, each inequality in the Bellman induction (see the proof of Statement~\ref{S2}) turns into an  equality. Thus, we must split the interval
	in such a way that the corresponding points move along the extremals that foliate
	the subdomains where the Bellman function 
	is not linear (inside domains of the linearity, every segment is an extremal).
	If at each step of the Bellman induction we manage to choose an infinitesimal partition, i.e.
	cut off an arbitrarily small part from one side of the interval, then we get some curve inside the domain 
	(the coordinates of its points are, in fact, the averages of~$\ex$ and~$\ex^2$ over
	the larger of two intervals that are obtained after each cutting).
	If we cut off from the right side of the interval, then 
	a resulting curve is called a \emph{left} delivery curve (since we consider an increasing test function, 
	this curve lays on the \emph{left} of the point at which we begin the induction). 
	This heuristic reasoning leads us to the 
	following rigorous definition.
	\begin{Def}\label{D4}
		Suppose $\ex$ is some test function on ${I = [l,r]}$.
		A curve~$\dc$ is called a left delivery curve if it is defined by the formula
		\begin{equation}\label{eq7}
			\dc(s) =\big(\av{\ex}{[l,s]},\av{\ex^2}{[l,s]}\big),\quad s \in (l,r],
		\end{equation}
		and for all $s \in (l,r]$ the following equation is fulfilled:
		\begin{equation}\label{B_extr_curve}
			B(\dc(s)) = \av{f(\ex)}{[l,s]}.
		\end{equation}		
	\end{Def}

	Cutting off from the left side of the interval, we come to the notion of 
	a \emph{right} delivery curve (it lies on the \emph{right} of the initial point). 
	The corresponding definition is symmetric to the definition of a left delivery curve.
	\begin{Def}\label{D5}
		Suppose $\ex$ is some test function on $I = [l,r]$.
		A curve~$\dc$
		is called a right delivery curve if it is defined by the formula
		\begin{equation}\label{eq11}
			\dc(s) =  \big(\av{\ex}{[s,r]},\av{\ex^2}{[s,r]}\big),\quad s \in [l,r),
		\end{equation}
		and for all $s \in [l,r)$ the following equation is fulfilled:
		\begin{equation}\label{B_extr_curve_r}
			B(\dc(s)) = \av{f(\ex)}{[s,r]}.
		\end{equation}		
	\end{Def}
	Definitions~\ref{D4} and~\ref{D5} postulate that the restrictions ${\ex|}_{[l,s]}$ are optimizers 
	for the corresponding points~$\dc(s)$ of the left delivery curve (which lies, of course, in 
	$\Omega_\eps$ entirely), and the restrictions ${\ex|}_{[s,r]}$ are optimizers for the points $\dc(s)$ of the right 
	delivery curve. Therefore, if we build a delivery curve, we 
	automatically obtain the optimizers for all the points of this curve.
	
	According to the procedure described above, delivery curves run along extremals. 
	Thus, they can consist of some parts
	of extremals and arcs of the upper parabola. 
	Also, if we take only non-decreasing test functions, then left delivery curves will run from left to right and 
	right delivery curves will run from right to 
	left (for right delivery curves, we assume that the ``time''~$s$ runs backwards, i.e. from $r$ to $l$).
	
	We will build optimizers for some Bellman candidate~$B$ as follows. 
	We will draw various curves along the extremals corresponding to our 
	candidate. After that, we will construct functions $\ex\in L^1(I)$ that generate these curves 
	in the sense of~(\ref{eq7}) or~(\ref{eq11}). 
	Next, we will verify that the obtained functions belong to $\BMO_\eps(I)$ 
	and satisfy~(\ref{B_extr_curve}) or~(\ref{B_extr_curve_r}). 
	The condition $\ex \in \BMO_\eps(I)$
	can be derived from general geometric considerations. The fact is that
	all our delivery curves turn out to be convex, because they are graphs of some convex function. In addition, their curvatures will not be too large:
	as a rule, any tangent to such a curve will lie under the upper boundary of~$\Omega_\eps$.
  These properties can be explained by the fact that these curves must run along the upper parabola or straight
  extremals, which intersect the upper boundary tangentially.
	It turns out that if some function $\ex\in L^1(I)$ generates a curve with the properties described above, then 
	$\ex \in \BMO_\eps(I)$. We formulate the 
	corresponding statement in the local form, which is more convenient for further applications.
	
	\begin{Le}\label{L4}
	Let $\ex$ be an integrable function on $I = [l,r]$ and let $\dc$ be the curve generated by this function in 
	the sense of~\textup{(\ref{eq7})}.
	Suppose $\dc$ lies in~$\Omega_\eps$ entirely\textup, coincides with the graph of a convex function\textup, and
	is differentiable in some point $b \in I$. If the tangent to $\gamma$ at the point~$\gamma(b)$ lies below 
	the upper boundary of~$\Omega_\eps$\textup, then all the Bellman points 
	$x^{[a,b]} = \big(\av{\ex}{[a,b]},\av{\ex^2}{[a,b]}\big)$\textup{,} $l \le a < b$\textup, belong to $\Omega_\eps$.
		
	If the curve~$\dc$ is generated by~$\ex$ in the sense of~\textup{(\ref{eq11}),} then 
	the Bellman point~$x^{[a,b]}$ is in $\Omega_\eps$ provided 
	the tangent to $\dc$ at the point $\dc(a)$ lies below the upper parabola.
\end{Le}
\begin{proof}
	We prove only the first half of the lemma, the proof of the second is similar. 
	Since the curve~$\dc$ is convex, the point~$\dc(a)$ must lie above the tangent to $\dc$ at
	the point~$\dc(b)$. 	
	The points $\dc(a)$, $x^{[a,b]}$, and $\dc(b)$ lie on one line and the last lies between the first two, because 
	it is their convex combination. Thus, the point~$x^{[a,b]}$ must lie below the tangent, and therefore, 
	below the upper boundary of~$\Omega_\eps$. On the other hand, by 
	the Cauchy--Schwartz inequality, the point~$x^{[a,b]}$ lies above the lower boundary.
\begin{figure}[H]
\begin{center}
\includegraphics{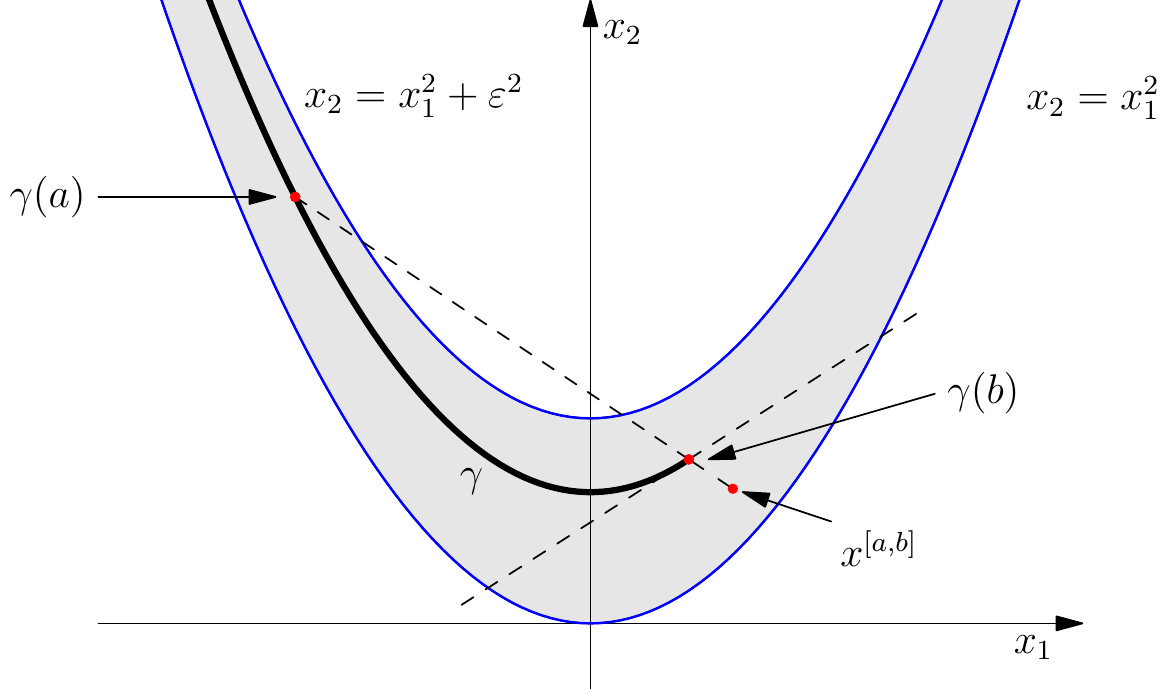}
\caption{Illustration to the proof of Lemma~\ref{L4}.}
\label{fig:pfp}
\end{center}
\end{figure}
	
	As we have already mentioned, the symmetric situation when $\dc$ and $\ex$ satisfy relation~(\ref{eq11}), 
	can be treated in a similar way.
	\end{proof}
	
\section{Homogeneous families of extremals}\label{s3}
	As already noted, an extremal intersecting the upper parabola must touch it.
	In this chapter, we assume that some subdomain of $\Omega_\eps$ is foliated by extremals that are tangential to the
	upper boundary, and look for a Bellman candidate in such a subdomain. In Section~\ref{s31}, we will see how such
	extremals must be arranged and how to calculate a Bellman candidate~$B$ corresponding to them (up to some
	constant of integration). In Section~\ref{s32}, we will consider the case when subdomains foliated by tangents are
	not bounded from one side. In such a situation, we will be able to specify the formula for our candidate~$B$, i.e.
	to get rid of the integration constant mentioned above. It is worth noting that all the arguments 
	in Sections~\ref{s31} and~\ref{s32} are, in fact, a repetition of the corresponding arguments in~\cite{SlVa2}.
	We state them here for completeness. In Section~\ref{s33}, using the approach described in Section~\ref{s24}, 
	we will find delivery curves and optimizers in the domains being considered. It will occur that the theory
	described in Sections~\ref{s31}, \ref{s32}, and \ref{s33} is sufficient to obtain $\Bell(x;\,f)$ 
	for $f\in\W^{0}$ with $c_{0}=\pm \infty$, i.e. when the sign of~$f'''$ does not change (the case $f'''<0$ corresponds to $c_0=-\infty$,
the case $f'''>0$ corresponds to $c_0=+\infty$). The corresponding theorems are stated in Section~\ref{s34}.

\subsection{Family of tangents to the upper boundary}\label{s31}
	Consider the tangent to the upper parabola at a point $(w, w^2+\eps^2)$. Its segment lying in $\Omega_\eps$
	is given by the following relation:
	\begin{equation}\label{e31}
		x_2 = 2w x_1 + \eps^2 - w^2,\quad \mbox{for}\quad x_1\in [w-\eps, w+\eps].
	\end{equation}	
	Consider some hypothetical family of extremals (they are segments of straight lines) such that 
	each of them is a tangent to the upper parabola. Parameterize this family 
	by the first coordinate of tangency points $w \in (w_1,w_2)$. 
	If the corresponding Bellman candidate~$B$ is not linear in both variables, then an extremal cannot 
	contain the whole segment~(\ref{e31}). Moreover, a tangency point $(w, w^2+\eps^2)$ is not an inner point of  
	an extremal, otherwise such an extremal 
	intersects with others. This can also be seen from convexity provided the function~$B$ is twice differentiable.
	Indeed, since the function $t_2 = B_{x_2}$ is constant along the extremals, it may be treated 
	as a function of $w$.
	Thus, $B_{x_2x_2} = t_2'(w)w_{x_2}$. Further, using equation~(\ref{e31}), we get
	$$
		w_{x_2} = \frac{1}{2(x_1-w)}.
	$$
	Fixing $w$, we see that on the corresponding extremal the sign of $B_{x_2x_2}$ changes in a neighborhood 
	of the point $x_1 = w$. But this contradicts the 
	condition $B_{x_2x_2} \le 0$. 

	Thus, each extremal line of our family lies either on the right of the point \mbox{$(w, w^2+\eps^2)$} or on the left of it. 
	Consider two families of extremals. The first 
	consists of segments of tangents to the upper parabola that lie on the right of their tangency points. 
	The second consists of those segments that lie on the left
	of the tangency points. 
	We make the substitution $w = u - \eps$ in the first case and $w = u + \eps$ in the second, i.e. we parameterize 
	the extremals by the first coordinate~$u$ of those points where they intersect the lower parabola. 
	The parameter~$u$ runs over some interval $(u_1, u_2) = (w_1\pm \eps,w_2 \pm \eps)$. 
	Therefore, our families of the right and left tangents are described, respectively, by the following equations:	
	\begin{enumerate}
		\item[(R)] $x_2 - 2(u-\eps)x_1 + u^2-2u\eps = 0$,\qquad $u\in(u_1,u_2)$,\quad $x_1 \in [u-\eps, u]$;
		\item[(L)] $x_2 - 2(u+\eps)x_1 + u^2+2u\eps = 0$,\qquad $u\in(u_1,u_2)$,\quad $x_1 \in [u, u+\eps]$.
	\end{enumerate}	
	We look for a Bellman candidate on subdomains of~$\Omega_\eps$ that are foliated by 
	families~(R) or~(L). We define such subdomains by $\Rt(u_1,u_2)$ and
	$\Lt(u_1,u_2)$, respectively (see Figures~\ref{fig:opk} and~\ref{fig:olk}). 
	
\begin{figure}[H]
\begin{center}
\includegraphics{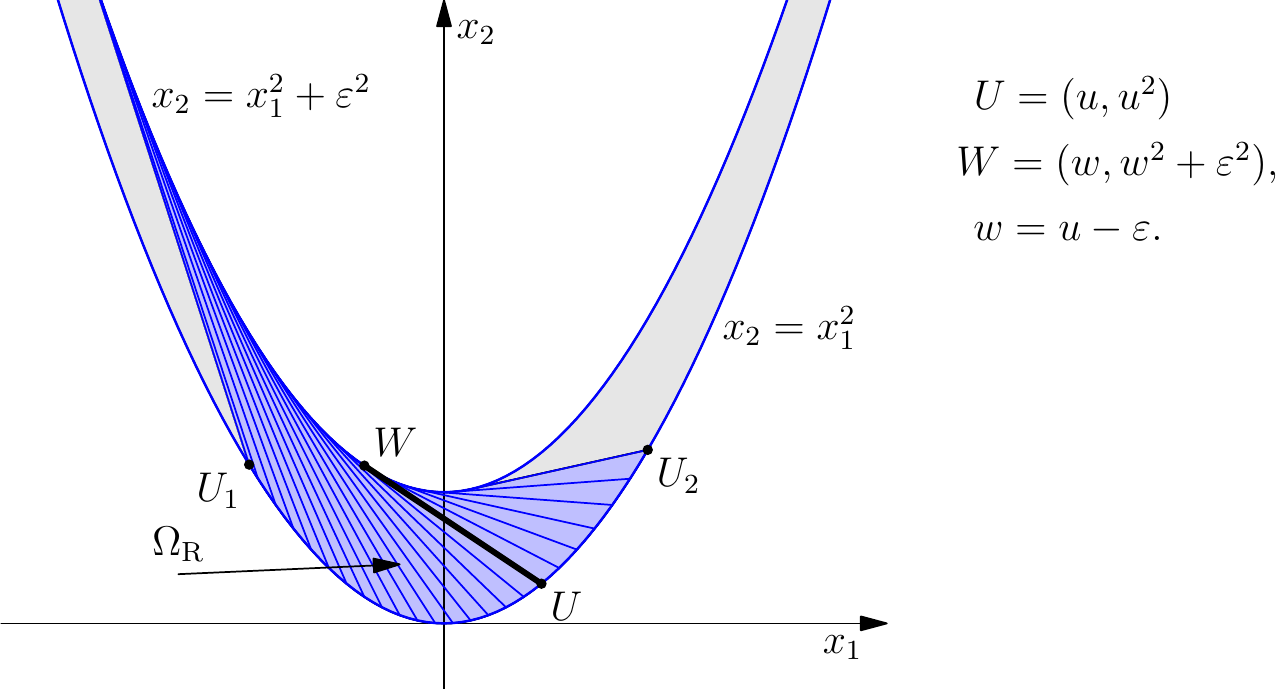}
\caption{A domain $\Rt$ with the right tangents.}
\label{fig:opk}
\end{center}
\end{figure}

	Expressing $u$ in terms of $x_1$ and $x_2$ for the tangents~(R) and~(L), we obtain, respectively, the following relations:	
	\begin{gather}
		\label{e32}
		u = u_{\mathrm{R}}(x_1,x_2) = x_1 +\Big(\eps - \sqrt{\eps^2-(x_{2}-x_{1}^{2})}\,\Big),\\
		\label{e33}
		u = u_{\mathrm{L}}(x_1,x_2) = x_1 - \Big(\eps - \sqrt{\eps^2-(x_{2}-x_{1}^{2})}\,\Big). \rule{0pt}{20pt} 	
	\end{gather}

	From now on, we establish the following rule for our notation.
	Any point on the lower boundary is denoted by a capital Latin letter 
	and the first coordinate of this point is denoted by the corresponding small letter.
	For example, we write $U$ for $(u,u^2)$ (see Figures~\ref{fig:opk} and~\ref{fig:olk}).
	
	Let $B$ be a Bellman candidate on~$\Rt$ or $\Lt$. Since the function~$B$ must be linear on the linear extremals and 
	satisfy the boundary condition 
	${B(U) = f(u)}$, it follows that $B$ can be written as
	\begin{equation}\label{e34}	
		B(x_1,x_2) = m(u)(x_1 - u) + f(u).
	\end{equation}
\begin{figure}[H]
\includegraphics{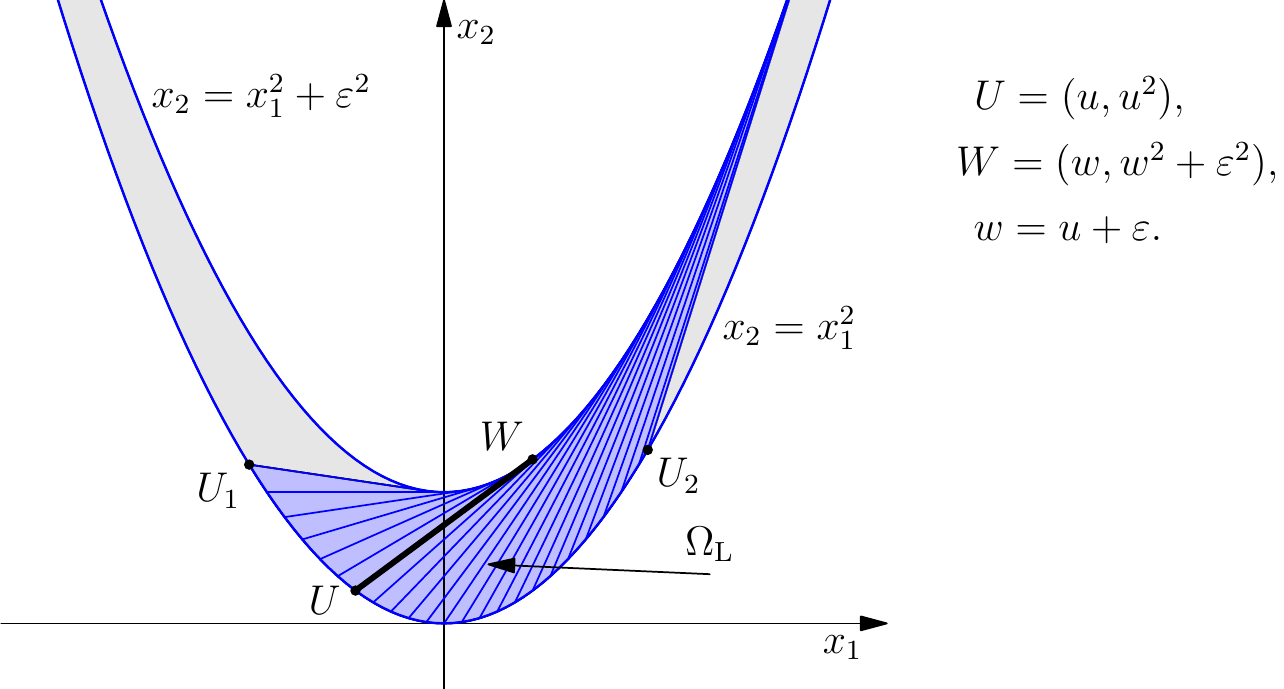}
\caption{A domain $\Lt$ with the left tangents.}
\label{fig:olk}
\end{figure}

	Consider case~(R). Using representation~(\ref{e34}) and the equation
	\begin{equation}\label{e38}
		u_{x_2} = \frac{1}{2(x_1-u+\eps)},
	\end{equation}
	by direct calculation, we obtain the following identity:
	$$
		t_2 = B_{x_2} = \frac{m'(u)}{2} - \frac{\eps m'(u) + m(u) - f'(u)}{2(x_1-u+\eps)}.
	$$
	If $u$ is fixed, the function $t_2$ must be constant. Therefore,
	\begin{gather}
		\label{e35}
		\eps m'(u) + m(u) - f'(u) = 0;\\
		\label{e36}
		t_2 = \frac{m'(u)}{2}.
	\end{gather}
	All the solutions of equation~(\ref{e35}) are of the form
	\begin{equation}\label{e37}
		\mrt(u) = e^{-u/\eps}\bigg(A + \eps^{-1}\int\limits_{u_1}^u f'(t)e^{t/\eps}\,dt\bigg),
	\end{equation}
	where $A$ is an integration constant. It is clear that
	\begin{equation}\label{tavt}
		A = e^{u_1/\eps}\mrt(u_1).
	\end{equation}
	Substituting solution~(\ref{e37}) into representation~(\ref{e34}) and expressing $u$ in terms of $x$ by~(\ref{e32}), we obtain 
	a family of functions (we still have a free parameter $A$) whose derivatives are constant along extremals~(R) foliating~$\Rt(u_1,u_2)$. 
	We denote such functions by~$\Brt(x;\,u_1,u_2)$.
	
	Next, by virtue of~(\ref{e36}), we can write
	$$
		\Brt_{x_2x_2} = t_2'(u)u_{x_2} = \frac{\mrt''(u)u_{x_2}}{2}.
	$$
	Using this equation and identity~(\ref{e38}), we see that the condition ${\Brt_{x_2x_2} \le 0}$ is equivalent to ${\mrt''(u) \le 0}$, 
	$u\in(u_1,u_2)$. 
	We recall that this condition is necessary for the local concavity of $\Brt(x;\,u_1,u_2)$.
	But in the situation being considered, this condition is also sufficient for the local concavity of the candidate. Indeed, if it is satisfied,
	then the function~$\Brt$ is concave along the direction $x_2$ and linear along an extremal. Since these directions are non-collinear, it follows
	that $\Brt$ is locally concave.
	
	Now we obtain a formula for $\mrt''$.
	Differentiating equation~(\ref{e35}) twice and solving it with respect to $m''$, we get
	\begin{equation}\label{e311}
		\mrt''(u) = e^{(u_1-u)/\eps}\mrt''(u_1) + \eps^{-1}e^{-u/\eps}\int\limits_{u_1}^u f'''(t)e^{t/\eps}\,dt.
	\end{equation}

	Reasoning for extremals~(L) in a similar way, we get the following relations:
	\begin{gather}
		\label{e39}
		-\eps m'(u) + m(u) - f'(u) = 0;\\
		\label{e323}
		t_2 = B_{x_2} = \frac{m'(u)}{2}.		
	\end{gather}
	All the solutions of equation~(\ref{e39}) have the following form:
	\begin{equation}\label{e310}
  	\mlt(u) = e^{u/\eps}\bigg(A + \eps^{-1}\int\limits_{u}^{u_2} f'(t)e^{-t/\eps}\,dt\bigg),  
  \end{equation}
  where 
  $$
  	A = e^{-u_2/\eps}\mlt(u_2).
  $$
  Setting $m(u) = \mlt(u)$ in~(\ref{e34}) and expressing $u$ in terms of $x$ by relation~(\ref{e33}), we obtain 
  the function $B(x) = \Blt(x;\,u_1,u_2)$ on $\Lt(u_1,u_2)$ with the partial derivatives that are constant along extremals~(L).
  The local concavity of the function~$\Blt$ is equivalent to the
	condition ${\mlt''(u) \ge 0}$ for all $u\in(u_1,u_2)$, and $\mlt''(u)$ satisfies
	\begin{equation}\label{e312}
  	\mlt''(u) = e^{(u-u_2)/\eps}\mlt''(u_2) + \eps^{-1}e^{u/\eps}\int\limits_{u}^{u_2} f'''(t)e^{-t/\eps}\,dt.
  \end{equation}
    
	We summarize this section.	 
	\begin{Prop}\label{S4}	
		Suppose the subdomain $\Rt(u_1,u_2) \subset \Omega_\eps$ is foliated by extremals~\textup{(R)} entirely. 
		Then a Bellman candidate in this subdomain has the form 
		\begin{equation}\label{e321}
  		\Brt(x;\,u_1,u_2) = \mrt(u)(x_1 - u) + f(u),
  		\end{equation}
  		where $\mrt(u)$ satisfies~\textup{(\ref{e37})} and $u = u_{\mathrm{R}}(x_1,x_2)$ 
  		can be calculated by~\textup{(\ref{e32})}. 
  		Besides\textup, the function $\mrt''(u)$ must satisfy ${\mrt''(u) \le 0}$\textup{,} $u\in(u_1,u_2)$.
  		
  		If the subdomain $\Lt(u_1,u_2) \subset \Omega_\eps$ is foliated by extremals~\textup{(L)} entirely\textup, 
  		then all the Bellman candidates in it have the form
  		\begin{equation}\label{e322}
  			\Blt(x;\,u_1,u_2) = \mlt(u)(x_1 - u) + f(u),
  		\end{equation}
  		where $\mlt(u)$ satisfies~\textup{(\ref{e310})} and $u = u_{\mathrm{L}}(x_1,x_2)$ 
  		is calculated by~\textup{(\ref{e33})}. 
  		Besides\textup, we must require that
  		${\mlt''(u) \ge 0}$\textup, $u\in(u_1,u_2)$.
  	\end{Prop}
  	
\subsection{Family of tangents coming from \texorpdfstring{$\pm\infty$}{the infinity}}\label{s32}
	Consider a domain $\Rt(-\infty,u_2)$ unbounded on the left and foliated by the right tangents.
	It turns out that the Bellman candidate $\Brt(x;\,-\infty,u_2)$ in it can be chosen uniquely by minimality 
	considerations. Similarly,
	if we consider a subdomain~$\Lt(u_1,+\infty)$ unbounded on the right, the minimal Bellman candidate 
	$\Blt(x;\,u_1,+\infty)$ can also be chosen uniquely.
	
	From~(\ref{e311}), (\ref{e35}), and~(\ref{tavt}) it follows easily that 
	$$
		\eps \mrt''(u)e^{u/\eps} = \big(f''(u_1) - \eps^{-1}f'(u_1)\big)e^{u_1/\eps} + \eps^{-1}A + 
		\int\limits_{u_1}^u f'''(t)e^{t/\eps}\,dt.
	$$
	Let $u_1 = -\infty$. Using limit relations~(\ref{e25}), we get
	$$
		\eps \mrt''(u)e^{u/\eps} = \eps^{-1}A + 
		\int\limits_{-\infty}^u f'''(t)e^{t/\eps}\,dt.
	$$
	Since $\mrt''(u) \le 0$ for all $u \in (-\infty,u_2)$ and integral on the right-hand side tends to zero 
	as $u\to-\infty$, we have $A \le 0$. On the other hand,
	$$
		\Brt(x;\,-\infty,u_2) = 
		e^{-u/\eps}\Bigg[A + \eps^{-1}\int\limits_{-\infty}^u f'(t)e^{t/\eps}\,dt\Bigg] 
		(x_1 - u) + f(u),
	$$
	and since $x_1 \le u$, the expression on the right is minimal when $A=0$. 
	Therefore, $\mrt(u) = \mrt(u;\,-\infty)$, where
	\begin{align}
		\label{e314}
		&\mrt(u;\,-\infty) = \eps^{-1}e^{-u/\eps}\int\limits_{-\infty}^u f'(t)e^{t/\eps}\,dt;\\
		\label{e315}
		&\mrt''(u;\,-\infty) = \eps^{-1}e^{-u/\eps}\int\limits_{-\infty}^u f'''(t)e^{t/\eps}\,dt.
	\end{align}
	
	Treating the case of left extremals~(L) for $u_2 = +\infty$ in a similar way, we have
	$\mlt(u) = \mlt(u;\,+\infty)$, where
	\begin{align}
		\label{e317}
		&\mlt(u;\,+\infty) = \eps^{-1}e^{u/\eps}\int\limits_u^{+\infty} f'(t)e^{-t/\eps}\,dt;\\
		\label{e318}
		&\mlt''(u;\,+\infty) = \eps^{-1}e^{u/\eps}\int\limits_u^{+\infty} f'''(t)e^{-t/\eps}\,dt.
	\end{align}
	
	We sum up the results of this section in the following proposition.
	\begin{Prop}\label{S5}
		Suppose the subdomain $\Rt(-\infty,u_2) \subset \Omega_\eps$ is foliated by extremals~\textup{(R)} entirely.
		In this domain\textup, we define the function~$\Brt$ by the formula
		\begin{equation}\label{e313}
			\Brt(x;\,-\infty,u_2) =  \mrt(u;\,-\infty)\,(x_1 - u)+ f(u),
		\end{equation}
		where $\mrt(u;\,-\infty)$ is given by~\textup{(\ref{e314})} and $u=u_{\mathrm{R}}(x_1,x_2)$ 
		can be calculated by~\textup{(\ref{e32})}. Also assume that
		$$
			\mrt''(u;\,-\infty) \le 0,\quad u\in(-\infty,u_2),
		$$
		where $\mrt''(u;\,-\infty)$ is calculated by~\textup{(\ref{e315})}.
		Then $\Brt(x;\,-\infty,u_2)$ is the minimal Bellman candidate in $\Rt(-\infty,u_2)$.
		
		Next\textup, suppose the subdomain $\Lt(u_1,+\infty) \subset \Omega_\eps$ is foliated by extremals~\textup{(L)}.
		In this domain\textup, we consider the function~$\Blt$ defined by the formula
		\begin{equation}\label{e316}
			\Blt(x;\,u_1,+\infty) = \mlt(u;\,+\infty)\,(x_1 - u) + f(u),
		\end{equation}
		where $\mlt(u;\,+\infty)$ is given by~\textup{(\ref{e317})} and $u = u_{\mathrm{L}}(x_1,x_2)$ can
		be calculated by~\textup{(\ref{e33})}. Also we assume that
		$$
			\mlt''(u;\,+\infty) \ge 0,\quad u\in(u_1,+\infty).
		$$
		Then $\Blt(x;\,u_1,+\infty)$ is the minimal Bellman candidate in $\Lt(u_1,+\infty)$.
	\end{Prop}
		
\subsection{Optimizers for the families of tangents}\label{s33}
	Let $B$ be a Bellman candidate in the whole domain~$\Omega_\eps$. We also assume that some part~$\Rt(u_1,u_2)$ 
	of~$\Omega_\eps$ is foliated by the right extremal tangents. 

	 From Section~\ref{s24}, it follows that delivery curves in $\Rt$ run along the upper parabola or along the tangents.
	Also, it can be easily seen that these delivery curves must be left. Indeed, draw a delivery curve up to
	some point on the upper boundary. By the definition of delivery curves, if we cut off a small interval from 
	the domain of a test function, we get an optimizer for the Bellman point corresponding to the residual interval.
	This point must be close to the initial point, and the Bellman point corresponding to the small interval
	can be located far away, almost on the lower boundary. Since the points corresponding to this split 
	run almost along a right extremal tangent, 
	the distant point must be on the right of the initial point. Therefore, the curve runs from the left.
		
	Consider the point $W_1 = \big(u_1-\eps, (u_1-\eps)^2+\eps^2\big)$ on the upper parabola. 
	Suppose some convex delivery curve~$\dc$ runs from a neighbor subdomain 
	and ends at $W_1$ (i.e. $\dc(r) = W_1$). We will see that this curve can be continued
	up to each point of~$\Rt$ in the way shown on Figure~\ref{fig:kdpr}: 
	we continue it along the upper parabola and, after that, along the tangent leading to the destination point. 
	Therefore, we will obtain optimizers for all the points of~$\Rt$. 
	The point~$W_1$ is called \emph{an entry node}: the information from the neighbor 
	subdomain is transmitted through it only.
\begin{figure}[H]
\begin{center}
\includegraphics{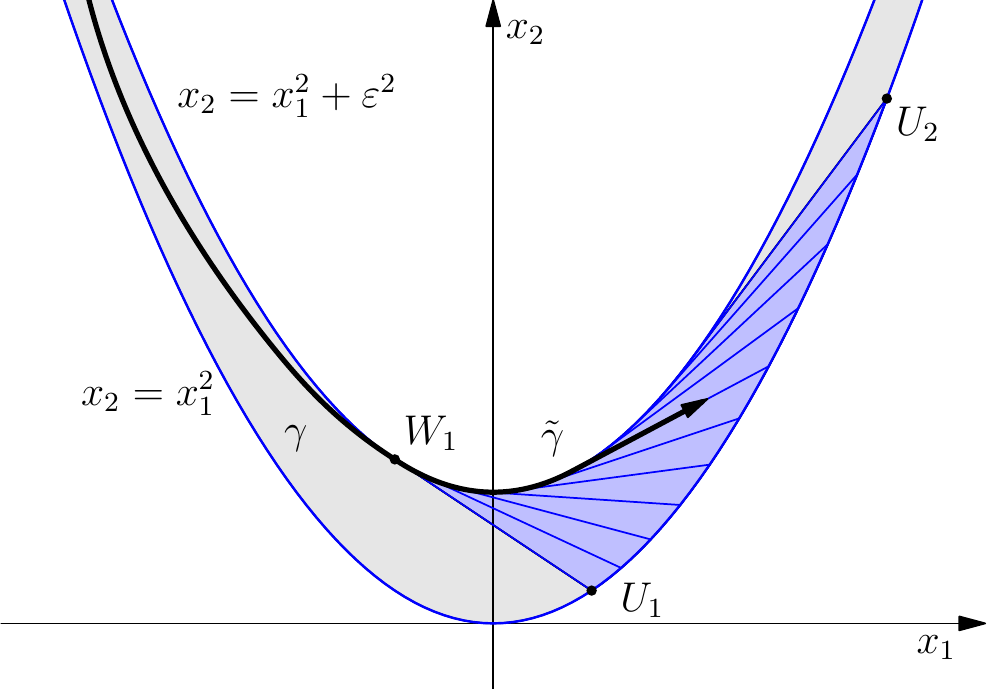}
\caption{Delivery curves in $\Rt$.}
\label{fig:kdpr}
\end{center}
\end{figure}
	
	For a subdomain~$\Lt(u_1,u_2)$ foliated by the left tangents, the situation is symmetric. 
	The point $W_2 = \big(u_2+\eps, (u_2+\eps)^2+\eps^2\big)$ is its entry node. If a convex right 
	delivery curve~$\gamma$ reaches this point (i.e. $\gamma(l)=W_2$), then $\gamma$ can be continued up to each 
	point in~$\Lt(u_1,u_2)$ (see Figure~\ref{fig:kdpl}).
\begin{figure}[H]
\begin{center}
\includegraphics{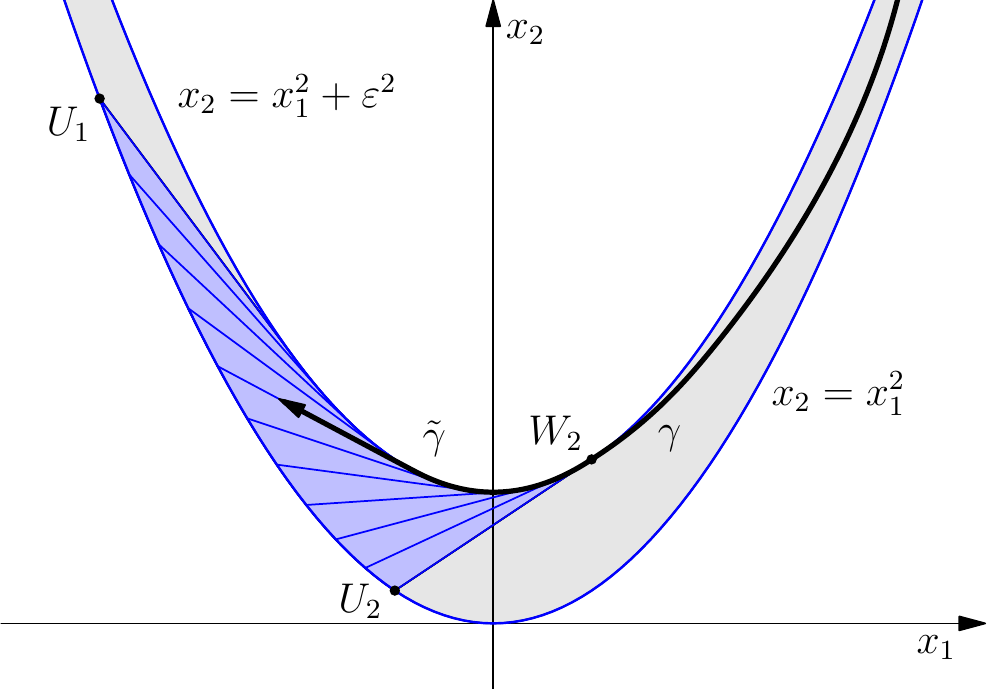}
\caption{Delivery curves in $\Lt$.}
\label{fig:kdpl}
\end{center}
\end{figure}

\paragraph{Points on the upper parabola.}
	Let $\dc$ be a convex left delivery curve that is generated by a test function~$\ex$ defined on the segment 
	$I=[l,r]$. Also, suppose it ends at the entry node~$W_1$ of the domain $\Rt(u_1,u_2)$. First, we prove that this curve can be 
	continued up to any point $W \in \Rt(u_1,u_2)$, lying on the upper boundary, in such a way that the resulting
	curve~$\widetilde \dc$ will also be a left delivery curve.	 
	
	Since delivery curves run either along extremals or along the upper parabola
	and extremals touch the upper parabola, we may assume that the convex curve~$\dc$ 
	also touches the upper parabola at the point~$W_1$. Thus, the curve~$\widetilde \dc$ cannot
	avoid being convex. But we do not regard these considerations as rigorous, and so 
	the convexity of~$\widetilde \dc$ will appear as a requirement.
		
	Now, continue the left delivery curve~$\dc$ along the upper parabola with preservation of the convexity. 
	In order to prove that the continuation~$\widetilde\dc$ is also a left delivery curve,
	we must construct a test function~$\widetilde \ex$ defined on some segment $[l,\widetilde r]$, $\widetilde r > r$,
	such that $\widetilde\dc$ is 
	generated by this function in the sense of~(\ref{eq7}) and relation~(\ref{B_extr_curve}) is fulfilled 
	for $\widetilde \ex$ and $\widetilde \dc$.

	We set $\widetilde \ex(s) = \ex(s)$ for $s \in I$. The question is how to define $\widetilde \ex(s)$ for $s>r$. 
	For $s>r$, the curve 
	$\widetilde\dc(s) = \big(\widetilde\dc_1(s), \widetilde \dc_2(s)\big)$ runs along the upper parabola, so
	$$
		\widetilde\dc_1(s)=\frac{1}{s-l}\int\limits_l^s \widetilde\ex(t)\,dt\quad\mbox{and}\quad
		\widetilde\dc_2(s)=\frac{1}{s-l}\int\limits_l^s {\widetilde\ex}^2(t)\,dt={\widetilde\dc}_1^2(s)+\eps^2.
	$$
	Therefore,
	$$
		\widetilde\ex^2(s)=\big([(s-l)\widetilde\dc_1]'\big)^2=\big((s-l)(\widetilde\dc_1^2+\eps^2)\big)',
	$$
	i.e.
	$$
		\widetilde\dc_1^2+2(s-l)\widetilde\dc_1\widetilde\dc_1'+((s-l)\widetilde\dc_1')^2=
		\widetilde\dc_1^2+2(s-l)\widetilde\dc_1\widetilde\dc_1'+\eps^2.
	$$
	Since we build the left delivery curve, we expect the function~$\widetilde\dc_1$ to be non-decreasing. 
	Therefore, taking the square root, we obtain
	\begin{equation}\label{extreq}
		\widetilde\dc_1'(s)=\frac{\eps}{s-l}.
	\end{equation}
	Note that the other root gives us the  backwards motion along the parabola.  
	Solving the equation~(\ref{extreq}), we get
	\begin{equation}\label{x_1}
		\widetilde\dc_1(s)=\eps\log(s-l) + c,
	\end{equation}
	and
	\begin{equation}\label{extrfn}
		\widetilde\ex(s)=\big((s-l)\widetilde\dc_1\big)'=\eps\log(s-l)+c+\eps.
	\end{equation}
	Now, using the continuity of the delivery curve at $s=r$, we obtain the constant in~(\ref{x_1}) and~(\ref{extrfn}):
	$$
		u_1-\eps=\dc_1(r)=\widetilde \dc_1(r)=\eps\log(r-l) + c.
	$$
	Therefore, $c=u_1-\eps\log(r-l)-\eps$ and equation~(\ref{extrfn}) takes the form
	\begin{equation}\label{eq10}
		\widetilde\ex(s)=\eps\log\frac{s-l}{r-l} + u_1, \quad s\in(r,\widetilde r],
	\end{equation}
	where the choice of $\widetilde r$ depends on the point we want to reach.

	Now we verify that $\widetilde\ex$ is an admissible test function and $\widetilde \dc$ 
	is a left delivery curve generated by this function, i.e, we prove the following statement.
	\begin{Prop}\label{S9}
		Consider a subdomain $\Rt(u_1,u_2)$\textup{,} $u_1 > -\infty$\textup, 
		foliated by the right tangents. Suppose some test function~$\ex$ defined on $I=[l,r]$ generates 
		a convex left delivery curve~$\dc$ that lies on the left of~$\Rt$ and ends at the entry node 
		$W_1=\big(u_1-\eps, (u_1-\eps)^2+\eps^2\big)$ 
		\textup{(}i.e.\textup, $\dc(r) = W_1$\textup{)}. 
		We continue this curve to the right along the upper parabola without leaving~$\Rt$. 
		If the resulting curve~$\widetilde \dc$ is convex\textup, then it is a left delivery curve generated by the
		test function
		$$	
			\widetilde\ex(s) =	\left\{ 
			\begin{aligned}
				\ex(s),\qquad &s \in I;\\
				\eps \log\frac{s-l}{r-l}+u_1,\;\; &s \in [r,\widetilde r].
			\end{aligned} \right.
		$$
	\end{Prop}
	\begin{proof}
		The fact that $\widetilde\ex$ generates $\widetilde \dc$ in the sense of~(\ref{eq7}) 
		follows from the construction of $\widetilde\ex$ (see the above considerations). 
		It remains to verify two points. First, it must be proved that $\widetilde\ex$ belongs to 
		$\BMO_\eps([l,\widetilde r])$. Second, we must verify 
		relation~(\ref{B_extr_curve}) for the function~$\widetilde\ex$, the curve~$\widetilde \dc$, and 
		the candidate~$B$.
		
		The fact that $\widetilde \ex \in \BMO_\eps([l,\widetilde r])$ follows from the geometric lemma~\ref{L4}.
		Indeed, if $[a,b]\subset I$, then the Bellman point $x^{[a,b]}$ is in $\Omega_\eps$, because $\ex\in\BMO_\eps(I)$. 
		If $b>r$, then the conditions of the lemma just mentioned are fulfilled.
	
		We turn to verification of~(\ref{B_extr_curve}). In $\Rt$, the Bellman candidate~$B$  
		coincides with~$\Brt$ (see Proposition~\ref{S4}). 
		Therefore, we must check that 
		$$
			\Brt(\widetilde\dc(s);\,u_1,u_2) = \av{f(\widetilde\ex)}{[l,s]},\quad s\in(r,\widetilde r].
		$$
		By~(\ref{e321}) and~(\ref{e37}), we have
		\begin{equation}\label{eq9}
			\Brt(\widetilde\dc(s)) = -e^{-u/\eps}\bigg(\eps A + \int\limits_{u_1}^u f'(t)e^{t/\eps}\,dt\bigg) + f(u),
		\end{equation}
		where $u = \widetilde\dc_1(s)+\eps$.
		On the other hand, using the same relations and the continuity of~$B$, we get
		$$
			B(\dc(r)) = -\eps e^{-u_1/\eps} A + f(u_1).
		$$

		Now, expressing $A$ in terms of $B(\dc(r))$, substituting the resulting expression into~(\ref{eq9}), 
		and then integrating by parts, we have
		\begin{align*}
			\Brt(\widetilde\dc(s))
			&=e^{(u_1-u)/\eps} B(\dc(r)) -\int\limits_{u_1}^u f'(t)e^{(t-u)/\eps}\,dt - e^{(u_1-u)/\eps} f(u_1) + f(u)
			\\
			&=e^{(u_1-u)/\eps} B(\dc(r))+\int\limits_{u_1}^u f(t)\,d(e^{(t-u)/\eps}).
		\end{align*} 
		Further, using~(\ref{x_1}), we obtain
		$$
			e^{(u_1-u)/\eps} = e^{(\dc_1(r)-\dc_1(s))/\eps} = \frac{r-l}{s-l}.
		$$
		We make the substitution $t=\widetilde\ex(\tau)$. Using formula~(\ref{eq10}) and the previous equation, 
		we get
		$$
			e^{(t-u)/\eps} = \frac{\tau-l}{r-l} e^{(u_1-u)/\eps} = \frac{\tau-l}{s-l}.
		$$
		It follows from the above that $\tau$ runs over $(r,s]$ provided $t$ runs over $(u_1,u]$.
		Using the substitution just described and the fact that $B(\dc(r)) = \av{f(\ex)}{[l,r]}$, we have
		\begin{align*}
			\Brt(\widetilde\dc(s)) 
			&= \frac{r-l}{s-l} B(\dc(r)) + \frac{1}{s-l}\int\limits_r^s f(\widetilde\ex(\tau))\,d\tau
			\\
			&= \frac{1}{s-l}\int\limits_l^s f(\widetilde\ex(\tau))\,d\tau
			= \av{f(\widetilde\ex)}{[l,s]}.
		\end{align*}
		This concludes the proof.
	\end{proof}

	Similarly, we can prove a symmetric proposition for $\Lt(u_1,u_2)$, $u_2<+\infty$.
	\begin{Prop}\label{S10}
		Consider a subdomain $\Lt(u_1,u_2)$\textup, $u_2<+\infty$\textup, foliated by the left tangents. Suppose 
		some test function~$\ex$ defined on $I=[l,r]$ generates 
		a convex right delivery curve $\dc$ that lies on the right of~$\Lt$ and ends at the entry 
		node $W_2 = \big(u_2+\eps, (u_2+\eps)^2+\eps^2\big)$ 
		\textup{(}i.e.\textup, $\dc(l) = W_2$\textup{)}. We continue this curve to the left along the upper parabola 
		without leaving~$\Lt$. If the resulting curve~$\widetilde \dc$ is convex\textup, 
		then it is a right delivery curve generated by the
		test function
		$$	
			\widetilde\ex(s) =	\left\{ 
			\begin{aligned}
				\eps \log\frac{r-l}{r-s}+u_2,\;\; &s \in [\widetilde l,l];\\
				\ex(s),\qquad &s \in I.
			\end{aligned} \right.
		$$
	\end{Prop}

\paragraph{Points inside the domain.}
	We have explained how to continue delivery curves from entry node~$W_1$ (or~$W_2$) to a 
	point in~$\Rt$ (respectively, in~$\Lt$) lying on the upper parabola. 	
	It occurs that each of the other points in these subdomains (but the points on the lower boundary) 
	can be reached if we continue the delivery curve along the corresponding tangent that contains this point.
	
	Now, let $\dc$ be a left delivery curve that is generated by a test function~$\ex$ defined on $I=[l,r]$.
	Suppose we have continued this curve from the point~$\dc(r)$ along
	some straight-line segment that ends at some point~$U$ on the lower boundary (e.g. along an extremal tangent).
	We want to find a function~$\widetilde \ex$ that generates the resulting curve~$\widetilde\dc$. 
	Set $\widetilde \ex(s) = \ex(s)$ for $s\in I$ and consider the case $s>r$. 
	Since three points $\tilde \gamma(r)$, $\tilde \gamma(s)$, and $U$ lie on a single line, we have
	$$
		\frac{u^2 - \widetilde\dc_2(s)}{u-\widetilde\dc_1(s)} = \frac{u^2 - \dc_2(r)}{u-\dc_1(r)},
	$$
	i.e.
	$$
		\big(u-\dc_1(r)\big)\bigg((s-l)u^2 - \int\limits_l^s\!\! \widetilde\ex^2\bigg) = 
		\big(u^2-\dc_2(r)\big)\bigg((s-l)u - \int\limits_l^s\!\! \widetilde\ex\bigg).
	$$
	Differentiating this identity with respect to~$s$, we obtain the quadratic equation on $\widetilde\ex(s)$:
	$$
		\big(u-\dc_1(r)\big)\big(u^2-\widetilde\ex^2(s)\big) = \big(u^2-\dc_2(r)\big)\big(u-\widetilde\ex(s)\big).
	$$
	We will see that its solution $\widetilde\ex(s) = u$, $s>r$, is suitable for us. The second solution corresponds to 
	the reverse motion along the straight line containing the segment $[\dc(r),U]$.
	
	It turns out that the following three conditions are sufficient for $\widetilde\dc$ to be a delivery curve:
	the curve~$\widetilde\dc$ must still be convex,
	the straight line that contains $[\dc(r),U]$ must lie below the upper boundary of~$\Omega_\eps$, 
	and the Bellman candidate~$B$ must be linear along the segment $[\dc(r),U]$. In our 
	situation where a delivery curve in continued in $\Rt$ along one of the extremal tangents, all these
	conditions are surely satisfied.	

	We prove the following general proposition.
	\begin{Prop}\label{S11}
		Let $\dc$ be a convex left delivery curve that is generated by a test function~$\ex$ 
		defined on $I=[l,r]$. We draw a straight-line segment from the point $\dc(r)$ to some
		point $U$ on the lower boundary with preservation of the convexity. 
		Suppose $B$ 
		is linear on the segment $[\dc(r),U]$ and the line
		containing this segment lies below the upper boundary. Then we can 
		continue~$\dc$ up to any point inside $[\dc(r),U]$ so that the resulting 
		curve~$\widetilde \dc$ will also be a left delivery curve. In this case\textup, 
		the curve~$\widetilde \dc$ is generated by 
		the test function
		$$
			\widetilde\ex(s) =	
			\begin{cases}
				\ex(s), &s \in I;\\
				u, &s \in [r,\widetilde r].
			\end{cases}
	  $$
	\end{Prop}
	\begin{proof}
		We must verify~(\ref{eq7}) and~(\ref{B_extr_curve}) for 
		$\widetilde \ex$, $\widetilde \dc$, and $B$. We must also make sure that 
		$\widetilde \ex \in \BMO_\eps([l,\widetilde r])$.
				
		Let $s \in (r,\widetilde r]$. For such $s$, we verify that the points of the curve 
		$\widetilde \dc(s) = \big(\av{\widetilde\ex}{[l,s]},\av{\widetilde\ex^2}{[l,s]}\big)$
		get into $[\dc(r),U]$. We also check that we can reach any point inside $[\dc(r),U]$
		provided $\widetilde r$ is sufficiently large.
		Indeed, we have the identity
		$$
			\int\limits_l^s \widetilde{\ex}^k(t)\,dt = \int\limits_l^r \ex^k(t)\,dt + (s-r)u^k,\quad 
			\mbox{for} \quad k=1,2,
		$$
		which implies the representation
		\begin{equation}\label{eq8}
			\widetilde\dc(s) = \alpha_- \dc(r) + \alpha_+ U,\quad
			\mbox{where}\quad \alpha_- = \frac{r-l}{s-l}\quad \mbox{and}\quad \alpha_+ = \frac{s-r}{s-l}.
		\end{equation}
		Thus, we have proved that $\widetilde \dc$ and $\widetilde \ex$ are related by~(\ref{eq7}). 
		
		The fact that $\widetilde \ex$ belongs to $\BMO_\eps([l,\widetilde r])$ follows 
		from the geometric lemma~\ref{L4}.
		
		It remains to verify equation~(\ref{B_extr_curve}). 
		Using the linearity of $B$ on $[\dc(r),U]$ and  
		representation~(\ref{eq8}), we obtain
		\begin{align*}
			B(\widetilde\dc(s))&=\alpha_- B(\dc(r))+\alpha_+ B(U)\\		
			&=\frac{r-l}{s-l}\av{f(\ex)}{[l,r]}+\frac{s-r}{s-l} f(u)\\			
			&=\frac{1}{s-l}\bigg(\int\limits_l^r f(\ex(t))\,dt+\int\limits_r^s f(u)\,dt\bigg)\\
			&=\av{f(\widetilde\ex)}{[l,s]}.
		\end{align*}
		The proposition is proved.
	\end{proof}
		
	Similarly, we can prove a symmetric statement for right delivery curves.
	\begin{Prop}\label{S12}
		Let $\dc$ be a convex right delivery curve that is generated by a test function~$\ex$ defined on $I=[l,r]$. 
		We draw a straight-line segment from the point $\dc(l)$ to some
		point $U$ on the lower boundary with preservation of the convexity. 
		Suppose $B$ is linear on the segment $[U,\dc(l)]$ and the line
		containing this segment lies below the upper parabola. Then we can continue~$\dc$ up to any point 
		inside $[U,\dc(l)]$ so that the resulting 
		curve~$\widetilde \dc$ will also be a right delivery curve. In this case\textup, 
		the curve~$\widetilde \dc$ is generated by the test function
		$$
			\widetilde\ex(s) =
			\begin{cases}
				u, &s \in [\widetilde l, l];\\
				\ex(s), &s \in I.
			\end{cases}
	  $$		
	\end{Prop}
	Applying Propositions~\ref{S9} and~\ref{S11} for the case $\Rt(u_1,u_2)$ or Propositions~\ref{S10} and~\ref{S12} for the case
	$\Lt(u_1,u_2)$, we can continue delivery curves from entry nodes up to any points of these domains, except 
	the points on the lower boundary.	
	But for each point $U$ on the lower boundary, we can take the optimizer~$\ex$ to be equal to $u$ 
	on the whole interval~$I$, because of the boundary condition (although it is clear without any
	optimizers that $\Bell$ and $B$ coincide on the lower boundary).
	
	It is worth mentioning that we only continue delivery curves already constructed, but not build new ones, i.e.
	we require some information from the left neighbor of~$\Rt$ or from the right neighbor of~$\Lt$. 
	In~\cite{SlVa2}, the domain $\Rt(u_1,u_2)$ with $u_1 \ne -\infty$ 
	was called \emph{left-incomplete}, and the domain $\Lt(u_1,u_2)$ with $u_2 \ne +\infty$ was called
	\emph{right-incomplete}.
		
	\paragraph{Unbounded domains.}
	Discuss domains $\Rt(-\infty,u_2)$ and $\Lt(u_1,+\infty)$ unbounded on one side. As usual, 
	we treat in detail only $\Rt(-\infty,u_2)$ and left 
	delivery curves in it. It turns out that 
	we can draw a left delivery curve from $-\infty$ to every point of this domain 
	(except the points on the lower boundary). At this time, we do not need any extra information.
	
	Consider some curve $\dc = (\dc_1,\dc_2)$ that runs along the upper parabola from~$-\infty$ up to 
	some point $W = (w,w^2+\eps^2)$. According to the arguments 
	preceding Proposition~\ref{S9}, such a curve is generated by the function
	$$
		\ex(s)=\eps\log(s-l)+c+\eps
	$$
	defined on $I=[l,r]$, and
	$$
		\dc_1(s)=\eps\log(s-l) + c.
	$$
	We set $I=[l,r]=[0,1]$ and calculate~$c$:
	$$
		w = \dc_1(1) = c.
	$$
	As usual, the fact that $\ex$ lies in $\BMO_\eps(I)$ follows from Lemma~\ref{L4}.
	In order to prove equation~(\ref{B_extr_curve}), we must repeat the corresponding reasoning 
	from the proof of Proposition~\ref{S9}. But now we must integrate from
	$-\infty$ and the constant~$A$ is equal to zero.
	We got the following statement.
	\begin{Prop}\label{S13}
		Consider a subdomain $\Rt(-\infty,u_2)$ foliated by the right tangents.
		If a point $W = (w,w^2+\eps^2)$ of the upper parabola lies in this 
		subdomain\textup, then we can construct a left delivery 
		curve running along the upper parabola from~$-\infty$ to~$W$. Such a curve is generated by the test function
		$$
			\ex(s)=\eps\log s+w+\eps,\quad s\in[0,1].
		$$			
	\end{Prop}

	For $\Lt(u_1,+\infty)$, we can formulate a symmetric proposition about right delivery curves running
	along the upper parabola.
	\begin{Prop}\label{S14}
		Consider a subdomain $\Lt(-\infty,u_2)$ foliated by the left tangents.
		If a point $W = (w,w^2+\eps^2)$ of the upper parabola lies in this subdomain\textup, then we 
		can construct a right delivery curve running along the upper parabola from~$+\infty$ to~$W$. 
		Such a curve is generated by the test function
		$$
			\ex(s)=-\eps\log(1-s)+w-\eps,\quad s\in[0,1].
		$$			
	\end{Prop}
	Concerning the points of $\Rt(-\infty,u_2)$ and $\Lt(-\infty,u_2)$ not lying on the upper boundary,
	we can continue our delivery curves up to them using Propositions~\ref{S11} and~\ref{S12}.
	Thus, we have obtained the optimizers for all the points of this domains. 
	In~\cite{SlVa2}, if no extra information from neighbors was required for a domain, it was called 
	\emph{complete}.  
	
\subsection{Function \texorpdfstring{$f'''$}{f'''} does not change its sign}\label{s34}
	It is stated in Proposition~\ref{S5} that the function $\Brt(x;\, -\infty, +\infty)$, 
	defined by~(\ref{e313}), is a Bellman candidate in the whole domain~$\Omega_\eps$
	provided $f$ satisfies some integral condition. Thus, from Statement~\ref{S2}, it follows that $\Bell\le \Brt$. 
	On the other hand, we have constructed (see the previous section)
	optimizers for all the points of the domain $\Rt(-\infty,+\infty) = \Omega_\eps$. 
	This gives us the converse inequality $\Bell\ge \Brt$. We come to the following theorem.
	\begin{Th}\label{T1}
		Suppose $0<\eps<\eps_0$\textup{,} $f \in \W$\textup, and 
		$$
			\int\limits_{-\infty}^u f'''(t)e^{t/\eps}\,dt \le 0, \quad \forall u \in \mathbb{R}.
		$$
		Then
		$$
			\Bell(x;\,f) = \Brt(x;\, -\infty, +\infty),
		$$
		where the function on the right is defined by~\textup{(\ref{e313})}.
	\end{Th}
	Using the second part of Proposition~\ref{S5} and the optimizers constructed in the previous section, 
	we get the symmetric theorem.
	\begin{Th}\label{T1b}
		Suppose $0<\eps<\eps_0$\textup{,} $f \in \W$\textup, and
		$$
			\int\limits_u^{+\infty} f'''(t)e^{-t/\eps}\,dt \ge 0, \quad \forall u \in \mathbb{R}.
		$$
		Then
		$$
			\Bell(x;\,f) = \Blt(x;\, -\infty, +\infty),
		$$
		where the function on the right is defined by~\textup{(\ref{e316})}.
	\end{Th}
	Obviously, these theorems treat the case where $f'''$ 
	has one and the same sign a.e. on~$\mathbb{R}$.
	\begin{Cor}\label{cor1}
		Suppose $0<\eps<\eps_0$ and $f \in\W^{0}$.
		If $c_{0}=-\infty$\textup{,} then 
		$$
			\Bell(x;\,f) = \Brt(x;\, -\infty, +\infty),
		$$
		and if $c_{0}=+\infty$\textup{,} then 
		$$
			\Bell(x;\,f) = \Blt(x;\, -\infty, +\infty).
		$$
	\end{Cor}
	
\subsection{Examples}
\paragraph{Example 1. The exponential function.}
The Bellman functions for $f(t)=\pm e^t$ were constructed in~\cite{SlVa}.
The function $f(t)=e^t$ belongs to~$\W$ only if $\eps_0<1$. Therefore, all the further formulas are reasonable 
only for $\eps<1$. We see that the function $f'''(t)=e^t$ is positive on the whole line. 
Thus, by Corollary~\ref{cor1}, the domain~$\Omega_\eps$ is foliated entirely by the left tangents.
We come to the following formula:
\begin{align*}
\Bell(x_{1},x_{2};\,e^t)&=\mlt(u;\,+\infty)\,(x_{1}-u)+f(u)\\
&=(x_1-u)\cdot\eps^{-1}e^{u/\eps}\int\limits_{u}^{\infty}e^t \cdot e^{-t/\eps}\,dt\;+\;e^u\\
&=\left(\frac{x_1-u}{1-\eps}+1\right)e^u=\frac{1-\sqrt{x_{1}^{2}-x_{2}+\eps^{2}}}{1-\eps}\,e^u,
\end{align*}
where the function~$u$ for left tangents is defined by formula~(\ref{e33}):
$$
u(x_{1},x_{2}) =x_{1} - \Big(\eps - \sqrt{x_{1}^{2}-x_{2}+\eps^{2}}\, \Big).
$$

Similarly, if $f(t)=-e^t$, the whole domain is foliated by the right tangents. In this case, we have
\begin{align*}
\Bell(x_{1},x_{2};\,-e^t)&=\mrt(u;\,-\infty)\,(x_{1}-u)+f(u)\\
&=(x_1-u)\cdot\eps^{-1}e^{-u/\eps}\int\limits_{-\infty}^{u}(-e^t) \cdot e^{t/\eps}\,dt\;-\;e^u\\
&=-\left(\frac{x_1-u}{1+\eps}+1\right)e^u=-\frac{1+\sqrt{x_{1}^{2}-x_{2}+\eps^{2}}}{1+\eps}\,e^u,
\end{align*}
where the function~$u$ for right tangents is defined by formula~(\ref{e32}): 
$$
u(x_{1},x_{2}) =x_{1} + \Big(\eps - \sqrt{x_{1}^{2}-x_{2}+\eps^{2}}\, \Big).
$$
We recall that the Bellman function for $f(t)=-e^{t}$ solves the infimum problem for $f(t)=e^{t}$ 
(see Remark~\ref{rem1}).

\paragraph{Example 2. A third-degree polynomial.}
	The simplest example of a function $f$ such that the sign of $f'''$ does not change, 
	is an arbitrary third-degree polynomial.
	In such a case, it is sufficient to obtain the Bellman function for
	$f(t) = \pm{t^3}$ (see Remark~\ref{rem2}). We see that $f'''(t) = \pm{6}$. Thus, due to Corollary~\ref{cor1},
	the whole domain is foliated by the left tangents for $\Bell(x;\, t^{3})$ or, respectively, by the 
	right tangents for $\Bell(x;\, -t^{3})$. For any $\eps \in [0,+\infty)$, the analytical expression for the 
	Bellman function can be calculated by~(\ref{e316}) and~(\ref{e317}) (or by~(\ref{e313}) and~(\ref{e314}), 
	respectively).
	For $f'''(t)=6$, we have
	\begin{align*}
		\Bell(x_{1},x_{2};\, t^{3}) &= \mlt(u;\,+\infty)\,(x_{1}-u)+f(u)\\
		&=(x_{1}-u)\cdot\eps^{-1}e^{u/\eps} \int\limits_{u}^{\infty}3t^{2}e^{-t/\eps}\,dt+u^{3}\\
		&=(6\eps^{2}+3u^{2}+6 \eps u)(x_{1}-u)+u^{3},
	\end{align*} 
	where the function $u$ for left tangents is defined by formula (\ref{e33}): 
	\begin{equation*}
		u(x_{1},x_{2}) =x_{1} - \Big(\eps - \sqrt{x_{1}^{2}-x_{2}+\eps^{2}}\, \Big).
	\end{equation*}
	For $f'''(t) = -6$, we have
	\begin{align*}
		\Bell(x_{1},x_{2};\, -t^{3}) &= \mrt(u;\,-\infty)\,(x_{1}-u)+f(u)\\
		&=(x_{1}-u)\cdot\eps^{-1}e^{-u/\eps}\int\limits_{-\infty}^{u}-3t^{2}e^{t/\eps}\,dt-u^{3}\\
		&=(-6\eps^{2}-3u^{2}+6\eps u)(x_{1}-u)-u^{3},
	\end{align*} 
	where the function $u$ for right tangents is defined by formula~(\ref{e32}): 
	\begin{equation*}
	u(x_{1},x_{2}) =x_{1} + \Big(\eps - \sqrt{x_{1}^{2}-x_{2}+\eps^{2}}\, \Big).
	\end{equation*}
	It is worth noting that
	\begin{align*}
	\Bell(x_{1},x_{2};\,t^{3})&=-3(x_{1}^{2}-x_{2})\eps-2x_{1}^{3}+3x_{2}x_{1}+O(\eps^{-1}) \quad 
	\text{as} \quad \eps \to \infty;\\ \rule{0pt}{15pt}
	\Bell(x_{1},x_{2};\,-t^{3})&=-3(x_{1}^{2}-x_{2})\eps+2x_{1}^{3}-3x_{2}x_{1}+O(\eps^{-1})\quad 
	\text{as} \quad \eps \to \infty.
	\end{align*}	

\section{Transition from right tangents to left ones}
	In this chapter, we treat the case when there are two domains of left and right tangents simultaneously. 
	There is also a triangle domain between them, where our Bellman candidate is linear. The reader 
	can glance at Figure~\ref{fig:ppklu} to understand what is meant. 
	In Section~\ref{s41}, we will construct a function corresponding to such a foliation and obtain some 
	conditions guaranteeing that this function is a Bellman candidate.
	Again, we note that the arguments in Section~\ref{s41} partially repeat the corresponding arguments in~\cite{SlVa2}. 
	Further, in Section~\ref{s42}, we will build optimizers for the triangle domain of linearity. 
	Finally, in Section~\ref{s43}, we will summarize this
	chapter and describe the conditions on~$f$ under which $\Bell(x;\,f)$ corresponds to the foliation discussed.
	In particular, it turns out that the transition between right and left tangents can occur for
	$f\in\W^1$ with $c_0=-\infty$ and $c_1=+\infty$, i.e. if  the sign of~$f'''$ changes once from minus to plus. 

\subsection{Angle}\label{s41}
	Let	$u_1 < v < u_2$. Consider two subdomains $\Rt(u_1,v)$ and $\Lt(v,u_2)$ foliated by extremals~(R) and~(L), 
	respectively. We can see a subdomain in the form of \emph{an angle} lying between $\Rt$ and $\Lt$. 
	It is bounded by the upper parabola and by the right and left tangents coming from the point
	$V = (v,v^2)$ (see Figure~\ref{fig:ppklu}):	
	$$
		\Ang(v) \df \big\{x\in \mathbb{R}^{2} \mid v-\eps \le x_1 \le v+\eps,\; 2vx_1-v^2+2\eps|v-x_1|\le x_2 \le x_1^2+\eps^2\big\}.
	$$	
\begin{figure}[H]
\begin{center}
\includegraphics{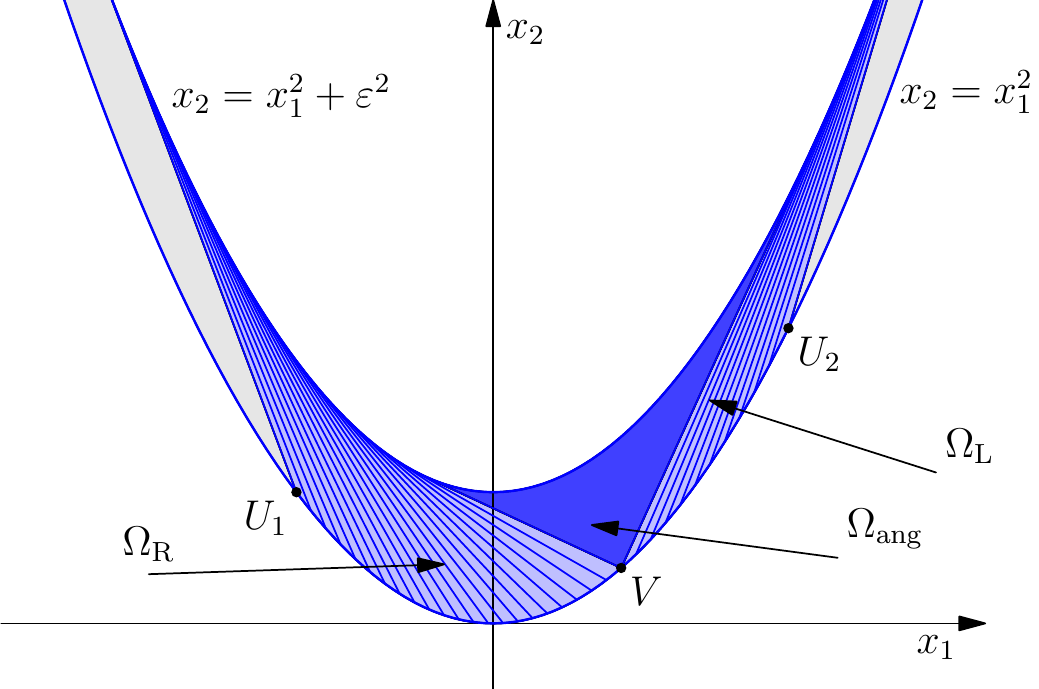}
\caption{Angle $\Ang$ lying between $\Rt$ and $\Lt$.}
\label{fig:ppklu}
\end{center}
\end{figure}
	Now we construct a Bellman candidate in the subdomain
	$$
		\RtLt(u_1,v,u_2) \df \Rt(u_1,v) \cup \Ang(v) \cup \Lt(v,u_2).
	$$
	We denote this candidate by $\Brtlt(x;\,u_1,v,u_2)$. The candidates in $\Rt$ and~$\Lt$ have been constructed already: 
	\begin{align*}
		&\Brtlt(x;\,u_1,v,u_2) = \Brt(x;\,u_1,v)\qquad \text{if}\quad x \in \Rt(u_1,v);\\ 
		&\Brtlt(x;\,u_1,v,u_2) = \Blt(x;\,v,u_2)\qquad \text{if}\quad x \in \Lt(v,u_2). 	
	\end{align*}
	We recall that $\Brt$ and $\Blt$ are, in fact, families of functions. 
	Concerning the domain $\Ang(v)$, the function we are looking for must be linear on it. 
	Indeed, by the continuity, the function~$\Brtlt$ is linear on the one-sided tangents that bound $\Ang(v)$.
	Thus, $\Brtlt$ is also linear on the whole subdomain~$\Ang(v)$ by the minimality. 
	Therefore, if $x \in \Ang(v)$, then
	$$
		\Brtlt(x;\,u_1,v,u_2) = \Bang(x;\,v) =  \alpha_1 x_1 + \alpha_2 x_2 + \alpha_0.
	$$
	Calculating the values of $\Bang$ in the vertices of the angle $\Ang(v)$, we have 
	$$
		\left\{ 
		\begin{aligned}
			&\alpha_1 v + \alpha_2 v^2 + \alpha_0 = f(v);\\
			&\alpha_1(v - \eps) + \alpha_2\big((v-\eps)^2 + \eps^2\big) + \alpha_0 = -\mrt(v)\eps + f(v);\\
			&\alpha_1(v + \eps) + \alpha_2\big((v+\eps)^2 + \eps^2\big) + \alpha_0 = \mlt(v)\eps + f(v).
		\end{aligned} \right.
	$$
	Solving this system, we obtain
	\begin{equation}\label{e319}	
	\begin{split}
		\alpha_1 &= \frac{\mrt(v)+\mlt(v)}{2} - \frac{\mlt(v)-\mrt(v)}{2\eps}\,v;\\
		\alpha_2 &= \frac{\mlt(v)-\mrt(v)}{4\eps};\\
		\alpha_0 &= \frac{\mlt(v)-\mrt(v)}{4\eps}\,v^2 - \frac{\mrt(v)+\mlt(v)}{2}\,v + f(v).
	\end{split}
	\end{equation}

	Now we discuss the concavity of $\Brtlt$. As it has already been verified, the local concavity of
	$\Brt(x;\,u_1,v)$ and $\Blt(x;\,v,u_2)$ is equivalent, respectively, to 
	the inequalities
	\begin{equation*}
	\begin{split}
		&\mrt''(u)\le 0\quad\mbox{for}\quad u \in (u_1,v);\\
		&\mlt''(u)\ge 0\quad\mbox{for}\quad u \in (v,u_2). 
	\end{split}
	\end{equation*}
	Suppose these inequalities are fulfilled. We want to obtain some conditions on~$v$ that are
	necessary and sufficient for the concatenation of $\Brt$, $\Bang$, and $\Blt$ to be locally concave.	
	In order for the function~$\Brtlt$ to be concave along the direction $x_{2}$, its derivative $\Brtlt_{x_2}$ 
	must be	monotonically decreasing in $x_{2}$. Therefore, the jumps of $\Brtlt_{x_2}$
	on the boundary of $\Ang$ must be non-positive. They are
	$$
		\delta_{\mathrm{R}} = \alpha_2 - \lim_{u\to v-}t_2(u)\quad\mbox{and}\quad 
		\delta_{\mathrm{L}} = \alpha_2 - \lim_{u\to v+}t_2(u).
	$$
	Using~(\ref{e35}) and~(\ref{e36}), we obtain
	$$
		\lim_{u\to v-}t_2(u) = \frac{\mrt'(v)}{2} = \frac{f'(v) - \mrt(v)}{2\eps},		
	$$
	and due to~(\ref{e39}) and~(\ref{e323}), we have
	$$
		\lim_{u\to v+}t_2(u) = \frac{\mlt'(v)}{2} = \frac{\mlt(v) - f'(v)}{2\eps}.
	$$
	Now, using formula~(\ref{e319}) for~$\alpha_2$, we get the expressions for the jumps:
	\begin{align*}
		\delta_{\mathrm{R}}\ &=\phantom{-}\frac{1}{4\eps}(\mrt(v)+\mlt(v)-2f'(v));\\		
		\delta_{\mathrm{L}}\ &=-\frac{1}{4\eps}(\mrt(v)+\mlt(v)-2f'(v)).
	\end{align*}
	We see that their signs are always different. On the other hand, both jumps are non-positive and, therefore, 
	are equal to zero. Thus, the condition
	$$
	\mrt(v)+\mlt(v) = 2f'(v)
	$$
	is necessary for the function $\Brtlt$ to be locally concave. 
	Thus, if our concatenation is locally concave, then its derivative~$\Brtlt_{x_2}$ must be continuous.
	The partial derivatives of $\Brtlt$ along the tangents bounding $\Ang(v)$ are also continuous 
	(constant). Therefore, the function~$\Brtlt$ 
	has continuous derivatives along two non-collinear directions, so the derivatives along all the directions are continuous. 
	But a $C^1$-smooth concatenation
	of locally concave functions is locally concave. Hence, the condition $\mrt(v)+\mlt(v) = 2f'(v)$ is also sufficient for the 
	local concavity of the concatenation~$\Brtlt$ provided its components~$\Brt$, $\Bang$ and~$\Blt$ are locally concave. Finally, 
	by~(\ref{e35}) and~(\ref{e39}), the resulting condition is equivalent to the identity
	\begin{equation}\label{e3}
		\mrt''(v)+\mlt''(v) =0.
	\end{equation}

	We summarize this section.
	\begin{Prop}\label{S6}
		Let $u_1 < v < u_2$. Consider the subdomains $\Rt(u_1,v)$ and $\Lt(v,u_2)$ foliated by extremals~\textup{(R)} and~\textup{(L),} respectively.
		We also suppose that the domain~$\Ang(v)$
		lying between them is a domain of linearity. Then a Bellman candidate in the union $\RtLt(u_1,v,u_2)$ of these domains has the form 
		\begin{equation}\label{e320}
			\Brtlt(x;\,u_1,v,u_2) =
			\begin{cases}
				\Brt(x;\,u_1,v),& x \in \Rt(u_1,v);\\
				\Bang(x;\,v),& x \in \Ang(v);\\
				\Blt(x;\,v,u_2),& x \in \Lt(v,u_2),
			\end{cases}  		
		\end{equation}
		where $\Bang(x;\,v) = \alpha_1 x_1 + \alpha_2 x_2 + \alpha_0$, and the coefficients~$\alpha_1$\textup{,} $\alpha_2$\textup, and~$\alpha_0$ are 
		calculated by~\textup{(\ref{e319})}. In addition\textup, the following relations must be fulfilled\textup{:}
		\begin{equation*}
			\left\{
			\begin{aligned}
				&\mrt''(u) \le 0,\quad u\in(u_1,v);\\
				&\mlt''(u) \ge 0,\quad u\in(v,u_2);\\
				&\mrt''(v)+\mlt''(v) = 0.
			\end{aligned} \right.	
		\end{equation*}
	\end{Prop}

\subsection{Optimizers in angle}\label{s42}
	Now we construct optimizers for the points inside an angle.
	Suppose~$B$ is a Bellman candidate in~$\Omega_\eps$ and 
	some part of $\Omega_{\eps}$ is represented by the construction $\RtLt(u_1,v,u_2)$ described in Proposition~\ref{S6}.    
	We have already learned (see Section~\ref{s33}) how to build delivery curves and optimizers in $\Rt(u_1,v)$ and $\Lt(v,u_2)$. It turns out that 
	we need information from both right and left neighbors of the angle in order to obtain optimizers for its points. 
	To be more precise, we need two delivery curves already built: a left delivery
	curve~$\dc_-$ that reaches some point~$P^-$ on the right boundary of $\Rt(u_1,v)$, and the right delivery curve~$\dc_+$ that reaches some point~$P^+$ 
	on the left boundary of $\Lt(v,u_2)$. If we have optimizers in two points of $\Ang(v)$, then we can construct an optimizer for any points of 
	the segment that connects them provided this segment lies in $\Ang(v)$ entirely.
	
	Let $x\in\Ang(v)$. We draw some straight line~$L$ that passes through~$x$ and does not intersect the upper parabola. This line intersects 
	both sides of the angle. We denote 
	the points of intersection by~$P^\pm$. 
	Then $x$ will be a convex combination of $P^\pm$: $x=\alpha_-P^-+\alpha_+P^+$, where $\alpha_\pm \ge 0$ and
	$\alpha_-+\alpha_+=1$. 
	We build the optimizer~$\ex_-$ for~$P^-$ on $I_-=[0,\alpha_-]$  and the optimizer~$\ex_+$ for~$P^+$ on
	$I_+=[\alpha_-,1]$ (see Section~\ref{s33}). 
	Concatenating $\ex_-$ and~$\ex_+$, we obtain the function~$\ex$ on $[0,1]$. 
	It is easy to see that $\ex$ satisfies conditions~(2) 
	and~(3) of Definition~\ref{D3}. This follows immediately from the representation of~$x$ as a convex combination of~$P^{\pm}$
	and from the linearity of~$B$ in~$\Ang$:
	\begin{align*}
		x_k=\alpha_-P^-_k+\alpha_+P^+_k
		=\int\limits_0^{\alpha_-}\ex^k_-(s)\,ds+\int\limits_{\alpha_-}^1\ex^k_+(s)\,ds
		=\int\limits_0^1\ex^k(s)\,ds
		=&\av{\ex^k}{[0,1]}\\
		&k=1,2;
		\end{align*}
		\begin{align*}
		B(x)=\alpha_-B(P^-)+\alpha_+B(P^+)
		&=\int\limits_0^{\alpha_-}f(\ex_-(s))\,ds+\int\limits_{\alpha_-}^1 f(\ex_+(s))\,ds
		\\
		&=\int\limits_0^1 f(\ex(s))\,ds=\av{f(\ex)}{[0,1]}.
	\end{align*}
	
\begin{figure}[H]
\begin{center}
\includegraphics{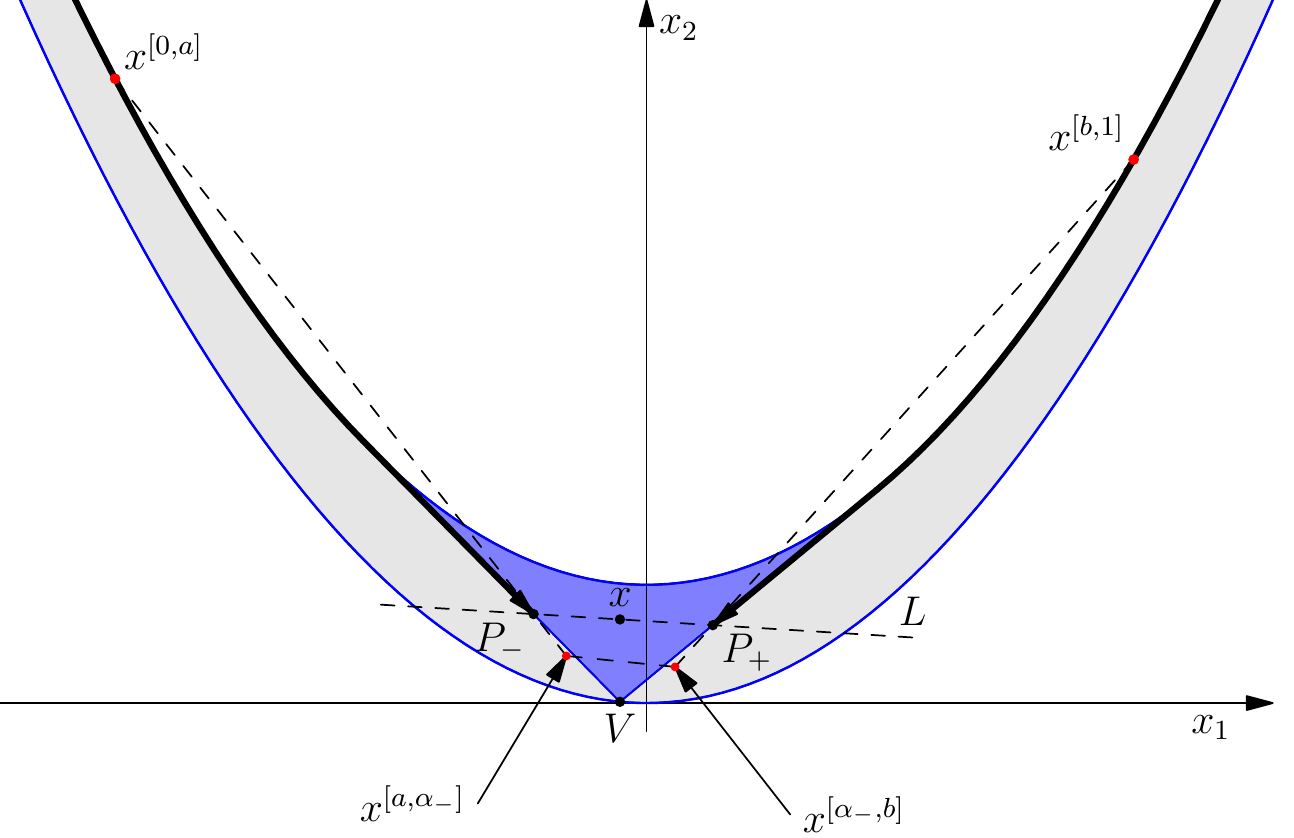}
\caption{Optimizers in $\Ang(v)$.}
\label{fig:ppoa}
\end{center}
\end{figure}	
	In order to prove that~$\ex$ is an optimizer for~$x$, it remains to verify that $\ex\in\BMO_\eps([0,1])$.
	Consider some subinterval $[a,b]\subset[0,1]$ and the Bellman point $x^{[a,b]} = \big(\av{\ex}{[a,b]},\av{\ex^2}{[a,b]}\big)$. 
	If $\alpha_-\notin (a,b)$, then~$x^{[a,b]}$ gets into~$\Omega_\eps$, because $\ex_{\pm}\in\BMO_\eps(I_\pm)$.
	Thus, we only need to consider the intervals $[a,b]$ such that $\alpha_-\in (a,b)$. Note that $P_{-}=x^{[0,\alpha_{-}]}$
	is a convex combination of $x^{[0,a]}$ and $x^{[a,\alpha_{-}]}$ and, therefore, lies on the segment connecting them.
	The point $x^{[0,a]}$ lies somewhere on the delivery curve coming from above and ending at $P_{-}$  
	(we already know how this curve is arranged: it is a convex curve that runs along the upper parabola and then 
	descend along the right tangent down to the point $P_{-}$). 
	Consequently, $x^{[0,a]}$ lies above~$L$, and so  $x^{[a,\alpha_{-}]}$ lies below~$L$.
	Similarly, we can verify that $x^{[\alpha_{-},b]}$ lies under~$L$. 	
	But the point~$x^{[a,b]}$ is a convex combination of~$x^{[a,\alpha_-]}$ 
	and $x^{[\alpha_-,b]}$. Therefore, it lies below~$L$ and, consequently, in~$\Omega_\eps$. 
	As a result, we have constructed optimizers~$\ex$ for all the points~$x$ in $\Ang(v)$.	

\subsection{Function \texorpdfstring{$f'''$}{f'''} changes its sign from minus to plus}\label{s43}
	Propositions~\ref{S5} and~\ref{S6}, together with the existence of optimizers in 
	the domains $\Rt(-\infty,v)$, $\Lt(v,+\infty)$, 
	and $\Ang(v)$, imply the following theorem. 
	\begin{Th}\label{T2}
		Let $0<\eps<\eps_0$ and $f \in \W$. Suppose there exists $v\in\mathbb{R}$ such that
		\begin{equation*}
		\begin{split}
			&\mrt''(u;\,-\infty) \le 0 \quad\mbox{for}\quad u\in(-\infty,v);\\
			&\mlt''(u;\,+\infty) \ge 0 \quad\mbox{for}\quad u\in(v,+\infty);\\
			&\mrt''(v;\,-\infty)+\mlt''(v;\,+\infty) = 0,
		\end{split}
		\end{equation*}
		where $\mrt''(u;\,-\infty)$ and $\mlt''(u;\,+\infty)$ are expressed by~\textup{(\ref{e315})} 
		and~\textup{(\ref{e318})}. Then
		$$
			\Bell(x;\,f) = \Brtlt(x;\,-\infty,v,+\infty),
		$$
		where the function on the right hand side is defined by~\textup{(\ref{e320}),} and its parts 
		$\Brt(x;\,-\infty,v)$ and $\Blt(x;\,v,+\infty)$ are defined 
		by~\textup{(\ref{e313})} and~\textup{(\ref{e316})}.
	\end{Th}
	It turns out that
	the conditions of Theorem~\ref{T2} can be satisfied if $f'''$ changes its sign from minus to plus.
	\begin{Th}\label{T2cor}
		Let $0<\eps<\eps_0$ and $f \in \W^{1}$ with $c_{0}=-\infty,$ $c_{1}=+\infty$.
		We denote
		$$
			g_\eps(u) \df (f'''\ast w_\eps)(u),
		$$
		where $w_\eps(t) = e^{-|t|/\eps}$. The function $g_{\eps}$ is continuous\textup, and 
		\begin{enumerate}
			\item[\textup{1)}] if $g_\eps < 0$ on  $\mathbb{R}$\textup, 
			then the conditions of Theorem~\textup{\ref{T1}} are satisfied\textup{;}
			\item[\textup{2)}] if $g_\eps > 0$ on  $\mathbb{R}$\textup, 
			then the conditions of Theorem~\textup{\ref{T1b}} are satisfied\textup{;}
			\item[\textup{3)}]  if $g_\eps(v) = 0$ for some $v\in\mathbb{R}$\textup, then 
			the conditions of Theorem~\textup{\ref{T2}} are satisfied.
		\end{enumerate}
	\end{Th}
	\begin{proof}
		First, we note that
		$$
		\eps^{-1}g_\eps(u) = \mrt''(u;\,-\infty) + \mlt''(u;\,+\infty),
		$$
 		where $\mrt''(u;\,-\infty)$ and $\mlt''(u;\,+\infty)$ are given  by~(\ref{e315}) and~(\ref{e318}).  		

 		Recall the point where~$f'''$ changes its sign was denoted $v_1$.
 		Consider case~1). 
 		It is clear that for $u \le v_1$ the inequality $\mrt''(u;\,-\infty)<0$ is always fulfilled, because $f'''(u)<0$ a.e. on $(-\infty,v_{1})$. 
 		On the other hand, for $u\ge v_1$ we use the condition $g_{\varepsilon}<0$:
		$$
			\mrt''(u;\,-\infty)=\varepsilon^{-1}g_{\varepsilon}- \mlt''(u;\,+\infty)< -\mlt''(u;\,+\infty).
		$$
		Since $f'''(u)>0$ a.e. for $u\geq v_1$, the inequality $-\mlt''(u;\,+\infty)<0$ is valid. Thus, we see that 
		$\mrt''(u;\,-\infty)<0$ for
		all $u\in\mathbb{R}$, and the conditions of Theorem~\ref{T1} are fulfilled.
		Case~2) can be treated similarly.
		
		Finally, we consider case 3). 
		We may treat only the case  $v\ge v_1$ (the case  $v\le v_1$ can be treated similarly). 
		The sign of $f'''$ is known, and so $\mlt''(u;\,+\infty)>0$ for $u \geq v$ 
		(this is one of the conditions of Theorem~\ref{T2}). We also know that $\mrt''(u;\,-\infty)<0$ for $u \le v_1$. 	
		Thus, it remains to verify
		that $\mrt''(u;\,-\infty)<0$ for $ u\in (v_1,v)$.		
		On the one hand, we have
		$$
	  	\mrt''(v;\,-\infty)=\varepsilon^{-1}g_{\varepsilon}(v)- \mlt''(v;\,+\infty)= -\mlt''(v;\,+\infty)<0.
		$$
		On the other hand, the function $e^{u/\eps}\,\mrt''(u;\,-\infty)$ increases monotonically 
		on  $(v_1,v)$, because
		$$
			e^{u/\eps}\,\mrt''(u;\,-\infty) = \varepsilon^{-1}\int\limits_{-\infty}^{u}f'''(t)e^{t/\varepsilon}\, dt
		$$
		and $f'''$ is positive on this interval. 
		Consequently, $e^{u/\eps}\,\mrt''(u;\,-\infty)$ is negative for all ${u\in(v_1,v)}$. 
		As a result, all the conditions of Theorem~\ref{T2} are fulfilled.
	\end{proof}

\subsection{Examples}
\paragraph{Example 3. The power function.}
	The function $f(t) = |t|^{p}$ was treated in~\cite{SlVa2}. For $p>2$, it gets into the class being considered: 
	$f \in \W^{1}$ with $c_{0}=-\infty$, $c_{1} = +\infty$. Here, we do not write an explicit expression for the Bellman function, but merely verify that
	the conditions of case~3) in Theorem~\ref{T2cor} are satisfied.	
	Indeed, the expression  
	\begin{align*}
	\eps^{-1} g_{\eps}(u) &= \mrt''(u;\,-\infty) + \mlt''(u;\,+\infty)\\
	&=\eps^{-1}\int\limits_{-\infty}^{\infty}\sign t \cdot p(p-1)(p-2)|t|^{p-3}e^{-|u-t|/\eps}\,dt
	\end{align*}
	has the unique root $u = 0$.
	Therefore, the vertex of the angle has coordinates $(0,0)$ for any $\eps \in (0,\infty)$, i.e. it does not
	depend on~$\eps$.

\paragraph{Example 4. The concatenation of two exponential functions.}
	We consider a certain family of functions that depend on a parameter.
  The third derivative of each of these functions changes its sign once from minus to plus. 
  For certain values of our parameter the domain will be foliated entirely by the tangents of the same type, and 
  for other values an angle will arise.

	Namely, we consider a function~$f$ such that its third derivative is given as follows:
	\begin{equation}\label{RR2}
	 f'''(t)=
	\begin{cases}
		\phantom{-} e^{t}, &t \geq 0;\\
	 -e^{ t/\alpha}, &t < 0.
	\end{cases}
	\end{equation}
	For example, we may set
	\begin{equation}\label{bc}
	 f(t)=
	\begin{cases}
		\qquad e^{t},  &t \geq 0;\\
	    -e^{t/\alpha}\alpha^{3}+\frac{t^{2}}{2}(1+\alpha)+t(1+\alpha^{2})+1+\alpha^{3}, &t < 0.
	\end{cases}
	\end{equation}
	For any positive~$\alpha$, this function belongs to $\Wone^{1}$ with $c_{0}=-\infty$, $c_{1} = +\infty$.
	Now we want to find all~$\alpha$ such that the condition of Theorem~\ref{T1b} is satisfied, i.e.
	\begin{equation*}
		\int\limits_{u}^{\infty}f'''(t)e^{-t/\eps}\;dt \geq 0 \quad \text{for}  \quad u  \in \mathbb{R}.
	\end{equation*}
	For $u <0$, we have
	\begin{equation*}
		\int\limits_{u}^{\infty}f'''(t)e^{-t/\eps}\,dt = 	\frac{\eps}{1-\eps}+\frac{\alpha\eps}{\alpha-\eps}
		-\frac{\alpha\eps}{\alpha-\eps}\exp\left(-\frac{u(\alpha-\eps)}{\alpha\eps} \right).
		\end{equation*}
	This expression is non-negative for all $u<0$ if and only if $\alpha < \eps$ and
	$$
		\frac{\eps}{1-\eps}+\frac{\alpha \eps}{\alpha-\eps}\geq 0. 
	$$
	Thus, the condition of Theorem~\ref{T1b} is satisfied when 
	\begin{equation}\label{usla}
		0 < \alpha \leq \frac{\eps}{2-\eps}.
	\end{equation}
	Therefore, in the case where the boundary values are defined by~(\ref{bc}), 
	condition~(\ref{usla}) is necessary and sufficient for 
	$\Omega_\eps$ to be foliated by the left tangents.
	Thus, for such values~$\alpha$, the Bellman function can be easily restored:
	$$
		\Bell(x_{1},x_{2};\,f)=(x_{1}-u)\,\mlt(u;\,+\infty)+f(u),
	$$
	where
	\begin{multline*}
	\mlt(u;\,+\infty) =\eps^{-1}e^{u/\eps}\int\limits_{u}^{\infty}f'(t)e^{-t/\eps}\,dt\\
	=\begin{cases}
	\displaystyle
	e^{u}\frac{1}{1-\eps},   &u \geq 0;\\
	\displaystyle \rule{0pt}{30pt} e^{u/\eps}\frac{\eps^3\!+\!\eps^3\alpha\!-\!2\alpha\eps^2}{(1\!-\!\eps)(\eps\!-\!\alpha)}+e^{u/\alpha}\frac{\alpha^3}{\eps-\alpha}+(u+\eps)(1+\alpha)+(1+\alpha^2),   &u < 0.\\
	\end{cases}
	\end{multline*}
	Also, we recall (see~\ref{e33}) that
	$$
		u  =  x_{1}-\eps+\sqrt{x_{1}^{2}-x_{2}+\eps^{2}}.
	$$	

	We note that for $\alpha$ considered above, case~2) in Theorem~\ref{T2cor} occurs.
	Now we verify that for
	\begin{equation}\label{uslaa}
		\alpha > \frac{\eps}{2-\eps}
	\end{equation}
	case~3) in this theorem comes into play.
	Indeed, if condition~(\ref{uslaa}) is fulfilled, the equation
	\begin{align*}
		\eps^{-1} g_{\eps}(u)= \mrt''(u;\,-\infty) + \mlt''(u;\,+\infty)=0
	\end{align*}
	has the unique root
	\begin{equation*}
		u =
		\begin{cases}
		\displaystyle 
		\frac{\alpha \eps}{\alpha-\eps} \log\left( \frac{2\alpha^{2}(1-\eps)}{(\alpha+\eps)(2\alpha-\alpha\eps-\eps)}\right),&\alpha \neq \eps;\\
		\displaystyle \rule{0pt}{27pt} -\frac{\eps(\eps+1)}{2(1-\eps)},&  \alpha =\eps.
		\end{cases}
	\end{equation*}

\paragraph{Example 5. A fourth-degree polynomial.}
	It is clear that any fourth-degree polynomial belongs to $\W^1$
for all $\eps_0 >0$, $c_0=-\infty$, $c_1=+\infty$, if the leading coefficient
is positive.
	According to Remark~\ref{rem2}, it is sufficient to consider polynomials of the form
	$f = \frac{1}{24}t^{4}-\frac{a}{6}t^{3}$, $a \in \mathbb{R}$.
	In such a case, $f'''(t) = t-a$.

	We do not write an explicit expression for the Bellman function, but only verify that condition~3) in 
	Theorem~\ref{T2cor} is fulfilled and look for the vertex of the angle. The expression
	\begin{equation*}
	\mrt''(u;\,-\infty) + \mlt''(u;\,+\infty) = \eps^{-1}\int\limits_{-\infty}^{\infty}(t-a)e^{-|u-t|/\eps}dt=2(u-a)
	\end{equation*}
	has the unique root $u=a$.
	Thus, for any $\eps \in (0,\infty)$, the vertex of the angle has the coordinates $(a,a^{2})$. 
	We note that the coordinates of the vertex do not depend on~$\eps$.

\paragraph{Example 6. The angle moves when \texorpdfstring{$\eps$}{epsilon} varies.}
Now we consider a more interesting case where the vertex of an angle varies depending on $\eps$.
Let $f$ be a $C^{3}$-smooth function such that
\begin{equation*}
f'''(t)=\begin{cases}
-t^{2},& t\leq 0;\\
\phantom{-}t\phantom{^2},& t>0.
\end{cases}
\end{equation*}
Then $f \in \W^{1}$ for any $\eps_{0}>0$, $c_{0}=-\infty$, $c_{1}=+\infty$.
We want to check condition~3) in Theorem~\ref{T2cor}:
\begin{align}
\mrt''(v;\,-\infty) &+ \mlt''(v;\,+\infty) \nonumber\\
&\hskip-30pt=\eps^{-1}\int\limits_{-\infty}^{v}\big(-t^{2}\chi_{(-\infty,0)}(t)+t\chi_{(0,\infty)}(t)\big)e^{(t-v)/\eps}\,dt \nonumber \\
&+\eps^{-1}\int\limits_{v}^{\infty}\big(-t^{2}\chi_{(-\infty,0)}(t)+t\chi_{(0,\infty)}(t)\big)e^{(v-t)/\eps}\,dt \nonumber \\
&\hskip-30pt=
\begin{cases}
-4\eps^{2}-2 v^{2}+2\eps^{2}e^{v/\eps}+\eps e^{v/\eps},& v< 0; \label{neq_v}\\
-2\eps^{2}e^{-v/\eps}+\eps e^{-v/\eps}+2 v,&v\ge0.
\end{cases} 
\end{align}
If $v \geq 0$, we can write the equation for the vertex of the angle as
\begin{equation}\label{lambert1}
\frac{v}{\eps} e^{\frac{v}{\eps}}=\left( \eps - \frac{1}{2}\right).
\end{equation}
It is clear that this equation has no positive solutions for ${\eps<1/2}$.
Therefore, we consider the case ${\eps\geq 1/2}$. 
We see that the solution of equation~(\ref{lambert1}) is the function
$v(\eps) = \eps W(\eps-1/2)$, where $W(z)$ is the Lambert function. 
The Lambert function is defined by the equation ${W(z)e^{W(z)}=z}$. 
It is clear that $v(1/2) =0$. Thus, since $v \geq 0$ and $W' >0$, it follows that $v'(\eps) >0$ for $\eps \geq 1/2$.
Therefore, for $\eps \geq 1/2$, condition~3) in Theorem~\ref{T2cor} is fulfilled, and
the vertex $v(\eps)$ of the angle moves to the right when $\eps$ grows.  

We claim that for $0 <\eps < 1/2$ the equation for the vertex of the angle has a negative solution.
We equate expression~(\ref{neq_v}) for $v<0$ to zero:
\begin{equation*}
e^{v/\eps}(2\eps^{2}+\eps)-4\eps^{2}-2v^{2}=0.
\end{equation*}
Note that the left hand side of this equation increases monotonically from $-\infty$ to the positive
number $\eps-2\eps^{2}$ as $v$ runs from $-\infty$ to $0$. Consequently, this equation has a unique root $v(\eps)$,
i.e. condition~3) of Theorem~\ref{T2cor} is fulfilled. 
It is easy to see that $v(\eps) \to 0$ as $\eps \to 0$ or $\eps \to 1/2$.
Besides, we can find a number $\tilde\eps$, $0<\tilde\eps<1/2$, with the following properties: 
if $\eps$ decreases from $1/2$ to $\tilde\eps$, then
the vertex of the angle moves from zero to a certain value~$\tilde v$, and if $\eps$ decreases from $\tilde\eps$
to zero, then the vertex returns from $\tilde v$ to zero.

\section{Transition from left tangents to right ones}
	In this chapter, we consider a transition from left tangents to right ones. Such a transition
	is performed through a subdomain foliated by extremal chords whose endpoints lie on the lower 
	parabola. The reader can look at Figure~\ref{fig:plkpl} to understand what is meant. 
	In Section~\ref{s51}, we will describe the form that any Bellman candidate must have in a subdomain foliated
	by extremal chords (see Figure~\ref{fig:ox})
	and also derive some conditions that these chords must satisfy.
	In Section~\ref{s52}, we will construct a Bellman candidate in 
	the domain shown in Figure~\ref{fig:plkpl}. 
	It turns out that in the case where domains~$\Lt$ and~$\Rt$ border on a domain foliated by chords, 
	the corresponding candidates~$\Blt$ and~$\Brt$ can be determined uniquely (i.e. the integration constant
	can be calculated explicitly).
	In Section~\ref{s53}, we will build delivery curves and optimizers in domains foliated by chords.
	As we will see, such a domain is another place (besides $\pm\infty$) where delivery curves can originate.
	Finally, in Section~\ref{s54}, we will prove that if $f\in\W^{0}$ and $c_{0}\neq \pm \infty$ 
	(i.e. $f'''$ changes its sign once, from plus to minus), then the Bellman function
	$\Bell(x;\,f)$ corresponds to the foliation described above.
	
	Before continuing, we recall our agreement on the notation. If a point on the lower boundary is denoted by
	a capital Latin letter, then the corresponding small letter denotes the first coordinate of this point 
	(and vice versa). Throughout this chapter, we use this rule very often.

\subsection{Family of chords}\label{s51}
	Let $A_0$, $A_1$, $B_1$ and $B_0$ be four points on the lower boundary of $\Omega_\eps$ with
	the abscissas $a_0$, $a_1$, $b_1$ and $b_0$ such that ${a_0<a_1< b_1<b_0}$ and ${b_0-a_0\leq 2\eps}$.
	We draw two segments $[A_0,B_0]$ and $[A_1,B_1]$. 
	It is easy to see that both of these segments lie in $\Omega_\eps$ entirely. 
	We consider the subdomain bounded by these segments and two arcs of the lower parabola: the one connects $A_0$ and $A_1$ 
	and the other connects $B_1$ and $B_0$. Suppose this subdomain is foliated entirely by a family of 
	non-intersecting chords with the endpoints
	lying on the different arcs of the lower parabola. We denote such a subdomain by $\Ch(a_0,b_0,a_1,b_1)$ 
	(see Figure~\ref{fig:ox}).
\begin{figure}[H]
\begin{center}
\includegraphics{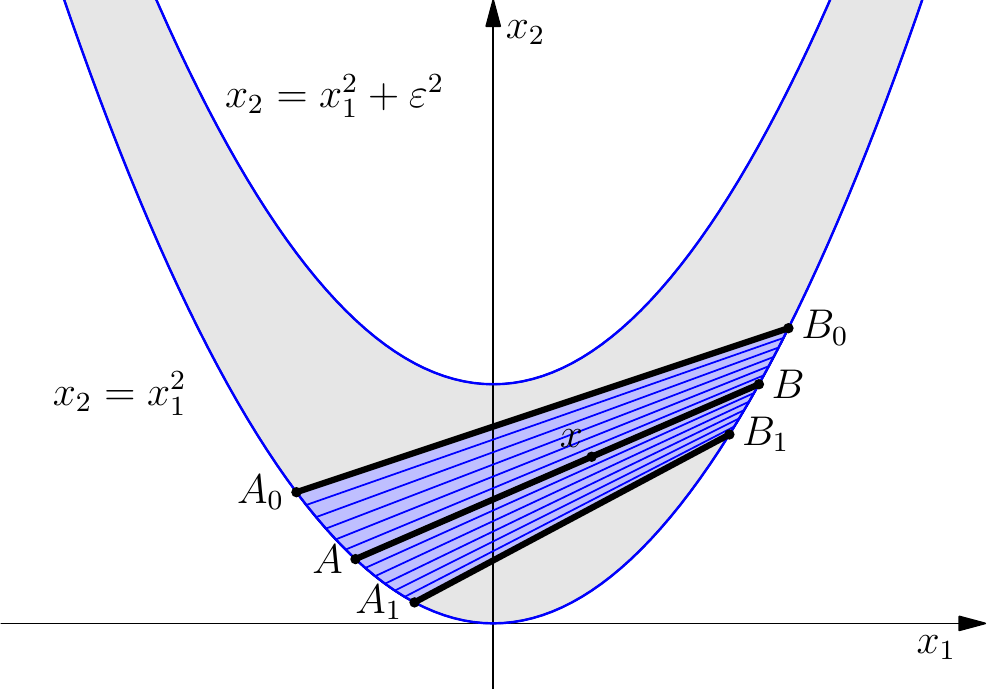}
\caption{A domain $\Ch$ with the chords.}
\label{fig:ox}
\end{center}
\end{figure}

	We see that for any point~$x$ in $\Ch(a_0,b_0,a_1,b_1)$ there are two numbers $a \in [a_0,a_1]$ and $b\in[b_1,b_0]$ such that
	the chord $[A,B]$ belongs to our family and contains~$x$.
	We want to construct a Bellman candidate in $\Ch(a_0,b_0,a_1,b_1)$ whose 
	partial derivatives are constant along the chords in our family. What is more, we derive some conditions on the chords that allow such a candidate to 
	exist at all. We denote the function required by $\Bch(x;\,a_0,b_0,a_1,b_1)$ (sometimes we write $\Bch(x)$ for short).
		
	First, we note that the principal difference between the cases of extremal chords and extremal tangents lies in the fact that 
	using the linearity along the chords, we can restore $\Bch$ in~$\Ch$ uniquely.
	Indeed, if we know that $\Bch(A)=f(a)$, $\Bch(B)=f(b)$, and $\Bch$ is linear along the chord $[A,B]$, then
	we can calculate the value of~$\Bch$ at any point~$x$ lying on this chord:
	\begin{equation}\label{e325}
		\Bch(x)=\frac{f(b)-f(a)}{b-a}x_1+\frac{bf(a)-af(b)}{b-a}.
	\end{equation}	
	However, the function $\Bch$ built in this way is a Bellman candidate only if its derivatives~$\Bch_{x_1}$ and~$\Bch_{x_2}$ are constant along the 
	extremals. We will get some condition on the chords that guarantees the constancy of~$\Bch_{x_1}$ and~$\Bch_{x_2}$ on them.

	We parametrize our chords $[A,B]$ by the values $\ell = b - a$. Then the right endpoint $B(\ell)$ moves to the right, i.e, 
	the function $b(\ell)$ increases. The left endpoint $A(\ell)$ moves to the left at the same time, i.e. the function $a(\ell)$ decreases.
	In addition, we assume that the functions $a$ and $b$ are differentiable and the inequalities $a'<0$ and $b'>0$ are fulfilled.
	The last requirement implies that the chords $[A,B]$ do not intersect. 	
	Domains foliated by extremal chords that share a common point on the boundary, can arise 
	if the boundary function~$f$ is not smooth enough (see~\cite{SlVa2} or~\cite{Va}). We will not encounter such domains due to our 
	assumptions on the smoothness of~$f$.
	
	In its turn, $\ell$ can be treated as a function of $x\in\Ch$, i.e. we consider the function $\ell(x)$. 
	For short, we often omit the arguments of the functions~$a$, $b$ and~$\ell$. We write the equation of the line
	passing through points $A$ and $B$:
	$$
		x_2=(a+b)x_1-ab.
	$$
	Now we calculate $\ell_{x_2}$. By the last relation, if $x_1$ is fixed, then $x_2$ is a differentiable function of $\ell$ with
	$$
		x_2' = x_1(a'+b')-(ab'+ba').
	$$
	But $x_1$ takes values from $a$ to $b$. Therefore, $x_2'$ runs between ${(a-b)a'}$ and ${(b-a)b'}$. 
	Each of this two values is greater than zero, and so
	$x_2'(\ell)>0$. Consequently, the inverse function~$\ell$ is differentiable in $x_2$, and
	\begin{equation}\label{e324}
		\ell_{x_2}=\frac{1}{x_1(a'+b')-(ab'+ba')}.
	\end{equation}
	
	We are ready to calculate the partial derivatives of~$\Bch$. As we have already mentioned, we are searching for a condition on the chords
	under which $\Bch_{x_1}$ and~$\Bch_{x_2}$ are constant along them. Since $\Bch$ is linear along the chords, it is sufficient to obtain
	a condition that guarantees the constancy of $\Bch_{x_2}$ along them.
	Differentiating identity~(\ref{e325}) in $x_2$, we get
	\begin{equation}\label{e333}
		\Bch_{x_2}(x_1,x_2) = \frac{\alpha x_1 + \beta}{(b-a)^2}\,\ell_{x_2},
	\end{equation}
	where
	\begin{align*}
		\alpha = &\big(f'(b)b'-f'(a)a'\big)(b-a)-\big(f(b)-f(a)\big)(b'-a');\\
		\beta = &\big(b'f(a)+bf'(a)a'-a'f(b)-af'(b)b'\big)(b-a)\\
		&-\big(bf(a)-af(b)\big)(b'-a'),
	\end{align*}
	and $\ell_{x_2}$ is given by~(\ref{e324}). 	
	Since $\Bch_{x_2}$ is constant along the chords, it does not depend on $x_1$ if $\ell$ is fixed. 
	But if the quotient of two linear functions does not depend
	on the variable, then their coefficients must be proportional, i.e. 
	$$
		{\alpha}(ab'+ba') = -{\beta}(a'+b').
	$$	
	Substituting the corresponding expressions for $\alpha$ and $\beta$, we obtain, after elementary calculations, the equivalent identity:
	\begin{equation*}
		a'b'\bigg(\frac{f'(a)+f'(b)}{2} - \frac{f(b)-f(a)}{b-a}\bigg)=0.
	\end{equation*}
	Dividing by $a'b'$, we have
	\begin{equation}
		\label{urlun}
		\av{f'}{[a,b]}=\frac{f'(a)+f'(b)}{2}.
	\end{equation}
	Thus, under the assumption $a'b'\ne 0$, the derivatives of~$\Bch$ are constant on the chords $[A,B]$ if and only if their ends 
	satisfy equation~(\ref{urlun}).

	Now we turn to the concavity of the function $\Bch$ constructed above. 	
	We note that at each point of~$\Ch$, our function is linear in one direction. Therefore, as in the case
	of extremal tangents discussed in the previous chapter, it is sufficient to verify the concavity along some other direction. 
	Since the 
	direction~$x_2$ always differs from the direction of chords, it is enough to study the sign of~$\Bch_{x_2x_2}$. 
	First, using~(\ref{urlun}), we simplify formula~(\ref{e333})
	for~$\Bch_{x_2}$. Since the expression for~$\Bch_{x_2}$ does not depend on $x_1$, we have
	\begin{align*}
		\Bch_{x_2}(x_1,x_2) 
		&=\frac{\big(f'(b)b'-f'(a)a'\big)(b-a)-\big(f(b)-f(a)\big)(b'-a')}{(a'+b')(b-a)^2}\\
		&=\frac{2f'(b)b'-2f'(a)a'-\big(f'(b)+f'(a)\big)(b'-a')}{2(a'+b')(b-a)}\\
		&=\frac{f'(b)-f'(a)}{2(b-a)}.
	\end{align*}
	Since $\ell$ strictly increases as $x_2$ grows (this is obvious by the geometric considerations, but the formal proof can be found in 
	the derivation of~(\ref{e324})), it is sufficient to study the sign of~$\Bch_{x_2 \ell}$. By direct calculations, we have
	\begin{equation}\label{e7}
	\begin{aligned}
		2\Bch_{x_2 \ell}
		&=\frac{f''(b)b'-f''(a)a'}{b-a}-\frac{f'(b)-f'(a)}{(b-a)^2}(b'-a')\\
		&=\frac{b'\big(f''(b)-\av{f''}{[a,b]}\big)-a'\big(f''(a)-\av{f''}{[a,b]}\big)}{b-a}\, .
	\end{aligned}
	\end{equation}
	On the other hand, differentiating equation~(\ref{urlun}) with respect to~$\ell$, we get
	\begin{equation}\label{e6}
		b'\big(f''(b)-\av{f''}{[a,b]}\big)+a'\big(f''(a)-\av{f''}{[a,b]}\big)=0\,.
	\end{equation}
	We introduce the following notation:
	\begin{equation}\label{e334}
		\Dlt(a,b) = f''(a)-\av{f''}{[a,b]}\quad\textrm{and}\quad
		\Drt(a,b) = f''(b)-\av{f''}{[a,b]}.
	\end{equation}
	Equation~(\ref{e6}), together with the inequalities $b'> 0$ and $a'< 0$, implies that $\Dlt(a,b)$ and $\Drt(a,b)$ have the same sign for 
	every chord $[A,B]$.
	Thus, by virtue of~(\ref{e7}), we see that $\Bch_{x_2 \ell}\le 0$ if and only if either $\Dlt(a,b)\le 0$ or $\Drt(a,b)\le 0$. What is more, each of 
	these two inequalities implies the other.
	
	We summarize this section in the following proposition.
	\begin{Prop}\label{S8}
		Consider a domain $\Ch(a_0,b_0,a_1,b_1)$ foliated entirely by non-intersecting chords $[A,B]$\textup, and
		parametrize the first coordinates $a$ and $b$ of their endpoints
		by $\ell = b-a$. Suppose $a$ and $b$ are differentiable functions such that $a'<0$ and $b'>0$.
		Under these assumptions\textup, we can build a function $\Bch(x;\,a_0,b_0,a_1,b_1)$ such that its partial derivatives are constant 
		along the chords $[A,B]$\textup,
		if and only if all the chords satisfy~\textup{(\ref{urlun})}. The function~$\Bch$ can be calculated by~\textup{(\ref{e325})}. 
		Also\textup, we have
		\begin{equation}\label{e326}
			\Bch_{x_2}(x) = \frac{f'(b)-f'(a)}{2(b-a)} = \frac{1}{2}\av{f''}{[a,b]},
		\end{equation}
		where $a$ and $b$ are the first coordinates of the endpoints of the chord $[A,B]$ passing through $x$.
		
		The function~$\Bch$ is locally concave \textup{(}and\textup, therefore\textup, it is a Bellman candidate\textup{)} 
		if and only if for every chord $[A,B]$ 
		one of the following two inequalities is fulfilled\textup:
		\begin{equation}\label{eq13}
			\Dlt(a,b)\le 0\quad\mbox{or}\quad\Drt(a,b)\le 0.
		\end{equation}
		Furthermore\textup, each of these two inequalities implies the other one.
	\end{Prop}

\subsection{Cup}\label{s52}	
	In the previous section, we dealt with subdomains $\Ch(a_0,b_0,a_1,b_1)$ lying between two chords $[A_0,B_0]$ and $[A_1,B_1]$ in $\Omega_\eps$. 
	Now we consider a subdomain arising in the case $a_1=b_1$.
	\begin{Def}\label{D10}
		Let $0\le b_0 - a_0 \le 2\eps$. Consider the subdomain of~$\Omega_\eps$ that lies between $[A_0,B_0]$ and the lower parabola. 
		Suppose there exists a family of non-intersecting 
		chords that foliate this subdomain entirely and have the following properties:
		\begin{enumerate}
			\item[1)]
				if we parametrize the first coordinates $a$ and $b$ of their endpoints by $\ell = b-a$, we obtain the differentiable functions 
				$a(\ell)$ and $b(\ell)$
				such that $a'<0$ and $b'>0$;
			\item[2)]
				each of these chords satisfies equation~(\ref{urlun});
			\item[3)]
				for each chord, one of two inequalities~(\ref{eq13}) is fulfilled.
		\end{enumerate}
		In such a situation, we call the subdomain being considered \emph{a cup} and denote it by $\Alv(a_0,b_0)$.
	\end{Def}
	The unique point~$c$ lying in the intersection of all the intervals $[a,b]$ is called \emph{the origin} of the cup. 
	The points $a_0$ and $b_0$ are called \emph{the ends} of the cup, and the value $\ell_0 = b_0-a_0$ is  called \emph{the size} of the cup.
	Note that if $\ell_0 = 2\eps$, the chord $[A_0,B_0]$ touches the upper parabola. In such a case, we say that the cup $\Alv(a_0,b_0)$ is \emph{full}. 	
	Also, the case $\ell_0 = 0$ is not excluded from the consideration. In this situation, the cup consists of the single point $(c,c^2)$.
	
	Using~(\ref{e325}), we construct a function in $\Alv(a_0,b_0)$ that is linear along the chords $[A,B]$. 
	Proposition~\ref{S8} implies that such a function is a Bellman candidate in the cup. We denote it by $\Balv(x;\,a_0,b_0)$.
	
	Now we assume that $u_1<a_0<b_0<u_2$ and $b_0-a_0=2\eps$. Consider a full cup $\Alv(a_0,b_0)$ together with two domains 
	$\Lt(u_1,a_0)$ and $\Rt(b_0,u_2)$ adjacent to the cup and foliated
	by extremals~(L) and~(R), respectively (see Figure~\ref{fig:plkpl}).
	
		Consider the union
	$$
		\LtRt(u_1,[a_0,b_0],u_2) \df \Lt(u_1,a_0) \cup \Alv(a_0,b_0) \cup \Rt(b_0,u_2).
	$$
	In this domain, we are looking for a function such that
	its partial derivatives are constant along the chords in~$\Alv$ and, respectively, along the corresponding tangents 
	in~$\Rt$ and~$\Lt$. Denote the function being sought by $\Bltrt(x;\,u_1,[a_0,b_0],u_2)$. In~$\Alv$ 
	it must coincide with~$\Balv$.
	Concerning the subdomains
	$\Lt$ and $\Rt$, the corresponding functions~$\Blt$ and~$\Brt$ are calculated by formulas~(\ref{e322}) and~(\ref{e321}), where 
	the functions~$\mlt$ and~$\mrt$ are not defined uniquely: we have the freedom to choose the values $\mlt(a_0)$ and $\mrt(b_0)$ 
	(see~(\ref{e310}) and~(\ref{e37})). But in the situation being considered, there is the only way to choose $\mlt(a_0)$ and $\mrt(b_0)$
	so that the corresponding functions~$\Blt$ and~$\Brt$ glue with~$\Balv$ continuously.
	Indeed, on the chord with ends~$a_0$ and $b_0$, the function~$\Balv$ can be calculated by the formula
	$$
		\Balv\big(x_1,(a_0+b_0)x_1-a_0b_0\big) = \frac{f(b_0)-f(a_0)}{b_0-a_0}(x_1-a_0)+f(a_0).
	$$
	
\begin{figure}[H]
\begin{center}
\includegraphics{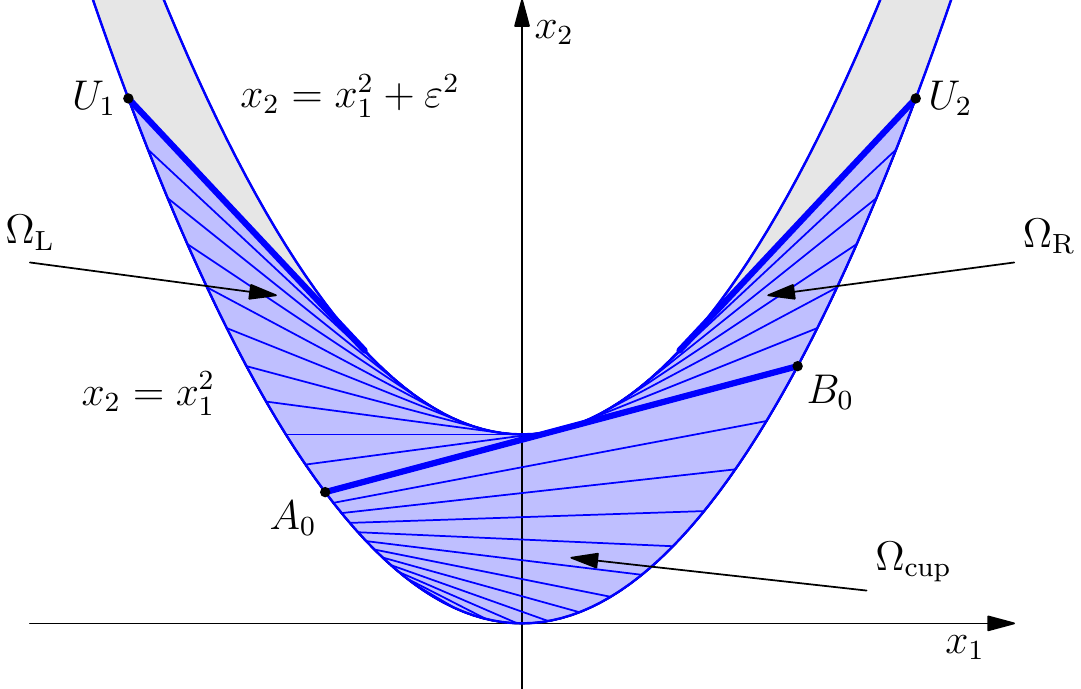}
\caption{A cup $\Alv$ lying between $\Lt$ and $\Rt$.}
\label{fig:plkpl}
\end{center}
\end{figure}
		
	On the other hand, by~(\ref{e322}) the limit values of~$\Blt$ on this chord are equal to $\mlt(a_0)(x_1-a_0)+f(a_0)$. 
	Therefore, the identity
	$$
		\mlt(a_0) = \frac{f(b_0)-f(a_0)}{b_0-a_0} = \Av{f'}{[a_0,b_0]}
	$$
	is necessary and sufficient for the concatenation of~$\Blt$ and~$\Balv$ to be continuous. 	
	Using chord equation~(\ref{urlun}), we can rewrite the equation obtained above as
	\begin{equation}\label{e332}
		\mlt(a_0) = \frac{f'(a_0)+f'(b_0)}{2}.
	\end{equation}
	By $\mlt(u;\,a_0)$ denote the coefficient $\mlt(u)$ satisfying this condition. Using~(\ref{e310}), we get
	\begin{equation}\label{e328}
		\mlt(u;\,a_0) 
		= \frac{f'(a_0) + f'(b_0)}{2}e^{(u-a_0)/\eps} + 
		\eps^{-1}e^{u/\eps}\int\limits_{u}^{a_0} f'(t)e^{-t/\eps}\,dt.
	\end{equation}
	Thus, in $\Lt(u_1,a_0)$, the function $\Bltrt(x;\,u_1,[a_0,b_0],u_2)$ coincides with the function
	\begin{equation}\label{e329}
		\Blt(x;\,u_1,[a_0,b_0]) \df \mlt(u;\,a_0)\, (x_1 - u) + f(u),
	\end{equation}
	where $u = u_{\mathrm{L}}(x_1,x_2)$ can be calculated by~(\ref{e33}).
	
	Using similar considerations, we see that the concatenation of $\Balv$ and $\Brt$ is continuous if and only if $\mrt(u) = \mrt(u;\,b_0)$, where
	\begin{equation}\label{e330}
		\mrt(u;\,b_0)
		= \frac{f'(a_0)+f'(b_0)}{2}e^{(b_0-u)/\eps} + \eps^{-1}e^{-u/\eps}\int\limits_{b_0}^u f'(t)e^{t/\eps}\,dt.
	\end{equation}
	This means that in $\Rt(b_0,u_2)$ the function $\Bltrt(x;\,u_1,[a_0,b_0],u_2)$ being sought must coincide with the function
	\begin{equation}\label{e331}
		\Brt(x;\,[a_0,b_0],u_2) \df \mrt(u;\,b_0)\, (x_1 - u) + f(u),
	\end{equation}
	where $u = u_{\mathrm{R}}(x_1,x_2)$ can be calculated by~(\ref{e32}).
	
	Before discussing the local concavity of the function $\Bltrt(x;\,u_1,[a_0,b_0],u_2)$ constructed above, 
	we show that $\Bltrt$ is not only continuous, but also 
	$C^1$-smooth. 	
	Let $t_2 = \Bltrt_{x_2}$. We treat $t_2$ as a function of~$u$ in $\Lt(u_1,a_0)$, and as a function of $a$~--- 
	the left ends of the extremal chords~--- in $\Alv(a_0,b_0)$.
	Using~(\ref{e326}), we obtain
	$$
		t_2(a_0) = \frac{f'(b_0)-f'(a_0)}{2(b_0-a_0)}.
	$$
	On the other hand, by~(\ref{e39}), (\ref{e323}), and (\ref{e332}), we have
	$$
		\lim_{u\to a_0-}t_2(u) = \frac{\mlt'(a_0;\,a_0)}{2}
		= \frac{\mlt(a_0;\,a_0) - f'(a_0)}{2\eps} =
		\frac{f'(b_0)-f'(a_0)}{2(b_0-a_0)}.
	$$
	Thus, the function $\Bltrt_{x_2}$ is continuous at the junction of~$\Lt$ and $\Alv$. Similarly, we can prove its continuity
	at the junction of~$\Alv$ and~$\Rt$. But the derivative of~$\Bltrt$ in the direction of the chord $[A_0,B_0]$ is also continuous (constant), i.e.
	on the chord just mentioned, the function~$\Bltrt$ has continuous derivatives in two non-collinear directions. Thus,
	the function~$\Bltrt$ turns out to be $C^1$-smooth. This implies that it is locally concave provided its components
	$\Blt(x;\,u_1,[a_0,b_0])$, $\Balv(x;\,a_0,b_0)$, and $\Brt(x;\,[a_0,b_0],u_2)$ are locally concave. 
	As mentioned above, the function~$\Balv$ is concave by the definition of a cup and 
	Proposition~\ref{S8}. Concerning the functions~$\Blt$ and~$\Brt$, they are locally concave if and only if
	the following inequalities are fulfilled:
	\begin{equation*}
	\begin{split}
		&\mlt''(u;\,a_0)\ge 0\quad\mbox{for}\quad u \in (u_1,a_0);\\
		&\mrt''(u;\,b_0)\le 0\quad\mbox{for}\quad u \in (b_0,u_2).
	\end{split}
	\end{equation*}
	Now we get expressions for $\mlt''(u;\,a_0)$ and $\mrt''(u;\,b_0)$.
 	Using equation~(\ref{e39}) differentiated once,
	we can express $\mlt''$ in terms of	$\mlt'$. After that, using~(\ref{e39}) one more time, we can express~$\mlt'$ in 
	terms of~$\mlt$. Applying these considerations to $\mlt''(a_0;\,a_0)$,
	we obtain
	$$
		\mlt''(a_0;\,a_0) = \eps^{-1}\big(\eps^{-1}\mlt(a_0;\,a_0)-\eps^{-1}f'(a_0)-f''(a_0)\big).
	$$
	Substituting expression~(\ref{e332}) for $\mlt(a_0;\,a_0)$ into this identity, we get
	$$
		\mlt''(a_0;\,a_0) = -\eps^{-1}\bigg[f''(a_0) - \frac{f'(b_0)-f'(a_0)}{2\eps}\bigg] = -\eps^{-1}\Dlt(a_0,b_0).
	$$
	Using~(\ref{e312}), we finally have
	\begin{equation}\label{e335}
		\mlt''(u;\,a_0) = -\eps^{-1}\Dlt(a_0,b_0)e^{(u-a_0)/\eps} + 
		\eps^{-1}e^{u/\eps}\int\limits_{u}^{a_0} f'''(t)e^{-t/\eps}\,dt.
	\end{equation}
	Similar reasoning gives the formula for $\mrt''(u;\,b_0)$:
	\begin{equation}\label{e336}
		\mrt''(u;\,b_0) = \eps^{-1}\Drt(a_0,b_0)e^{(b_0-u)/\eps} + 
		\eps^{-1}e^{-u/\eps}\int\limits_{b_0}^u f'''(t)e^{t/\eps}\,dt.
	\end{equation}
	
	As usual, we summarize this section in one proposition.
	\begin{Prop}\label{S7}
		Suppose $b_0 - a_0 = 2\eps$ and $\Alv(a_0,b_0)$ is a full cup. 
		Consider domains $\Lt(u_1,a_0)$ and $\Rt(b_0,u_2)$ adjacent to $\Alv(a_0,b_0)$.
		The Bellman candidate
		in the union $\LtRt(u_1,[a_0,b_0],u_2)$ has the form
		\begin{equation}\label{e327}
			\Bltrt(x;\,u_1,[a_0,b_0],u_2) =  
			\begin{cases}
				\Blt(x;\,u_1,[a_0,b_0]),& x \in \Lt(u_1,a_0);\\
				\Balv(x;\,a_0,b_0),& x \in \Alv(a_0,b_0);\\
				\Brt(x;\,[a_0,b_0],u_2),& x \in \Rt(b_0,u_2),
			\end{cases}   		
		\end{equation}
		where $\Balv(x;\,a_0,b_0)$ can be restored by the linearity on the chords according to~\textup{(\ref{e325})}. 
		The functions $\Blt(x;\,u_1,[a_0,b_0])$ and $\Brt(x;\,[a_0,b_0],u_2)$ can be calculated by~\textup{(\ref{e329})} 
		and~\textup{(\ref{e331})}, respectively.
		In addition\textup, the following inequalities must be fulfilled\textup:
		\begin{equation}\label{eq14}
			\left\{
			\begin{aligned}
				&\mlt''(u;\,a_0)\ge 0\quad\mbox{for}\quad u \in (u_1,a_0);\\
				&\mrt''(u;\,b_0)\le 0\quad\mbox{for}\quad u \in (b_0,u_2),
			\end{aligned} \right.
		\end{equation}
		where $\mlt''(u;\,a_0)$ and $\mrt''(u;\,b_0)$ can be calculated by~\textup{(\ref{e335})} 
		and~\textup{(\ref{e336})}.
		\end{Prop}

\subsection{Optimizers on chords}\label{s53}	
	We consider a domain~$\Ch$ foliated by chords (see Section~\ref{s51}). 
	For every point~$x\in\Ch$, there is a unique extremal chord $[A,B]$ passing through it. Therefore,
	a delivery curve coming to~$x$ can only start at $A$ or $B$, because it must run along the extremal.
	Indeed, in the situation being considered, we have left and right delivery curves: the segments $[A,x]$ and $[x,B]$.
	Such curves are generated by a step function~$\ex$ that can take two values: $a$ and $b$.
	Namely, if $x=\alpha_-A+\alpha_+B$, $\alpha_-+\alpha_+ =1$, we set
	$$	
		\ex(s) =	\left\{ 
		\begin{aligned}
			&a,\quad s \in [0,\alpha_-];\\
			&b,\quad s \in (\alpha_-,1].
		\end{aligned} \right.
	$$
	We can see that~$\ex$ is, indeed, an optimizer for~$x$. Property~(1) of Definition~\ref{D3} follows from the fact 
	that all the Bellman points generated
	by~$\ex$ lie on the chord $[A,B]$. Property~(2) is fulfilled by the construction of~$\ex$. 
	Finally, property~(3) follows from the linearity 
	of the Bellman candidate along the chord $[A,B]$.
	
	Further, it is easy to see that the curve
	$$
		\dc\cii{A}(s) =\big(\av{\ex}{[0,s]},\av{\ex^2}{[0,s]}\big),\quad s \in (0,1],
	$$
	is a left delivery curve that starts at~$A$, runs along $[A,B]$, and ends at~$x$. Similarly, 
	we can define the right delivery curve~$\dc\cii{B}$ that starts at~$B$ and
	ends, again, at~$x$.
	
	Now we consider the construction $\LtRt(u_1,[a_0,b_0],u_2)$ described in Section~\ref{s52}. Let~$W_0$ be the tangency point of the chord $[A_0,B_0]$ 
	and the upper parabola. This point is the entry node for both domains $\Lt(u_1,a_0)$ and $\Rt(b_0,u_2)$. After we connect~$A_0$ and~$W_0$ 
	with the left delivery curve~$\dc\cii{A_0}$ 
	generated by the optimizer for~$W_0$, we can continue this curve up to every point in $\Rt(b_0,u_2)$ (see Section~\ref{s33}). 
	On the other hand, the 
	right delivery curve~$\dc\cii{B_0}$ that connects $B_0$ and~$W_0$, can be continued up to every 
	point in $\Lt(u_1,a_0)$.
	
	We conclude that delivery curves can originate not only at $\pm\infty$, but also in cups. 
	Thus, we have all the information required for the construction of delivery curves
	in domains adjacent to cups.
	
\subsection{Function \texorpdfstring{$f'''$}{f'''} changes its sign from plus to minus}\label{s54}
	It turns out that the cup, together with two domains~$\Lt$ and~$\Rt$ adjacent to it, always arises when
	$f'''$ changes its sign once, from plus to minus.
	We state and prove the appropriate theorem.
	\begin{Th}\label{T3}
		Suppose $0<\eps<\eps_0$\textup, $f \in \W^{0}$\textup, and $c_0 \neq \pm \infty$. Then we can build a full cup $\Alv(a_0,b_0)$ originated at $c_0$\textup. We also have 
		$$
			\Bell(x;\,f) = \Bltrt(x;\,-\infty,[a_0,b_0],+\infty),
		$$
		where the function~$\Bltrt$ is defined by~\textup{(\ref{e327})}.
	\end{Th}
	First, we note that a cup is a local construction. 	
	Its existence under the conditions of the theorem follows from the general lemma, in which 
	the function~$f$ is considered only in some neighborhood of~$c_0$. 
	\begin{Le}\label{L5}
		Consider a segment $\Delta = [c-\ell_0,c+\ell_0]$\textup, where ${c\in\mathbb{R}}$ is its center and 
		the positive number ${2\ell_0}$ is its length. Consider a 
		function~${f\in C^2(\Delta)\cap W_3^1(\Delta)}$. Suppose ${f'''>0}$ a.e. on the left half 
		${[c-\ell_0,c]}$ 
		of $\Delta$ and ${f'''<0}$ a.e. on the right half ${[c,c+\ell_0]}$. 		
		Then there exist two functions $a(\ell)$ and $b(\ell)= a(\ell) + \ell$\textup, 
		$\ell \in (0,\ell_0]$\textup, with the
		following properties\textup:
		\begin{enumerate}
			\item[\textup{1)}] $a(\ell)<c<b(\ell)$\textup;	
			\item[\textup{2)}] $a(\ell)$ and $b(\ell)$ solve equation~\textup{(\ref{urlun})}\textup;
			\item[\textup{3)}] 
				$\Dlt\big(a(\ell),b(\ell)\big)<0$ and $\Drt\big(a(\ell),b(\ell)\big)<0$\textup;
			\item[\textup{4)}] $a$ and~$b$ are differentiable functions such that $a'<0$ and $b'>0$.
		\end{enumerate}
	\end{Le}

	Setting $\ell_0 = 2\eps$ and using the lemma just stated, we see that the non-intersecting chords 
	$[A(\ell), B(\ell)]$ form a full cup $\Alv(a_0,b_0)$ with ends
	$a_0 = a(\ell_0)$ and $b_0 = b(\ell_0)$.
	
	Further, since $\Dlt(a_0,b_0)< 0$ and $\Drt(a_0,b_0) < 0$, it follows that conditions~(\ref{eq14}) in 
	Proposition~\ref{S7} are satisfied. 		
	Suppose the domains $\Lt(-\infty,a_0)$ and $\Rt(b_0,+\infty)$ adjoin  our cup. Proposition~\ref{S7}
	tells us that the function
	${\Bltrt(x;\,-\infty,[a_0,b_0],+\infty)}$ defined  
	by~(\ref{e327}) is a Bellman candidate in the domain ${\LtRt(-\infty,[a_0,b_0],+\infty) = \Omega_\eps}$. 	
	Therefore,
	Statement~\ref{S2} guarantees that $\Bell \le \Bltrt$. 
	The converse estimate $\Bell \ge \Bltrt$ follows from the existence of optimizers 
	for each point in~$\Omega_\eps$ (see Section~\ref{s53}). It remains to prove 
	Lemma~\ref{L5}.

\paragraph{Proof of Lemma~\ref{L5}.}
	First, without loss of generality, we can set $c=0$. 
	This follows from the linear substitution in all the conditions on the required functions~$a$ and~$b$.
	
	Now we verify that for any~$\ell$, $0<\ell\le\ell_0$, there exist points $a$ and $b=a+\ell$ solving  
	equation~(\ref{urlun}), and for such points the relation $a<0<b$ is always fulfilled.
	Note that for all the points~$a$ and~$b$ such that ${-\ell_0\le a<b\le 0}$, the left part of chord 
	equation~(\ref{urlun}) is strictly smaller than its right part.
	Indeed, the requirement on the sign of $f'''$ implies that $f''$ is strictly increasing on
	$[-\ell_0,0]$, and so $f'$ is strictly convex on this interval.
	Thus, on $(a,b)$
	the function~$f'$ is strictly less than the linear function whose graph contains the points 
	$(a,f'(a))$ and $(b,f'(b))$. This implies that the average of~$f'$
	over $[a,b]$ is strictly less than the average of this linear function, i.e.
	$$
		\av{f'}{[a,b]} < \frac{f'(a)+f'(b)}{2}.
	$$	
	Similarly, for any points $a$ and $b$ such that ${0\le a<b\le \ell_0}$, the left part of equation~(\ref{urlun}) is 
	strictly greater than its right part.	

	If we fix~$\ell$ and set $b=a+\ell$, then we can treat the difference between the left and right parts 
	of~(\ref{urlun}) as a continuous 	
	function of $a\in[-\ell_0,0]$. We see that this function takes 
	both positive and negative values. Therefore, it vanishes at some point~$a$, and the pair
	$a$ and $b = a + \ell$ solves equation~(\ref{urlun}).  
	Besides, in view of our considerations in the beginning of the proof, we have $a<0$ and $b>0$.
		
	Now we prove that $\Dlt(a,b)<0$ and $\Drt(a,b)<0$ if $a$ and~$b$ solve equation~(\ref{urlun}).
	Consider the function 
	$$
		q(t) = f'(t) + \alpha_1 t + \alpha_2,
	$$ 
	where the coefficients $\alpha_1$ and $\alpha_2$ are chosen so that $q(a)=q(b)=0$. 
	It is easily shown that such a function has the following properties:
	\begin{enumerate}
		\item[1)]
			$q'' = f'''$;
		\item[2)]
			equation~(\ref{urlun}) on the ends of chords is equivalent to the identity $\av{q}{[a,b]} = 0$;
		\item[3)]
			the inequalities $\Dlt(a,b)<0$ and $\Drt(a,b)<0$ can be rewritten as $q'(a)<0$ and $q'(b)<0$, respectively.
	\end{enumerate}
	Further, by the condition on the sign of~$f'''$, the function~$q$ is strictly convex on $[a,0]$ and 
	strictly concave on $[0,b]$. Thus, by simple geometric 
	considerations, $q$ has at most one root on $(a,b)$. If this root does not exist, then the  identity
	$\av{q}{[a,b]} = 0$ cannot hold 
	(this identity means precisely that the areas of two hatched domains on Figure~\ref{fig:fq} are equal).	
	
\begin{figure}[H]
\begin{center}
\includegraphics{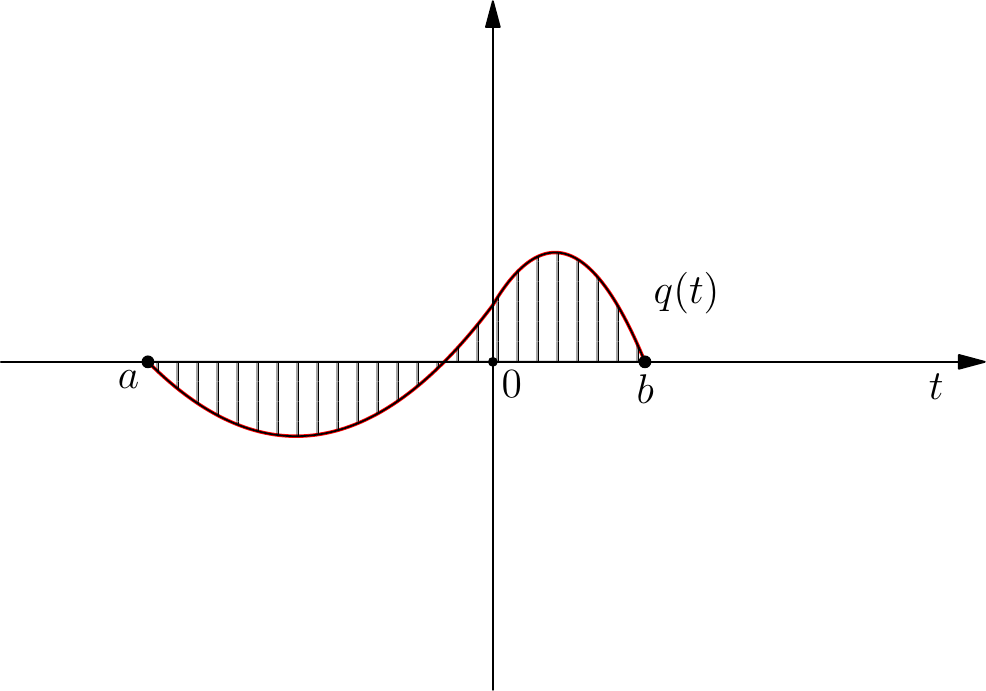}
\caption{A function with zero mean (its convexity changes at $t=0$).}
\label{fig:fq}
\end{center}
\end{figure}
	
	But if $q'(a)\ge 0$ or $q'(b)\ge 0$, the function~$q$ has no roots on~$(a,b)$ by geometric considerations. 
	Thus, we have proved the estimates $\Dlt(a,b)<0$ and $\Drt(a,b)<0$.
	
	Now we find points $a_0$ and $b_0=a_0+\ell_0$ solving equation~(\ref{urlun}). 	
	This equation can be written as ${\Phi(a_0,\ell_0)=0}$, where
	$$
		\Phi(a,\ell) = \ell\big(f'(a)+f'(a+\ell)\big)-2\big(f(a+\ell)-f(a)\big).
	$$
	Differentiating~$\Phi$ with respect to the first variable, we have
	\begin{align*}
		\Phi_a'(a,\ell) &= \ell\big(f''(a)+f''(a+\ell)\big)-2\big(f'(a+\ell)-f'(a)\big) \\
		&= \ell\big(\Dlt(a,b)+\Drt(a,b)\big).
	\end{align*}
	Therefore, $\Phi_a'(a_0,\ell_0)<0$. Consequently, by the implicit function theorem, 
	there exists an interval $(\tilde\ell, \ell_0]$ on which we can define a unique
	differentiable function $a(\ell)$ satisfying the identity $a(\ell_0) = a_0$ and, together with 
	the function $b(\ell) = a(\ell)+\ell$, solving chord equation~(\ref{urlun}).
	In addition,
	$$
		a'(\ell) = -\frac{\Phi_\ell'\big(a(\ell),\ell\big)}{\Phi_a'\big(a(\ell),\ell\big)}.
	$$
	But
	$$
		\Phi_\ell'\big(a,\ell\big) = \ell f''(a+\ell) + f'(a) - f'(a+\ell) = \ell \Drt(a,b),
	$$
	and so $-1<a'(\ell)<0$ and $b'(\ell) = a'(\ell) + 1 > 0$ for $\ell\in(\tilde \ell,\ell_0]$.
	
	Further, let $(\tilde\ell, \ell_0]$ be the union of all the appropriate intervals, i.e. the intervals such that the identity $\Phi(a,\ell) = 0$,
	together with the requirement $a(\ell_0)=a_0$,
	defines a unique differentiable function~$a(\ell)$ on them. We claim that $\tilde\ell = 0$.
	Indeed, let $\tilde\ell > 0$. 	
	We choose some decreasing sequence $\ell_n$ on $(\tilde\ell, \ell_0]$ that
	converges to $\tilde\ell$. Then $a_n = a(\ell_n)$ is an increasing sequence and, besides, $a_n<0$.	
	We denote its limit by $\tilde a$. By continuity, we have 
	$\Phi(\tilde a,\tilde \ell) = 0$. Then, using the implicit function theorem again, we can increase the interval $(\tilde\ell, \ell_0]$. 
	But this contradicts the assumption of its maximality.
	
	As a result, we have the functions~$a$ and~$b$ defined on~$(0,\ell_0]$ and satisfying all the conditions required. 
	\qed

\subsection{Examples}
	\paragraph{Example 7. A fourth-degree polynomial.}
	
	In example~5, we discussed the case of an arbitrary fourth-degree polynomial with positive leading coefficient.
	Now we apply Theorem~\ref{T3} to a fourth-degree polynomial with negative leading coefficient.
	Such a polynomial belongs to~$\W^{0}$ for any $\eps_{0}>0$. From Remark~\ref{rem2}, it follows that, without loss of generality, we may
	set $f(t)=-(t-c)^{4}$. The conditions of Theorem~\ref{T3} are satisfied for such a function, and so it remains
	to find an analytic expression for the Bellman function.
	
	First, we are looking for a domain foliated by chords (a cup). Let $a=c-\sigma$ and $b=c+\tau$. Then after this substitution and
	simple transformations, equation~(\ref{urlun}) takes the form
	$$
		(\sigma-\tau)(\sigma+\tau)^{2}=0.
	$$
	Since the ends of the chords must lie on the opposite sides from the point $c$ (the cup origin), 
	the numbers $\sigma$ and $\tau$
	must have the same sign. Thus, their sum cannot vanish, and so $\sigma=\tau$. Therefore,
	all the chords are parallel to each other and the ends of the cup are $c-\eps$ and $c+\eps$.
	For any $\sigma \in [0,\eps]$, the Bellman function on $[A,B]$, where $a=c-\sigma$ and $b=c+\sigma$,
	can be calculated by the formula
	\begin{align*}
		\Bell(x_{1},x_{2}) &= \Bell(x_{1},(a+b)x_{1}-ab)= \Bell(x_{1},2cx_{1}-c^{2}+\sigma^{2})\\
		&=\frac{f(c+\sigma)-f(c-\sigma)}{2\sigma}(x_{1}-c+\sigma)+f(c-\sigma)
		=-\sigma^{4}.
	\end{align*}

	Now we find the Bellman function in the remaining domains. As we know, the domain on the right of the cup
	is foliated by the right tangents, and so the Bellman function in it is given by
	$$
	\Bell(x_{1},x_{2})=
	\mrt(u;\,c+\eps)\,(x_{1}-u)-(u-c)^4,
	$$
	where $u=x_{1}+\eps-\sqrt{x_{1}^{2}-x_{2}+\eps^{2}}$.
	The function $\mrt(u;\,c+\eps)$ can be calculated by~(\ref{e330}):
	\begin{align*}
		\mrt(u;\,c+\eps)&=-4\eps^{-1}e^{-u/\eps}\int\limits_{c+\eps}^{u}(t-c)^3 e^{t/\eps}\,dt\\
		&=-8\eps^3e^{1-\frac{u-c}{\eps}}-4(u-c)^3+12\eps(u-c)^2
		-24\eps^2(u-c)+24\eps^3.
	\end{align*}

	On the left of the cup, the domain is foliated by the left tangents and we have
	$$
		\Bell(x_{1},x_{2})=
		\mlt(u;\,c-\eps)\,(x_{1}-u)-(u-c)^4,
	$$
	where $u=x_{1}-\eps+\sqrt{x_{1}^{2}-x_{2}+\eps^{2}}$.
	The function $\mlt(u;\,c-\eps)$ can be calculated by~(\ref{e328}):
	\begin{align*}
		\mlt(u;\,c-\eps)&=-4\eps^{-1}e^{u/\eps}\int\limits_{u}^{c-\eps}(t-c)^3 e^{-t/\eps}\,dt\\
		&=8\eps^3e^{1+\frac{u-c}{\eps}}-4(u-c)^3-12\eps(u-c)^2
		-24\eps^2(u-c)-24\eps^3.
	\end{align*}

\section{General case}
	In this chapter we will obtain the function $\Bell(x;\,f)$ for ${f\in\W^N}$, ${N\in\mathbb{Z}_+}$.
	In Sections~\ref{s61} and~\ref{s62}, we will study another construction that is, in some sense, 
	a mixture of an angle and a cup. 
	In Section~\ref{s63}, we will see that all our constructions 
	will suffice for the announced function~$\Bell$ to be built. Also, we will describe the general form
	of this function.
	Finally, in Section~\ref{Alg}, we will explain how to obtain~$\Bell$.
	
\subsection{Trolleybus}\label{s61}
	The following considerations, which are not intended to be rigorous,
	will lead us to a new construction (the last of those that are required for the general case). 
	We have seen in Section~\ref{s43} that in the situation where $f'''$ changes its sign from minus
	to plus, an angle~$\Ang$ can arise. If $f'''$ changes its sign from plus to minus, then the cup~$\Alv$
	arises around the point where the sign changes. Now we assume that $f'''$ changes its sign twice.
	Then one point where the sign changes generates a cup and the other can generate an angle.
	It is not difficult to imagine a situation where the angle and the cup stick together. It turns
	out, that they can not only stick, but ``mix'' with each other and generate one of the constructions
	shown in Figures~\ref{fig:ptr} and~\ref{fig:ltr}. Now we give a rigorous description of such
	constructions and build corresponding Bellman candidates.

\begin{figure}[H]
\begin{center}
\includegraphics{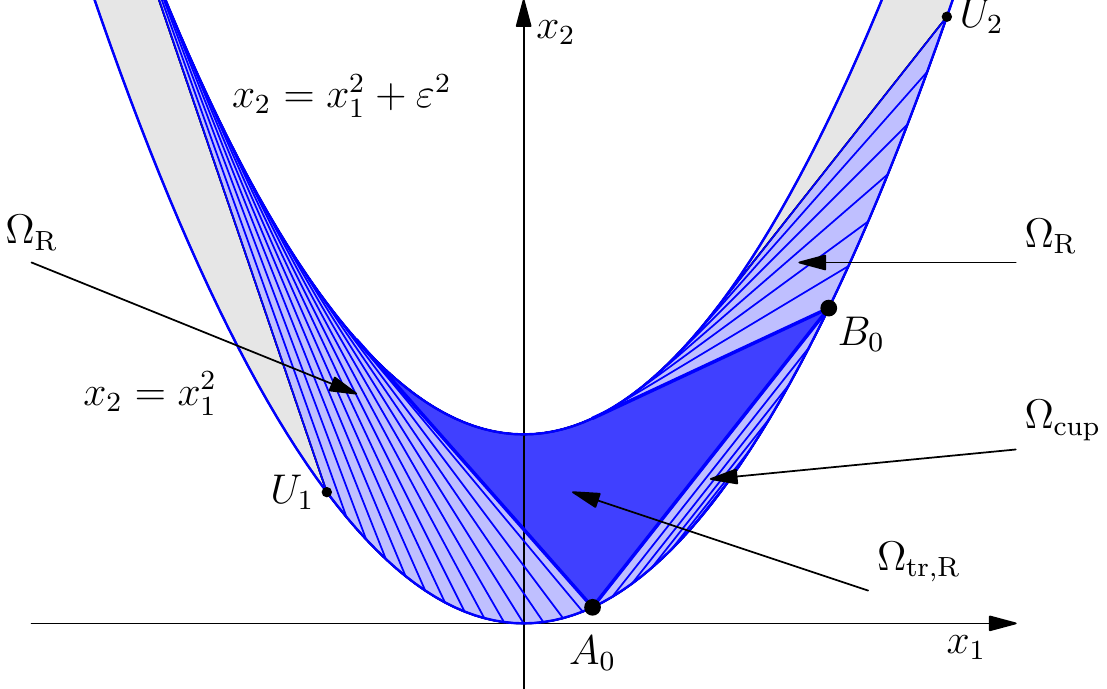}
\caption{A right trolleybus $\RTroll$.}
\label{fig:ptr}
\end{center}
\end{figure}
			
	Suppose ${u_1<a_0<b_0<u_2}$ and $b_0-a_0\le2\eps$. Consider a cup $\Alv(a_0,b_0)$ (it may be not full) and 
	the domains $\Rt(u_1,a_0)$ and $\Rt(b_0,u_2)$ foliated
	by the extremal tangents. The quadrangular subdomain of~$\Omega_\eps$, bounded by the upper chord
	$[A_0,B_0]$, the right tangents coming from $A_0$ and $B_0$, and the arc of the upper parabola, is called
	\emph{the right trolleybus}\footnote{Glancing at Figure~\ref{fig:ptr}, the reader will hardly understand
	why such a name was chosen. The point is the low artistic skills of the authors. When this construction
	was drawn on a blackboard for the first time, the one-sided tangents, bounding the subdomain, 
	were almost parallel and looked like trolley poles drawing the electricity from the upper parabola.} 
	and is denoted by $\RTroll(a_0,b_0)$ (see Figure~\ref{fig:ptr}).	
	Similarly, we can define the left trolleybus $\LTroll(a_0,b_0)$ and 
	the corresponding construction shown in Figure~\ref{fig:ltr}.

\begin{figure}[H]
\begin{center}
\includegraphics{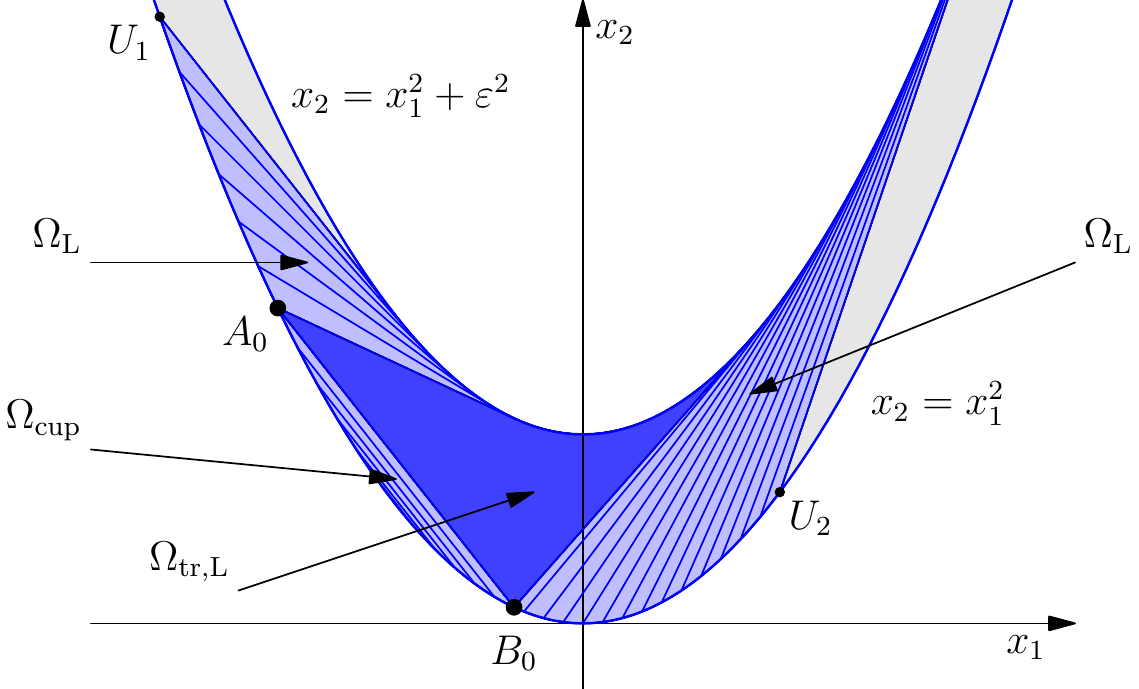}
\caption{A left trolleybus $\LTroll$.}
\label{fig:ltr}
\end{center}
\end{figure}
	
Note that for $b_0-a_0=2\eps$ the trolleybus degenerates into an angle 
adjacent to a cup.
	
We consider the construction with the right trolleybus. Our goal is to 
build a Bellman candidate in the domain
$$
\RtRt(u_1,[a_0,b_0],u_2) = 
\Rt(u_1,a_0) \cup \Alv(a_0,b_0) \cup \RTroll(a_0,b_0) \cup \Rt(b_0,u_2).
$$
We denote the function required by $\Brtrt(x;\,u_1,[a_0,b_0],u_2)$. 
In the trolleybus, our candidate is linear by the minimality:
\begin{align*}
\Brtrt(x;\,u_1,[a_0,b_0],u_2)&=\Brtroll(x;\,a_0,b_0) 
\\
&= \beta_1 x_1 + \beta_2 x_2 + \beta_0,\quad x\in\RTroll(a_0,b_0).
\end{align*}
We already know that the Bellman candidate coincides with $\Balv(x;\,a_0,b_0)$ 
in $\Alv(a_0,b_0)$ and with $\Brt(x;\,b_0,u_2)$ in $\Rt(b_0,u_2)$. 
The latter function is not defined uniquely (the value $\mrt(b_0)$ must be 
chosen). The necessary and sufficient conditions for the concatenation of
$\Brtroll(x;\,a_0,b_0)$, $\Balv(x;\,a_0,b_0)$, 
and $\Brt(x;\,b_0,u_2)$ to be continuous, can be written as
\begin{equation}\label{e8}
	\left\{ 
	\begin{aligned}
		&\beta_1 a_0 + \beta_2 a_0^2 + \beta_0 = f(a_0);\\
		&\beta_1 b_0 + \beta_2 b_0^2 + \beta_0 = f(b_0);\\
		&\mrt(b_0) = \beta_1 + 2(b_0-\eps) \beta_2.
	\end{aligned} \right.
\end{equation}
Indeed, the first two identities must be fulfilled by the boundary condition, 
and they imply that $\Brtroll$ is glued to $\Balv$ continuously. The last 
identity guarantees that the concatenation of $\Brtroll(x;\,a_0,b_0)$ and
$\Brt(x;\,b_0,u_2)$ is continuous. We have obtained this equation expressing
$x_2$ in terms of $x_1$ on the right boundary of the trolleybus (see equation~(R) 
in Section~\ref{s31})	and then equating the coefficient of $x_1$ with $\mrt(b_0)$.

Now, assume that the functions $\Balv(x;\,a_0,b_0)$ and $\Brt(x;\,b_0,u_2)$ 
are locally concave. In order for their concatenation with $\Brtroll(x;\,a_0,b_0)$ 
to be locally concave, it is necessary that the jumps of the derivative in $x_2$ 
are non-positive on the corresponding boundaries of~$\RTroll$. Using~(\ref{e326}), 
we see that on the lower boundary of the trolleybus (i.e. on the chord
$[A_0,B_0]$), the jump can be calculated as follows:
$$
\delta_1=\beta_2-\frac12\Av{f''}{[a_0,b_0]}.
$$
On the right boundary of the trolleybus (i.e. on the right tangent coming
from~$B_0$), the jump can	be calculated by the formula
$$
\delta_2 =\lim_{u\to b_0+}t_2(u)-\beta_2,
$$
where $t_2 = \Brt_{x_2}$ in $\Rt(b_0,u_2)$.	
Using~(\ref{e36}), (\ref{e35}), and, after that, the last identity in~(\ref{e8}), 
we obtain
$$
\lim_{u\to b_0+}t_2(u) = \frac{\mrt'(b_0)}{2}
= \frac{f'(b_0) - \mrt(b_0)}{2\eps} = 
\frac{f'(b_0) - \beta_1 - 2(b_0-\eps) \beta_2}{2\eps}.
$$
Therefore, we have
\begin{equation}\label{e9}	
\delta_2 = \frac{f'(b_0)-\beta_1 - 2\beta_2 b_0}{2\eps}.
\end{equation}
Subtracting the first equation in~(\ref{e8}) from the second one, we get
\begin{equation}\label{eq1}
\beta_1+\beta_2(a_0+b_0) = \frac{f(b_0) - f(a_0)}{b_0-a_0} = \Av{f'}{[a_0,b_0]}.
\end{equation}
Using chord equation~(\ref{urlun}), we obtain
\begin{equation}\label{eq2}
\beta_1+\beta_2(a_0+b_0) = \frac{f'(a_0)+f'(b_0)}{2}. 
\end{equation}
Expressing $\beta_1$ in terms of $\beta_2$ and substituting the resulting 
expression into~(\ref{e9}), we have
\begin{align*}
\delta_2 &= \frac{f'(b_0)-f'(a_0)-2\beta_2(b_0-a_0)}{4\eps} 
\\
&= -\frac{b_0-a_0}{2\eps}\Big(\beta_2 - \frac{1}{2}\av{f''}{[a_0,b_0]}\Big) 
\\
&= -\frac{b_0-a_0}{2\eps}\delta_1.
\end{align*}
But $\delta_1$ and $\delta_2$ must have the same sign and so $\delta_1=\delta_2=0$.
In its turn, this condition implies that the concatenation of $\Balv(x;\,a_0,b_0)$, 
$\Brtroll(x;\,a_0,b_0)$, and $\Brt(x;\,b_0,u_2)$ is $C^1$-smooth (because on the
boundaries of the trolleybus the derivatives in two 
non-collinear directions~--- along $x_2$ and along the corresponding 
boundary~--- are glued continuously). 
But if the concatenation is $C^1$-smooth and its components are locally concave, 
then it is also locally concave.
Therefore, the identity $\delta_1 = 0$ or
\begin{equation}\label{eq3}
\beta_2 = \frac{1}{2}\Av{f''}{[a_0,b_0]}
\end{equation}
is a necessary and sufficient condition for the concatenation of 
the linear function $\Brtroll(x;\,a_0,b_0)$ with the locally concave functions 
$\Balv(x;\,a_0,b_0)$ and $\Brt(x;\,b_0,u_2)$ to be locally concave. 	
Substituting expression~(\ref{eq3}) into~(\ref{eq1}), we get
\begin{equation}\label{eq4}
\beta_1 = \Av{f'}{[a_0,b_0]} - \frac{1}{2}(b_0+a_0)\Av{f''}{[a_0,b_0]}.
\end{equation}
Also, the expression for $\beta_2$ can be substituted in~(\ref{eq2}):
\begin{equation}\label{eq5}
  \begin{split}
	\beta_1 &= \frac{f'(a_0)+f'(b_0)}{2} - (a_0+b_0)\frac{f'(b_0)-f'(a_0)}{2(b_0-a_0)}
		\\
		&= \frac{b_0 f'(a_0)-a_0 f'(b_0)}{b_0-a_0}.
	\end{split}
\end{equation}
Summing the first and the second equations in~(\ref{e8}) and, after that, 
substituting expressions~(\ref{eq3}) and~(\ref{eq4}), we obtain
\begin{equation}\label{eq12}	
\beta_0=\frac{b_0 f(a_0)-a_0 f(b_0)}{b_0-a_0}+\frac{1}{2}a_0 b_0\Av{f''}{[a_0,b_0]}.
\end{equation}
Finally, substituting expressions~(\ref{eq3}) and~(\ref{eq5}) in the last 
equation into~(\ref{e8}), we have
\begin{align*}
\mrt(b_0)
&=\frac{b_0f'(a_0)-a_0f'(b_0)}{b_0-a_0}+(b_0-\eps)\frac{f'(b_0)-f'(a_0)}{b_0-a_0}
\\
&=f'(b_0)-\eps\Av{f''}{[a_0,b_0]}.
\end{align*}
We note that if $b_0-a_0 = 2\eps$, then
$$
\mrt(b_0) = f'(b_0) - \frac{f'(b_0)-f'(a_0)}{2} = \mrt(b_0;\,b_0),
$$
where $\mrt(u;\,b_0)$ is given by~(\ref{e330}) 
(that expression was defined only for the case $b_0-a_0 = 2\eps$). Now we extend 
the notation~$\mrt(u;\,b_0)$ to	the general case $b_0-a_0 \le 2\eps$:
\begin{equation*}
\mrt(u;\,b_0)
	= \big(f'(b_0)-\eps\av{f''}{[a_0,b_0]}\big)e^{(b_0-u)/\eps} + 
		\eps^{-1}e^{-u/\eps}\int\limits_{b_0}^u f'(t)e^{t/\eps}\,dt.
\end{equation*}
It is easy to prove that formula~(\ref{e336}) for~$\mrt''(u;\,b_0)$ remains true. 
	
Thus, we have 
$$
\mrt(u) = \mrt(u;\,b_0),\quad u\in (b_0,u_2),
$$ 
and
$$
\Brtrt(x;\,u_1,[a_0,b_0],u_2) = \Brt(x;\,[a_0,b_0],u_2),\quad x \in \Rt(b_0,u_2),
$$ 
where the function $\Brt(x;\,[a_0,b_0],u_2)$ is still given by~(\ref{e331}).

Now we consider the concatenation of~$\Brtroll$ and $\Brt(x;\,u_1,a_0)$. 
Arguing the same way as for the right boundary of the trolleybus, we get a 
necessary	and sufficient condition for our concatenation to be continuous 
on the left boundary:
$$
\mrt(a_0) = \beta_1 + 2(a_0-\eps) \beta_2.
$$
Substituting expressions~(\ref{eq3}) and~(\ref{eq5}) into this formula, we obtain
\begin{equation}\label{e4}
\begin{split}
\mrt(a_0)
&=\frac{b_0f'(a_0)-a_0f'(b_0)}{b_0-a_0}+(a_0-\eps)\frac{f'(b_0)-f'(a_0)}{b_0-a_0} 
\\
&=f'(a_0)- \eps\Av{f''}{[a_0,b_0]}.
\end{split}
\end{equation}
Using equation~(\ref{e35}) twice, we have
$$
\mrt(a_0) = f'(a_0)-\eps \mrt'(a_0) =
f'(a_0) - \eps \big(f''(a_0)-\eps \mrt''(a_0)\big).
$$	
This allows us to rewrite identity~(\ref{e4}) as
\begin{equation}\label{eq15}
\mrt''(a_0) =  \eps^{-1}\big(f''(a_0) - \av{f''}{[a_0,b_0]}\big) = \eps^{-1}\Dlt(a_0,b_0),
\end{equation}
where $\Dlt$ is defined by the first relation in~(\ref{e334}).

Now we verify that the resulting condition implies not only that the 
concatenation of $\Brt(x;\,u_1,a_0)$ and $\Brtroll(x;\,a_0,b_0)$ is 
continuous, but also that it is	$C^1$-smooth. We set $t_2=\Brt_{x_2}$ 
in $\Rt(u_1,a_0)$. Using (\ref{e36}), (\ref{e35}), and (\ref{e4}), we get
$$
\lim_{u\to a_0-}t_2(u)=\frac{\mrt'(a_0)}2=\frac{f'(a_0)-\mrt(a_0)}{2\eps}
=\frac{1}{2}\Av{f''}{[a_0,b_0]}=\beta_2.		
$$
As usual, this implies the $C^1$-smoothness of the concatenation on the left 
boundary of the trolleybus. But since the concatenation is $C^1$-smooth, 
it is locally concave provided all its components are locally concave. 
	
Now we discuss the left trolleybus $\LTroll(a_0,b_0)$ and construct a candidate
$\Bltlt(x;\,u_1,[a_0,b_0],u_2)$	in the union
$$
\LtLt(u_1,[a_0,b_0],u_2) 
=\Lt(u_1,a_0)\cup\Alv(a_0,b_0)\cup\LTroll(a_0,b_0)\cup\Lt(b_0,u_2).
$$
In order to build such a candidate, we can reason in the same way as we 
did for~$\Brtrt$.
We obtain
\begin{align*}
\Bltlt(x;\,u_1,[a_0,b_0],u_2) &= \Bltroll(x;\,a_0,b_0)
\\ 
&= \beta_1 x_1 + \beta_2 x_2 +\beta_0, \quad x\in\LTroll(a_0,b_0),
\end{align*}
where $\beta_1$, $\beta_2$ and $\beta_0$ are the same as for the right trolleybus. 	
Further, defining the function~$\mlt(u;\,a_0)$, $u\in (u_1,a_0]$, by the formula
$$
\mlt(u;\,a_0) = \big(f'(a_0) + \eps\av{f''}{[a_0,b_0]}\big)e^{(u-a_0)/\eps} + 
\eps^{-1}e^{u/\eps}\int\limits_{u}^{a_0} f'(t)e^{-t/\eps}\,dt
$$
(clearly, this formula coincides with~(\ref{e328}) if $b_0-a_0 = 2\eps$), we obtain
$$
\Bltlt(x;\,u_1,[a_0,b_0],u_2) = \Blt(x;\,u_1,[a_0,b_0]),\quad x\in \Lt(u_1,a_0),
$$
where the function on the right is defined by~(\ref{e329}).
Note that formula~(\ref{e335}) for $\mlt''(u;\,a_0)$ is still correct. 
Concerning the domain $\Lt(b_0,u_2)$, we have
$$
\Bltlt(x;\,u_1,[a_0,b_0],u_2) = \Blt(x;\,b_0,u_2), x\in\Lt(b_0,u_2),
$$
where the coefficient $\mlt(u)$, participating in the definition 
of $\Blt(x;\,b_0,u_2)$, satisfies
\begin{equation}\label{eq16}
\mlt''(b_0) = -\eps^{-1}\Drt(a_0,b_0).
\end{equation}
	
Now we note that by~(\ref{e335}) identity~(\ref{eq15}) is equivalent to the equation
\begin{equation}\label{eq19}
\mrt''(a_0) + \mlt''(a_0;\,a_0) = 0,	
\end{equation}
which has the same form as equation~(\ref{e3}) for the vertex of an angle. 	
Similarly, from~(\ref{e336}), it follows that relation~(\ref{eq16}) is 
equivalent to the	equation
\begin{equation}\label{eq20}
\mlt''(b_0) + \mrt''(b_0;\,b_0) = 0.
\end{equation}
 	
Now, we can formulate a proposition in which our construction with a right 
trolleybus is described.
\begin{Prop}\label{S15}
Let $u_1<a_0<b_0<u_2$ and $b_0-a_0\le2\eps$. Consider a cup 
$\Alv(a_0,b_0)$\textup, two domains $\Rt(u_1,a_0)$ and $\Rt(b_0,u_2)$ 
foliated by the right extremal tangents\textup, and the linearity domain
$\RTroll(a_0,b_0)$ located between them all. A Bellman candidate in the 
union $\RtRt(u_1,[a_0,b_0],u_2)$ of these four domains has the form
\begin{equation*}
\Brtrt(x;\,u_1,[a_0,b_0],u_2) =  
	\begin{cases}
	\Brt(x;\,u_1,a_0),& x \in \Rt(u_1,a_0);\\
	\Balv(x;\,a_0,b_0),& x \in \Alv(a_0,b_0);\\
	\Brtroll(x;\,a_0,b_0),& x \in \RTroll(a_0,b_0);\\
	\Brt(x;\,[a_0,b_0],u_2),& x \in \Rt(b_0,u_2).
  \end{cases}  		
\end{equation*}
Here\textup, $\Brtroll(x;\,a_0,b_0)=\beta_1x_1+\beta_2x_2+\beta_0$ is the 
linear function	with the coefficients given by~\textup{(\ref{eq3}), 
(\ref{eq5}),} and \textup{(\ref{eq12})}. In addition\textup, the following 
conditions must be satisfied\textup:
\begin{equation*}
 \left\{
		\begin{aligned}
				&\mrt''(u;\,b_0) \le 0\quad\mbox{for}\quad u \in (b_0,u_2);\\
				&\mrt''(a_0) + \mlt''(a_0;\,a_0)= 0;\\
				&\mrt''(u) \le 0 \quad\mbox{for}\quad u \in (u_1,a_0),
		\end{aligned} 
 \right.
\end{equation*}
where $\mlt''(u;\,a_0)$ and $\mrt''(u;\,b_0)$ are given by~\textup{(\ref{e335})}
and~\textup{(\ref{e336}),} respectively.
\end{Prop}
	
We also give a symmetric proposition for a left trolleybus.	
\begin{Prop}\label{S16}
Let $u_1<a_0<b_0< u_2$ and $b_0-a_0\le2\eps$. Consider a cup 
$\Alv(a_0,b_0)$\textup, two domains $\Lt(u_1,a_0)$ and $\Lt(b_0,u_2)$ 
foliated by the left extremal tangents\textup, and the linearity domain
$\LTroll(a_0,b_0)$ located between them all. A Bellman candidate in the 
union $\LtLt(u_1,[a_0,b_0],u_2)$ of these four domains has the form
\begin{equation*}
\Bltlt(x;\,u_1,[a_0,b_0],u_2) =  
	\begin{cases}
		\Blt(x;\,u_1,[a_0,b_0]),& x \in \Lt(u_1,a_0);\\
		\Balv(x;\,a_0,b_0),& x \in \Alv(a_0,b_0);\\
		\Bltroll(x;\,a_0,b_0),& x \in \LTroll(a_0,b_0);\\
		\Blt(x;\,b_0,u_2),& x \in \Lt(b_0,u_2).
	\end{cases}  		
\end{equation*}
Here\textup, $\Bltroll(x;\,a_0,b_0)=\beta_1x_1+\beta_2x_2+\beta_0$ 
is the linear function with the coefficients given by~\textup{(\ref{eq3}),
(\ref{eq5}),} and \textup{(\ref{eq12})}. In addition\textup, the following 
conditions must be fulfilled\textup:
\begin{equation*}
	\left\{
	\begin{aligned}
		&\mlt''(u;\,a_0) \ge 0\quad\mbox{for}\quad u \in (u_1,a_0);\\
		&\mlt''(b_0) + \mrt''(b_0;\,b_0) = 0;\\
		&\mlt''(u) \ge 0 \quad\mbox{for}\quad u \in (b_0,u_2).
	\end{aligned} 
	\right.
\end{equation*}
\end{Prop}
	
\subsection{Optimizers in trolleybuses}\label{s62}	
In this subsection we discuss delivery curves and optimizers in constructions with trolleybuses. 
We treat in detail only the case of a right trolleybus.
	
Suppose $B$ is a Bellman candidate in the whole domain~$\Omega_\eps$ and some 
part of the corresponding	foliation forms the construction 
$\RtRt(u_1,[a_0,b_0],u_2)$ described in Proposition~\ref{S15}.
We already know how to build delivery curves in the domain~$\Rt(u_1,a_0)$ and 
in the cup $\Alv(a_0,b_0)$\footnote{It is worth noting that $\Rt(u_1,a_0)$ 
and $\Alv(a_0,b_0)$ are not connected with each other by delivery curves. 
This situation should not be confused with the case of a full cup and two 
domains adjacent to it.} (see Sections~\ref{s33} and~\ref{s53}).	
Let~$A_1$ and~$B_1$ be the points where the rear and the front ``trolley poles'' 
of $\RTroll(a_0,b_0)$ touch 
the upper parabola: $A_1 = \big(a_0-\eps,(a_0-\eps)^2+\eps^2\big)$ and
$B_1 = \big(b_0-\eps,(b_0-\eps)^2+\eps^2\big)$. Let~$\dc$ be the left delivery 
curve that runs along the upper parabola in $\Rt(u_1,a_0)$ and ends at 
the point~$A_1$. This point is the entry node of $\RTroll$ and, as we will 
see later, the curve~$\gamma$ can	be continued from~$A_1$ up to each point 
of the trolleybus. To get an idea of how we are going to do this, 
the reader can look at Figure~\ref{fig:ptrkd}, which shows various delivery 
curves in the trolleybus.
\begin{figure}[H]
\begin{center}
\includegraphics{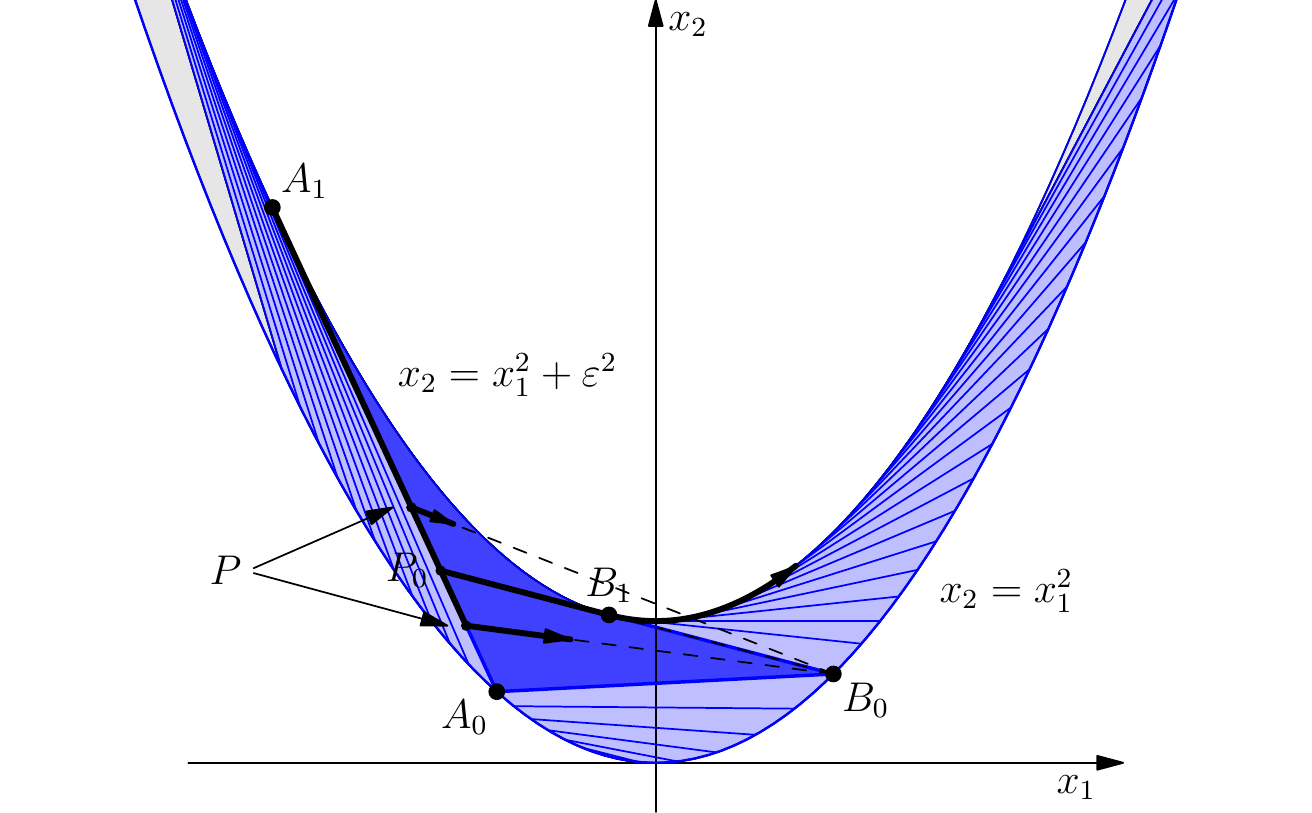}
\caption{A right trolleybus $\RTroll$ and  delivery curves.}
\label{fig:ptrkd}
\end{center}
\end{figure}

Let~$P_0$ be the point where the straight line, containing ``the front pole''
$[B_1,B_0]$ of the trolleybus, intersects ``the rear pole'' $[A_1,A_0]$. We use 
Proposition~\ref{S11} from Section~\ref{s33} twice and continue~$\dc$ with the
segment $[A_1,P_0]$ and, after that, with the segment $[P_0,B_1]$. 	
An important feature of the curve just constructed is that it ``transits'' 
through the trolleybus and ends at the entry node of~$\Rt(b_0,u_2)$. Then 
this curve can be continued up to any point of $\Rt(b_0,u_2)$ 
(see Section~\ref{s33} again). Now we consider the points lying in the 
triangle with vertices $P_0$, $A_0$, and $B_0$. The curve~$\dc$ can be continued 
up to any such a point~$x$ in the same way as above. First, we find the point~$P$ 
where the straight line, containing the segment $[x,B_0]$, intersects the 
segment $[A_1,A_0]$. After that, we continue~$\dc$ with	the segments $[A_1,P]$ 
and $[P,x]$.

It remains to consider the points of the trolleybus that get into the triangle 
with vertices $P_0$, $A_1$, and $B_1$. First, we find point~$P\in[A_1,A_0]$ 
in the same way as described above and continue~$\dc$ with the segment $[A_1,P]$. 
By Proposition~\ref{S11}, the new curve is still a left delivery curve. 	
We continue it with the segment $[P,x]$. Although the conditions of 
Proposition~\ref{S11} are not satisfied this time, we still obtain a left 
delivery curve as a result. The fact is that the conditions of another
proposition~--- some modification	of Proposition~\ref{S11}~--- are fulfilled. 
This modification allows us to overcome the difficulties appearing
from the fact that our curve is continued along the segment intersecting 
the upper boundary transversally.
	
\begin{Prop}\label{S18}
Let~$\dc$ be a convex left delivery curve that is generated by a test 
function~$\ex$ defined on $I=[l,r]$.
Suppose~$\dc$ ends with a straight segment described in
Proposition~\textup{\ref{S11}}.
By $A_0$ we denote the point where the line\textup, containing this 
segment\textup, intersects the lower parabola.
On the lower parabola\textup, we also choose a point~$B_0$ such that 
$0<b_0-a_0\le 2\eps$.
Further\textup, let $x\in[\dc(r),B_0]$ be a point such that 
the candidate~$B$ is linear on $[\dc(r),x]$. Also\textup, 
we assume that on the straight segment that the curve~$\dc$ ends with\textup,
there exists a point $\dc(s_0)$\textup, $s_0\in I$\textup, such that the 
line~$L$\textup, containing the segment $[\dc(s_0),x]$\textup, does not 
intersect the upper parabola 
\textup(see Figure~\ref{fig:nkdt}\textup). 		
If we now continue the curve~$\dc$ with the segment $[\dc(r),x]$\textup, 
then the resulting curve~$\widetilde\dc$ remains a left delivery curve. 
It is generated by the function
$$
\widetilde\ex(s) = 
	\begin{cases}
		\ex(s), &s \in I;\\
		b_0, &s \in [r,\widetilde r].
	\end{cases}
$$
\end{Prop}

\begin{figure}[H]
\begin{center}
\includegraphics{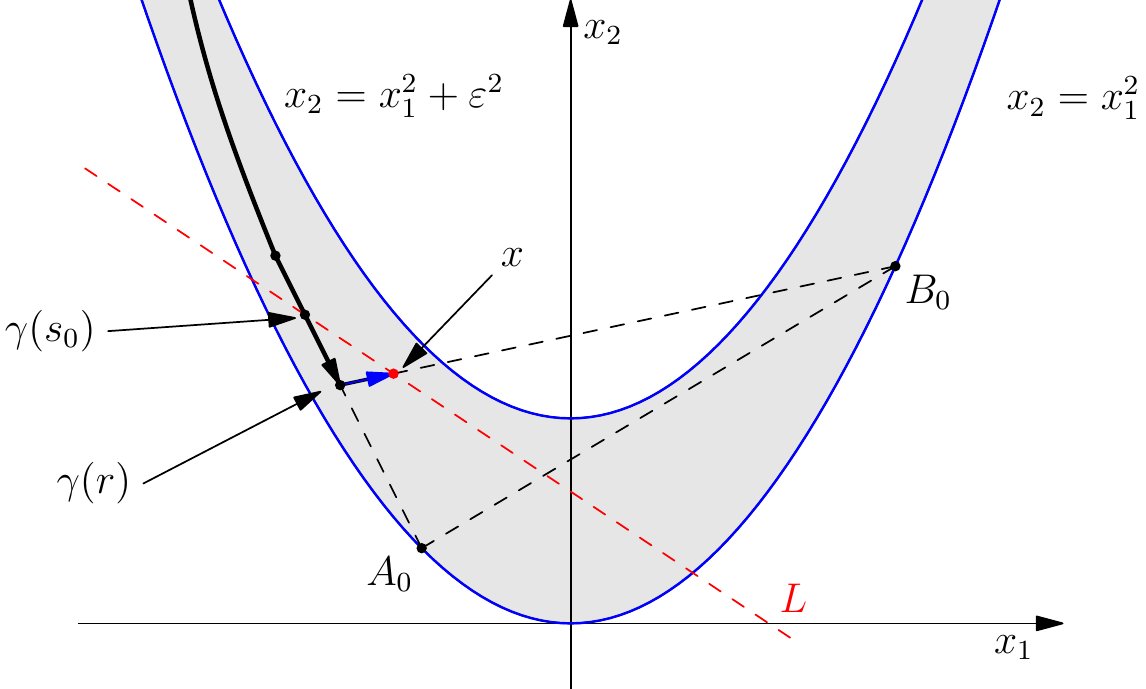}
\caption{Illustration to the proof of Proposition~\ref{S18}.}
\label{fig:nkdt}
\end{center}
\end{figure}

\begin{proof}
The fact that~$\widetilde\ex$, $\widetilde\dc$, and~$B$ are related 
by~(\ref{eq7}) and~(\ref{B_extr_curve}), can be proved in the same 
way as in Proposition~\ref{S11}. Thus, we only need to verify that 
${\widetilde\ex\in\BMO_\eps([l,\widetilde r])}$. We take an arbitrary 
interval ${[c,d]\subset[l,\widetilde r]}$
and see where the Bellman point 
$x^{[c,d]} = \big(\av{\ex}{[c,d]},\av{\ex^2}{[c,d]}\big)$ is located.
If $d\le r$, then the required estimate, as usual, follows from the fact 
that $\ex\in\BMO_\eps(I)$. Therefore, it is sufficient to consider 
the case where $d>r$ and the point $\widetilde\dc(d) = x^{[l,d]}$
lies on the added segment $[\dc(r),x]$ (in such a case, this point, of course, 
lies under the line~$L$).
		
Next, if $c>r$, then~$\widetilde\ex$ is identically equal to~$b_0$ on $[c,d]$, 
and there is nothing to prove. If $s_0\le c\le r$, then on $[c,d]$ 
the function~$\widetilde\ex$ is a step function with values~$a_0$ and~$b_0$.
But then~$x^{[c,d]}$ lies on the chord $[A_0,B_0]$, which is contained 
in~$\Omega_\eps$ entirely, because $|a_{0}-b_{0}|\le2\eps$.

Now, let $c < s_0$. In this case, the point $\widetilde\dc(c)=x^{[l,c]}$ 
lies on the initial delivery curve above~$L$.
But the points~$\widetilde\dc(c)$, $x^{[c,d]}$, and~$\widetilde\dc(d)$ 
lie on one line. The last point is a convex combination
of the first two and locates between them. Hence, $x^{[c,d]}$ lies below~$L$ 
and, therefore, under the upper parabola.		
\end{proof}

Thus, since certain delivery curves in trolleybus intersect the upper parabola 
transversally, it is not always possible to employ Lemma~\ref{L4} and 
Proposition~\ref{S11} directly. But we can overcome this difficulty using
Proposition~\ref{S18}.

\begin{figure}[H]
\begin{center}
\includegraphics{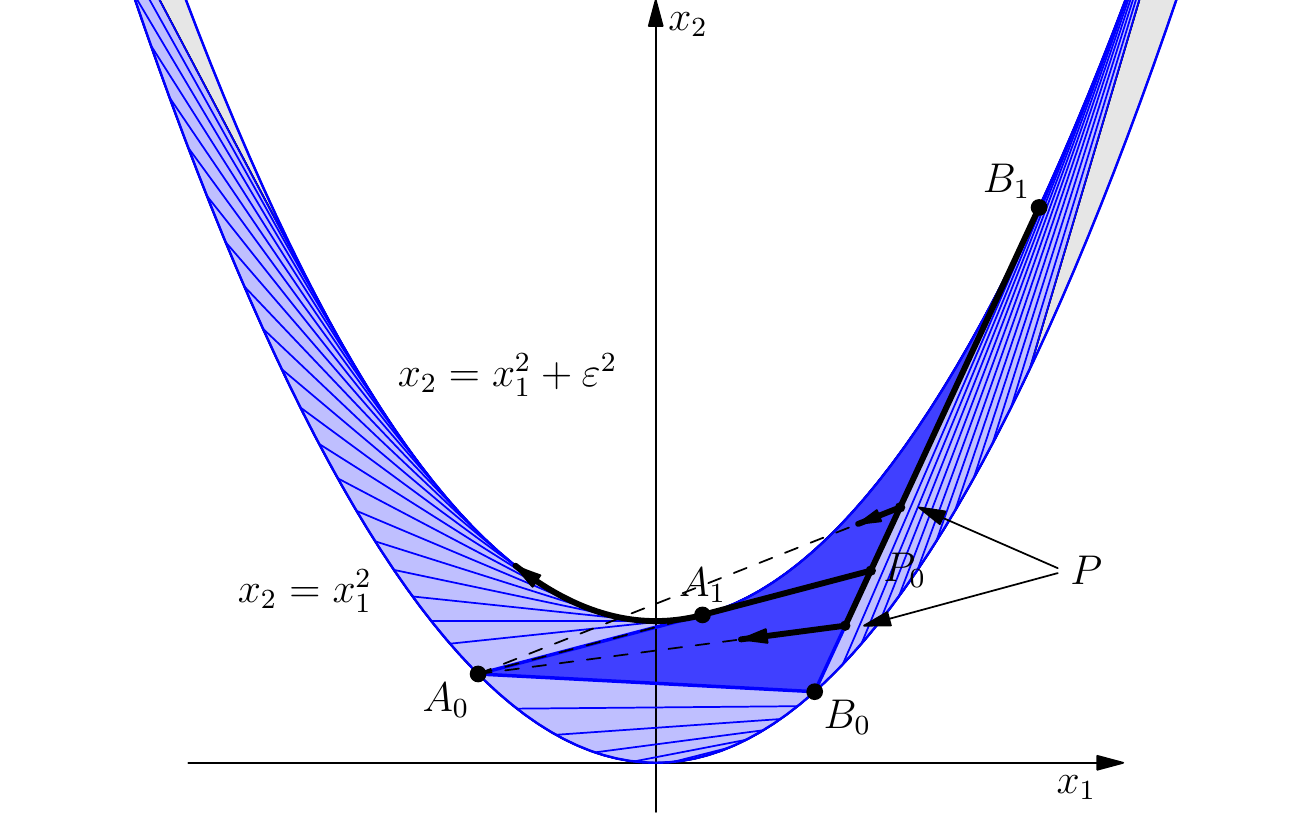}
\caption{A left trolleybus $\LTroll$ and  delivery curves.}
\label{fig:ltrkd}
\end{center}
\end{figure}
In left trolleybuses, delivery curves can be constructed exactly the same way. 
We omit detailed arguments for this case 
(however, Figure~\ref{fig:ltrkd} clarifies the matter entirely).

\subsection{Foliation in  general case}\label{s63}	
Now, using the components already constructed, we build a global Bellman 
candidate in the whole domain~$\Omega_\eps$.
First, we fix some signature~$\Sigma$ consisting of a finite number of  
symbols~$\mathrm{R}$ and~$\mathrm{L}$ that are arranged in an arbitrary order. 
We associate the pairs~$\mathrm{R}\mathrm{L}$ in this signature with angles, 
the pairs~$\mathrm{L}\mathrm{R}$ with full cups, and the pairs 
$\mathrm{R}\mathrm{R}$ and $\mathrm{L}\mathrm{L}$ with
trolleybuses (right and left, respectively) attached to cups 
(not necessarily full). We suppose these angles and cups are pairwise 
disjoint and arranged in the same order as the corresponding pairs of 
symbols in~$\Sigma$. We notice that all the domains located between 
them, together with two domains on the edges, have the form $\Rt$ or $\Lt$.
We assume that these domains are foliated by the suitable tangents.
Then, by one of Propositions~\ref{S5}, \ref{S7}, \ref{S15}, or~\ref{S16}, 
Bellman candidates are defined uniquely in these domains.  
If we now assume that near each angle and each cup the conditions of the 
corresponding proposition~--- either one of the propositions just
listed or Proposition~\ref{S6} about an angle~--- are satisfied, then we 
obtain some candidate~$B^{\Sigma}$ in the whole domain~$\Omega_\eps$.
It turns out that the Bellman function we are looking for has precisely 
such a form.	

\begin{figure}[H]
\begin{center}
\includegraphics{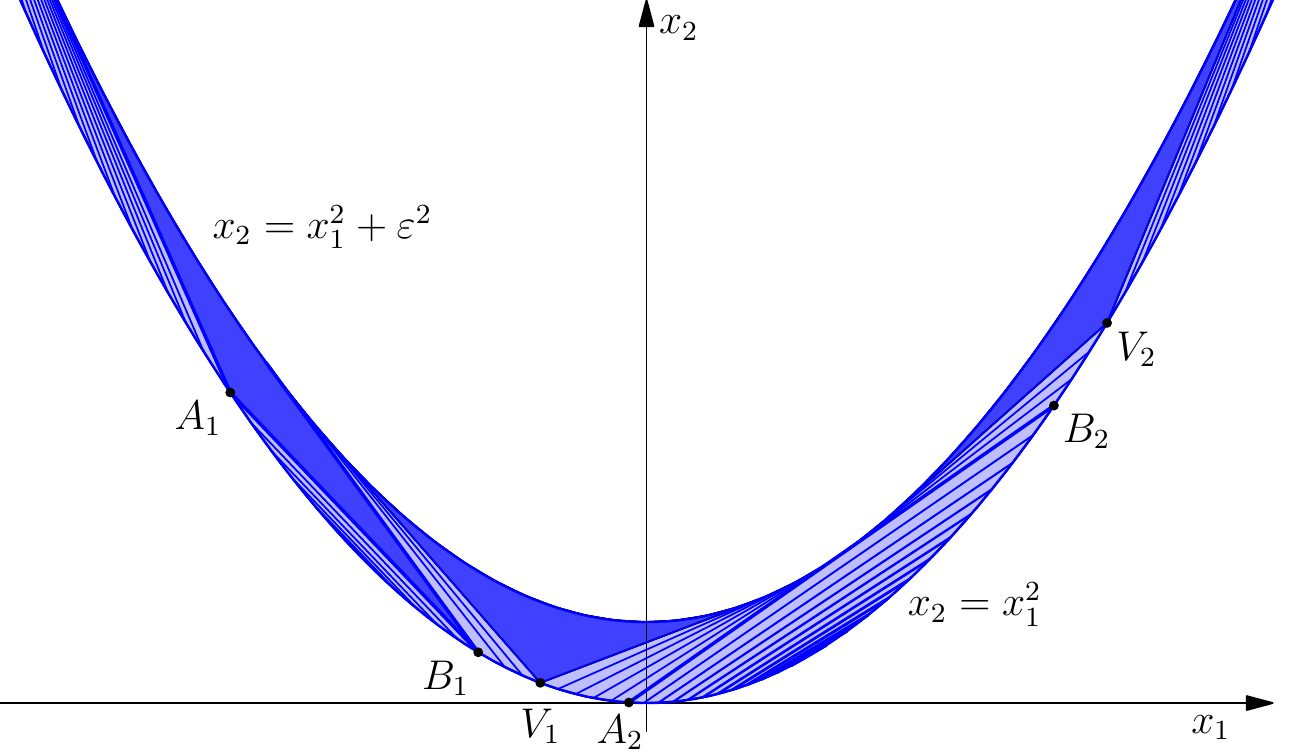}
\caption{A global candidate $B^{\mathrm{RRLRL}}$.}
\end{center}
\end{figure}	

\begin{Th}\label{T4}
Suppose $0<\eps<\eps_0$\textup, $N\in\mathbb{Z}_+$\textup, and $f\in\W^N$. 
Then we can choose a signature~$\Sigma$ such that a certain Bellman 
candidate~$B^\Sigma$ corresponds to it. In this case\textup, we have
$\Bell(x;\,f) = B^{\Sigma}(x)$.
\end{Th}
	
In order to prove this theorem, we need some preparation. First, we present
some new definitions.

Let $c$ be a point where the third derivative $f'''$ changes its sign from
$+$ to $-$. By Lemma~\ref{L5}, there exist continuously differentiable functions
$a(\ell)$ and $b(\ell)=a(\ell)+\ell$ on $[0,2\eps]$ that generate
a cup originated at $c$, i.e. $a(0)=b(0)=c$. 
Together with the pair $a$ and $b$ we will need another pair of mutually inverse
functions $\tilde a$ and $\tilde b$. These functions have the same values as $a$ and $b$, but 
their arguments are different. They are the first coordinates of the opposite
ends of the chord. Namely, each of the pairs $\{\tilde a(u),u\}$ and
$\{u,\tilde b(u)\}$ is a pair of points $\{a,b\}$ satisfying the cup equation~(\ref{urlun}).

Fixing the cup size $\ell$, we define the following function $D$ 
on  $[a(\ell),b(\ell)]$:
$$
D(u)\df
\begin{cases}
-\Dlt(u,\tilde b(u)),&\quad a(\ell)\le u<c;
\\
\phantom-\Drt(\tilde a(u),u),&\quad c<u\le b(\ell).
\end{cases}
$$
Recall that the differentials $\Dlt(a,b)$ and $\Drt(a,b)$ 
were introduced by~(\ref{e334}):
$$
\Dlt(a,b)=f''(a)-\Av{f''}{[a,b]},\qquad \Drt(a,b)=f''(b)-\Av{f''}{[a,b]}.
$$
The function $D$ can be naturally continued to $c$ by zero.
Note that the value of $\ell$ determines the domain of $D$ only, the value of 
$D$ at any fixed point does not depend on $\ell$.

\begin{Def}
The function  
$$
F(u,\ell)\df
\begin{cases} 
\displaystyle
e^{u/\eps}\Big[D(a(\ell))e^{-a(\ell)/\eps}
+\int_u^{a(\ell)}\!\!\!\!f'''(t)e^{-t/\eps}dt\Big], &u\in(-\infty,a(\ell));
\\
D(u), & u\in[a(\ell),b(\ell)]\,;
\\
\displaystyle
e^{-u/\eps}\Big[D(b(\ell))e^{b(\ell)/\eps}
+\int_{b(\ell)}^u\!\!f'''(t)e^{t/\eps}dt\Big], &u\in(b(\ell),+\infty)\,,
\end{cases}
$$
defined on the entire real axis,
will be called \emph{a force function} or simply \emph{a force}.
\end{Def}

It is natural to call the point $c$ \emph{the source} of the force~$F$. We call $[a(\ell),b(\ell)]$ \emph{the screen} of~$F$.
We refer to the left and right parts $[a,c]$ and $[c,b]$ as to \emph{the left and right screens} respectively.

Let us try to explain some mnemonic sense of the terminology introduced above.
We know that one can build a cup around the source of any force. The reader can
imagine that the force of any cup tries to push an angle away. As
a result, an angle between two cups will be placed exactly at the point
of balance of these two forces, where their sum is zero. There is
only one counter-intuitive point: a force pushing an angle to the right
is negative, and it is positive if it acts in the opposite direction.
The segment $[a(\ell),b(\ell)]$ is called a screen because it screens, in a
sense, the action of the force: we will see that the larger the screen is the
less (in absolute value) is the power off the screen.

Some special cases should be mentioned separately. The formula for a force without a screen 
(zero screen: $\ell=0$) is especially simple:
$$
F(u,0)=
\begin{cases} 
\displaystyle
e^{u/\eps}\int_u^c\!\!f'''(t)e^{-t/\eps} dt, & u\in(-\infty,c]\,;
\\
\displaystyle
e^{-u/\eps}\int_c^u\!\!f'''(t)e^{t/\eps} dt, & u\in[c,+\infty)\,.
\end{cases}
$$
Forces with sources 
at infinity provide the simplest cases of the last formula. If $c=+\infty$, then
$$
F(u,\ell)=
e^{u/\eps}\int\limits_u^{+\infty} f'''(t)e^{-t/\eps} dt,
\qquad u\in(-\infty,+\infty)\,,
$$
and if $c=-\infty$, then
$$
F(u,\ell)=
e^{-u/\eps}\int\limits_{-\infty}^u f'''(t)e^{t/\eps} dt,\qquad u\in(-\infty,+\infty)\,.
$$
These two expressions do not depend on $\ell$ (a finite screen at infinity
cannot touch finite points). Therefore, we can always assume that the forces
originated at infinity have zero screens, i.e. $\ell=0$.

Note that the formula for a force almost coincides with the second derivative of the coefficient~$m$ in the expression
for the Bellman function. Namely, formula~(\ref{e335}) for $\eps\mlt''(u,a(\ell))$ gives the
force on the left of the screen, and formula~(\ref{e336}) for $\eps\mrt''(u,b(\ell))$
coincides with the same force on the right of the screen. Thus,
we glue two expressions for $\eps\mrt''$ and $\eps\mlt''$, extending them
continuously to the screen $[a(\ell),b(\ell)]$.

We introduce a few more notions.
\begin{Def}
An interval $[\,c,t^+]$, where
$$
t^+=t^+(\ell)\df\sup\{t \mid F(s,\ell)\le0,\;\forall s,\; c\le s\le t\},
$$ 
is called \emph{the right tail} of the force $F$. \emph{The left tail} 
is an interval $[t^-,c\,]$, where 
$$
t^-=t^-(\ell)\df\inf\{t \mid F(s,\ell)\ge0,\;\forall s,\; t\le s\le c\}\,.
$$
The points $t^-(\ell)$ and $t^+(\ell)$ will be called the ends of the left
and right tails correspondingly.
\end{Def}

We recall that the requirements on the signs of $\mlt''(u;\,a_0)$ and
$\mrt''(u;\,b_0)$ appear in Propositions~$\ref{S7}$, $\ref{S15}$, and~\ref{S16},
where they guarantee the local concavity of the candidate on the left and
on the right of the cup. Note that the tails are the maximal intervals where 
the force has the sign required. They show the size of the maximal region around the cup that the tangents can foliate. 

Further, we recall that the equation
$$
\mrt''(v;\,b_1)+\mlt''(v;\,a_2)=0
$$
appears in Proposition~\ref{S6} as an equation for the vertex of the 
angle~$\Ang(v)$ and, also, in Propositions~\ref{S15} and~\ref{S16} as an 
equation for one of the trolleybus vertices. This inspires the following
definition.

\begin{Def}
Two forces are called \emph{balanced} if they satisfy two following conditions.  First, their tails have non-empty
intersection. Second, in this intersection we can choose a point~$v$ lying strictly between
the sources of the forces and such that
\begin{equation}
\label{balance_eq}
F_1(v,\ell_1)+F_2(v,\ell_2)=0\,.
\end{equation}
This point $v$ will be called \emph{the balance point} and the equation above will be called
\emph{the balance equation}. A family of forces is called  \emph{balanced}
if either this family consists of one element, or each pair of neighbor forces
is balanced.
\end{Def}

Suppose we have found a balanced family of forces with $2\eps$-screens such that 
the union of the tails covers the entire real axis. Moreover, let no balance point be inside
any screen. Then, as it was explained in the beginning of this section, we are 
done.  We have the desired foliation consisting of alternating cups and angles
with vertices at balance points. The corresponding function $B^\Sigma$
would be the desired Bellman candidate. However, if some points of balance 
are inside the screens, then such a family does not help to finish the construction
so quickly. The following definition helps us to overcome these difficulties.

\begin{Def}\label{polbalans}
A balanced family of force functions is called
\emph{completely balanced} if it is under two following conditions.  First, there are no balance points
inside any screen. Second, at least one end of each screen whose size
is less than $2\eps$ coincides with a balance point.
\end{Def} 

We are now ready to state the proposition that immediately implies Theorem~\ref{T4}.
	
\begin{Prop}\label{S17}
For any $f\in\W^N$, there exists a family of completely balanced forces
such that their tails cover the entire axis.
\end{Prop}

Now, we explain how to derive our theorem from this proposition.	
In the simplest case of one force, the Bellman function is already known 
(see Theorem~\ref{T1} for $c=-\infty$, Theorem~\ref{T1b} for $c=+\infty$,
and Theorem~\ref{T3} for a finite $c$). Thus, we consider the situation when 
we have several completely balanced forces whose tails cover the entire axis. 
For each force with a source at a finite point, we build a cup whose size
is equal to the size of the corresponding screen. If a cup is not full, then we have a balance point at least at one of its ends. 
We build the right or left trolleybus over such a cup depending on where (at what
end of the cup) we have a balance point. After that, we construct the angles 
with vertices at the remaining balance points. In such a way, we obtain a 
collection of disjoint constructions, which includes cups, trolleybuses,
and angles. We foliate all the remaining subdomains by the left or right tangents. By our definition of balance points and tails, and also by
Propositions~\ref{S5}, \ref{S6}, \ref{S7}, \ref{S15}, and~\ref{S16}, 
we obtain a Bellman candidate~$B^{\Sigma}$ with a corresponding 
signature~$\Sigma$ consisting of symbols~$\mathrm{R}$ and~$\mathrm{L}$. 
As usual, Statement~\ref{S2} implies the estimate $\Bell\le B^{\Sigma}$, 
and the reverse inequality $B^{\Sigma}\le\Bell$ follows from the existence 
of optimizers in each of the constructions involved	(the reader can easily 
imagine the delivery curves that originate in the full cups or $\pm\infty$,
``transit'' through the trolleybuses, and continue up to the angles).

\subsection{Properties of  force functions}

In this section, we investigate the properties of force functions,
needed to prove Proposition~\ref{S17}. 

\begin{Le}\label{strict}
The strict inequality $F<0$ is fulfilled at all interior points of the right tail\textup,
except possibly for the points where $f'''$ changes its sign from $+$ to $-$.
With the same possible exception, $F>0$ at each interior point of the left tail.
\end{Le}

\begin{proof}
The fact that the strict inequality is fulfilled in the screen, was proved in Lemma~\ref{L5}
($\Dlt<0$ and $\Drt<0$). Let $F(u_0,\ell)=0$ for some $u_0\in(b,t^+)$. Then
$$
F(u,\ell)=\int\limits_{u_0}^u f'''(t)e^{-(u-t)/\eps}dt
$$
in some neighborhood of $u_0$. Since $F(u,\ell)\le0$, the function $f'''$ must be 
non-positive in some right neighborhood of $u_0$ and non-negative in some 
left neighborhood, i.e. $u_0$ coincides with one of the points~$c_j$.
\end{proof}

We also prove two formulas we use for calculating derivatives of $F$.

\begin{Le}\label{dD}
\begin{align}
\label{dDlt}
d\Dlt(a,b)=\Big(f'''(a)+\frac{2\Dlt(a,b)}{b-a}\Big)da,
\\
\label{dDrt}
d\Drt(a,b)=\Big(f'''(b)-\frac{2\Drt(a,b)}{b-a}\Big)db.
\end{align}
\end{Le}

\begin{proof}
We begin with writing down the derivative of the cup equation~\eqref{e6}. We use
an invariant form not depending on the parametrization:
\begin{equation}
\label{dcup}
\Drt(a,b)db+\Dlt(a,b)da=0\,.
\end{equation}
Using this relation, we write down the differential of the average $\Av{f''}{[a,b]}$
in two different forms
\begin{align*}
d\Av{f''}{[a,b]}&=\frac{f''(b)db-f''(a)da}{b-a}-\frac{f'(b)-f'(a)}{(b-a)^2}(db-da)
\\
&=\frac{\Drt db-\Dlt da}{b-a}
=\frac{2\Drt db}{b-a}
=-\frac{2\Dlt da}{b-a}\,.
\end{align*}
This immediately yields both~\eqref{dDlt} and~\eqref{dDrt}.
\end{proof}

After this preparation, it is easy to find the partial derivative of the force
with respect to the screen size.

\begin{Le}\label{dlF}
$$
\frac\partial{\partial\ell}F(u,\ell)=
\Big(\frac2\ell-\frac1\eps\Big)
\begin{cases}
-D(a)a'e^{-(a-u)/\eps}, & u\in(-\infty,a(\ell))\,;
\\
\qquad 0, & u\in(a(\ell),b(\ell))\,;\rule{0pt}{17pt}
\\
-D(b)b'e^{-(u-b)/\eps}, & u\in(b(\ell),+\infty)\,,\rule{0pt}{17pt}
\\
\end{cases}
$$
or
$$
\frac\partial{\partial\ell}F(u,\ell)=
\Big(\frac2\ell-\frac1\eps\Big)\frac{\Dlt\Drt}{\Dlt+\Drt}
\begin{cases}
\phantom-e^{-(a-u)/\eps}, & u\in(-\infty,a(\ell))\,;
\\
\qquad 0, & u\in(a(\ell),b(\ell))\,;\rule{0pt}{17pt}
\\
-e^{-(u-b)/\eps}, & u\in(b(\ell),+\infty)\,.\rule{0pt}{17pt}
\end{cases}
$$
\end{Le}
\begin{proof}
We can easily check this formulas by direct calculation. Consider,
for example, the case $u<a$:
$$
\frac\partial{\partial\ell}F(u,\ell)=
e^{u/\eps}\Big[-\frac{d\Dlt(a,b)}{d\ell}e^{-a/\eps}
+\Dlt(a,b)\frac{a'}{\eps}e^{-a/\eps}
+f'''(a)e^{-a/\eps}a'\Big]\,.
$$
Using formula~\eqref{dDlt}, we simplify this expression:
$$
\frac\partial{\partial\ell}F(u,\ell)=
\Dlt(a,b)\Big(\frac1\eps-\frac2\ell\Big)e^{-(a-u)/\eps}a'.
$$
Thus, we have got the first representation of the derivative. To
obtain the second one, we must express $a'$
in terms of $\Dlt$ and $\Drt$. 
Taking into account that $db=d\ell+da$ and using~\eqref{dcup}, we have:
$$
a'=
-\frac{\Drt(a,b)}{\Dlt(a,b)+\Drt(a,b)}\,.
$$
Similarly, we can check the formulas for the case $u>b$.
\end{proof}

\begin{Cor}\label{preobrazovanie}
The force is strictly
increasing with respect to the screen size on the right of the screen and strictly decreasing on the left.
Inside the screen\textup, the force does not depend on this size.
\end{Cor}

\begin{proof}
In Lemma~\ref{L5}, it was proved that $\Dlt(a,b)<0$, $\Drt(a,b)<0$, $a'<0$,
and $b'>0$. Therefore, on the whole interval $\ell\in(0,2\eps)$ we have
\begin{align*}
\frac\partial{\partial\ell}F(u,\ell)>0&\qquad\text{for}\quad u>b\,;
\\
\frac\partial{\partial\ell}F(u,\ell)<0&\qquad\text{for}\quad u<a\,.
\end{align*}
Thus we are done.
\end{proof}

Some simple corollaries of this fact are listed below.

\begin{Cor}\label{rostxvostov}
The tails grow as the screen shrinks.
\end{Cor}

\begin{Cor}\label{noscreen} If $\ell>0$\textup, then
\begin{align*}
F(u,\ell)>F(u,0)&\qquad\text{for}\quad u>c\,;
\\
F(u,\ell)<F(u,0)&\qquad\text{for}\quad u<c\,.
\end{align*}
\end{Cor}
 
The last inequalities will be used together with the following relation
between two forces.

\begin{Le}\label{F12}
Let $F_1$ and $F_2$ be two forces with sources $c_1$ and $c_2$\textup,
$c_1<c_2$. Then two following relations between these forces are fulfilled\textup:
\begin{align*}
F_1(u,\ell_1)=e^{(c_2-u)/\eps}F_1(c_2,\ell_1)+F_2(u,0)\,,&\qquad u\ge c_2\,;
\\
F_2(u,\ell_2)=e^{(u-c_1)/\eps}F_2(c_1,\ell_2)+F_1(u,0)\,,&\qquad u\le c_1\,.
\end{align*}
\end{Le}

\begin{proof}
The statement of the lemma becomes trivial being rewritten by the 
definition of forces:
\begin{align*}
e^{-u/\eps}&\bigg[D_1(b_1)e^{b_1/\eps}
+\int\limits_{b_1}^u f'''(t)e^{t/\eps}dt\bigg]
\\
&=e^{(c_2-u)/\eps}e^{-c_2/\eps}\bigg[D_1(b_1)e^{b_1/\eps}
+\int\limits_{b_1}^{c_2} f'''(t)e^{t/\eps}dt\bigg]
+e^{-u/\eps}\int\limits_{c_2}^u f'''(t)e^{t/\eps}dt\,,
\\
&\hskip200pt u\in[c_2,+\infty)\,.
\end{align*}
The second identity is similar.
\end{proof}

We state the following simple corollary.
\begin{Cor}\label{ptintail}
Let $F_1$ and $F_2$ be two forces with sources $c_1$ and $c_2$\textup,
$c_1<c_2$. If $c_2$ gets into the tail of $F_1$\textup, then $F_1(u)\le F_2(u)$ for
$u\ge c_2$. If $c_1$ gets into the tail of $F_2$\textup, then the same inequality is true for $u\le c_1$.
\end{Cor}

\begin{proof}
If $c_2\le t_1^+$, then using Lemma~\ref{F12} and Corollary~\ref{noscreen}
we can write the following inequality:
$$
F_1(u,\ell_1)=e^{(c_2-u)/\eps}F_1(c_2,\ell_1)+F_2(u,0)\le F_2(u,0)\le F_2(u,\ell_2)
$$
for $u\in[c_2,+\infty)$.
In a similar way, if $c_1\ge t_2^-$, then
$$
F_2(u,\ell_2)=e^{(u-c_1)/\eps}F_2(c_1,\ell_2)+F_1(u,0)\ge F_1(u,0)\ge F_1(u,\ell_1)
$$
for $u\in(-\infty,c_1]$.
\end{proof}

Till now, we were investigating the dependence of a force from the size of its screen.
Now we treat the behavior of a force with respect to the first variable.

\begin{Le}\label{duF}
$$
\frac\partial{\partial u}F(u,\ell)= 
\begin{cases}
-f'''(u)+\eps^{-1}F(u,\ell), & u\in(-\infty,a(\ell))\,;
\\
\displaystyle
-f'''(u)+\frac2{\tilde b(u)-u}F(u,\ell), & u\in(a(\ell),c)\,;\rule{0pt}{25pt}
\\
\displaystyle
\phantom-f'''(u)-\frac2{u-\tilde a(u)}F(u,\ell), & u\in(c,b(\ell))\,;\rule{0pt}{25pt}
\\
\phantom-f'''(u)-\eps^{-1}F(u,\ell), & u\in(b(\ell),+\infty)\,.\rule{0pt}{20pt}
\\
\end{cases}
$$
\end{Le}

\begin{proof}
The formulas for the derivatives out of the screen are evident. We use Lemma~\ref{dD} to calculate $D'(u)$. 
On the left screen, we have
$a=u$, $b=\tilde b(u)$, and $D(u)=-\Dlt(u,\tilde b(u))$. Therefore, formula~\eqref{dDlt}
yields
$$
D'(u)=-f'''(u)+\frac{2D(u)}{{\tilde b-u}}\,.
$$
Similarly, using~\eqref{dDrt}, we get 
$$
D'(u)=f'''(u)-\frac{2D(u)}{u-\tilde a}
$$
on the right screen.
\end{proof} 

To determine balance points, we need to know the behavior of the sum of
two neighbor forces.

\begin{Cor}\label{carmoebuli}
If $F_1$ and $F_2$ are two forces with sources $c_1$ and $c_2$\textup,
$c_1<c_2$\textup, then
{\small
$$
\eps\frac\partial{\partial u}(F_1(u,\ell_1)+F_2(u,\ell_2))= 
\begin{cases}
\displaystyle
F_2(u,\ell_2)-\frac{2\eps}{u-\tilde a_1(u)}\,F_1(u,\ell_1), & u\in(c_1,b_1)\,;
\\
\quad F_2(u,\ell_2)-F_1(u,\ell_1), & u \in (b_1,a_2)\,;\rule{0pt}{20pt}
\\
\displaystyle\rule{0pt}{25pt}
\frac{2\eps}{\tilde b_2(u)-u}\,F_2(u,\ell_2)-F_1(u,\ell_1), & u\in(a_2,c_2)\,.
\end{cases}
$$}
\end{Cor}

\begin{Cor}\label{increasing}
If $F_1$ and $F_2$ are two forces with sources $c_1$ and $c_2$\textup,
$c_1<c_2$\textup, then the sum $F_1+F_2$ is strictly increasing in the
intersection of the right tail of $F_1$ and the left tail of $F_2$.
\end{Cor}

\begin{proof}
By the formula from the preceding corollary, we have
$\frac\partial{\partial u}(F_1+F_2)>0$ for all
$u\in(c_1,t_1^+)\bigcap(t_2^-,c_2)$, except possibly for a finite number of points
(see Lemma~\ref{strict}).
\end{proof}

\begin{Cor}\label{balance}
If $F_1$ and $F_2$ are two forces with sources $c_1$ and $c_2$ such that
${c_1<t_2^-\le t_1^+<c_2}$\textup, then the sum $F_1+F_2$ has exactly one root 
in the intersection of the tails\textup, $[t_2^-,t_1^+]$.
\end{Cor}

\begin{proof}
By the preceding corollary, the sum $F_1+F_2$ is strictly increasing on $[t_2^-,t_1^+]$. Therefore, since the continuous function 
$F_1+F_2$ has opposite signs at $t_2^-$ and $t_1^+$ (because $F_i(t_i^\pm)=0$), it has exactly one root on this interval.
\end{proof}

We conclude our investigation of the force functions with other two
important facts.

\begin{Le}\label{metia}
If the source of a force function belongs to a tail of another force\textup, 
then both tails of the first force are included in the tail of the second.
\end{Le}

\begin{proof} First, we note that both sources of two forces cannot be covered 
by the tails of each other. Indeed, if this occurs, the sum $F_1+F_2$ would be non-negative
at the left end of the segment $[c_1,c_2]$ ($F_1(c_1)=0$, $F_2(c_1)\ge0$) and non-positive
at its right end ($F_1(c_2)\le0$, $F_2(c_2)=0$). But since $F_1+F_2$ is strictly increasing on $[c_1,c_2]$ (see Corollary~\ref{increasing}), 
this is impossible.

Assume that $c_2$ lies in the right tail of $F_1$. We have to check that both 
tails of $F_2$ are in the right tail of $F_1$, i.e.
$[t_2^-,t_2^+]\subset[c_1,t_1^+]$. We have just proved that $c_1<t_2^-$. 
The second inequality $t_1^+>t_2^+$ is contained in Corollary~\ref{ptintail}.
The case where $c_1$ is in the left tail of $F_2$, can be treated similarly.
\end{proof} 

\begin{Le}
\label{balance2}
If two forces are balanced\textup, then the source of one of them cannot lie in a tail of another one.
\end{Le}

\begin{proof}
Let $c_1<c_2$. If we assume that $c_1\ge t_2^-$, then
$$
	F_1(c_1)+F_2(c_1)=F_2(c_1)\ge0.
$$ 
If $c_2\le t_1^+$, then
$$
	F_1(c_2)+F_2(c_2)=F_1(c_2)\le0.
$$ In any case, the sum $F_1+F_2$ 
cannot have a root on
$(c_1,c_2)$, i.e. the forces $F_1$ and $F_2$ cannot be balanced.
\end{proof}
\subsection{Algorithm}\label{Alg}

\paragraph{Cleaning.}
Consider some collection of points $\{c_k\}_{k=0}^N$ and forces $\{F_k\}$ generated by these points. Then we can remove from the collection those points $c_k$ that lie in a tail of some force function $F_j$, $j\ne k$. We call such an operation \emph{the cleaning}. We denote 
the set $\{c_{k_j}\}_{j=0}^m$ of points that remain after the cleaning by $\{c_{k}\}_{k=0}^N$, though the number $N$ may have changed. 
What is more, the symbol $c_k$ may denote another point after cleaning.

The union of forces' tails cannot become smaller after the cleaning. Indeed, the cleaning removes only those forces whose tails are contained entirely in a tail of some other force.

\paragraph{Compression.}

Let $\{F_k\}$ be a balanced collection of forces. Suppose some~$u_{j+1}$ --- the balance point of $F_j$ and $F_{j+1}$~--- got into the screen of $F_j$. We generate a new collection of forces by the following rule. First, we reduce the screen of $F_j$ in such a way that $u_{j+1}$ becomes the right end of this screen. The point $u_{j+1}$ remains to be a balance point of newly defined $F_j$ and old $F_{j+1}$. The reduction of the screen enlarges the tails of $F_j$, so they could cover some neighbor points $c_k$. Then we have to make the cleaning. The procedure just described is called 
\emph{the right compression}. A similar procedure (the decreasing of  $\ell_{j+1}$ and the cleaning), where $u_{j+1}$ gets into the screen of $F_{j+1}$, is called \emph{the left compression}. We note that the left compression can change the structure of the force collection only on the right of $u_{j+1}$, and the right compression does not change the structure of the forces on the right of $u_{j+1}$.

Indeed, consider the right compression. The new tail of the force cannot reach the point $c_{j+1}$ (see Lemma~\ref{balance2}), because the forces $F_j$ and $F_{j+1}$ are still balanced. So, all the forces on the right of $u_{j+1}$ remain the same. But what can happen on the left? Nothing can happen provided $j = 0$: either $c_0 = -\infty$ and there is nothing on the left, or the point $c_0$ is the last point and its left tail still reaches $-\infty$. But if $j>0$, the numeration of the remaining forces could change. Assume that the former point $c_j$ got a number $i$, $i \leq j$, after the compression. The balance point of the forces $F_i$ and $F_{i-1}$ could move only to the left, because the new force $F_i$ is not less than the old one (either by Lemma~\ref{preobrazovanie} if there was no cleaning, or by Corollary~\ref{ptintail} if the cleaning was performed). Consequently, the only new balance point that could  get inside a screen, is the point in the right screen of $F_{i-1}$. Thus, the new balance points cannot get into the left screens after the right compression. The only point that can get into the right screen, lies on the left of the compressed screen.

The situation is symmetric for the left compression. All the changes occur on the right of the screen being compressed. What is more, the only screen that can get a new balance point is the left screen of the first newly defined force on the right of the screen being compressed.

\paragraph{The whole algorithm }\hskip-8pt consists of a series of left compressions beginning from $F_1$ and going to the right, and the right compressions being performed from right to left. Of course, we can change the order of the left and the right compressions. We note, that in fact we begin not from the leftmost and rightmost forces, because there are no balance points both on the left of $c_0$ and on the right of $c_N$.  Indeed, either $c_{0}=-\infty$, or the left tail of $F_0$ fills  the ray $(-\infty,c_0]$. Similarly, either $c_{N}=+\infty$, or the right tail of $F_N$ fills the ray $[c_N,+\infty$).

Our algorithm begins with the cleaning of the family $\{F_k(u,2\eps)\}_{k=0}^N$. The tails of neighbor forces have non-empty intersection, because $[c_{j-1},v_j]$ lies in the tail of $F_{j-1}$, and $[v_j,c_j]$ lies in the tail of $F_j$. This property persists after the cleaning, and by Lemma~\ref{balance} we got a balanced family of forces. 

Thus, in order to prove that the algorithm provides a system balanced completely, it remains to verify that there are no balance points inside the screens. Indeed, one of the ends of the small screens (those that are smaller than $2\eps$) coincides with a balance point. Each small screen was compressed, so a balance point arrived at one of its ends. All the points that lied inside the left screens were sent to the boundary of their screens, as we performed the left compressions. Hence, we removed all the balance points from the left screens with the left compressions, what is more, this procedure did not send any balance points into the right screens. Similarly, the passage from right to left (execution of the right compressions) removed the balance points from the right screens and did not change the situation inside the left ones. So there are no balance points inside the screens, and we are done.

The only thing we have to mention is that the union of tails of the achieved collection coincides with the whole real line. This is 
a consequence of the fact that all the tails
of the initial family cover the whole line, and both the cleaning and the compression do not reduce this cover.

\subsection{Examples}
\paragraph{Example 8. A fifth-degree polynomial.}\label{ex8}

As usual, it is the third derivative that mainly influences the geometry of
extremals. So, we have to choose essential parameters in the formula for 
the third derivative, in order to deal with more pleasant expressions 
throughout our computations. Using Remarks~\ref{rem2} and~\ref{remm2}, 
we can investigate only the case 
$$
f=\frac1{60}t^5-\frac{d}{6}t^3.
$$
The easiest case appears when $d \leq 0$. Then we have $f''' \geq 0$, so the whole
parabolic strip is foliated by the left tangents, due to Corollary~\ref{cor1}. 
So we assume that $d > 0$. Therefore, the function~$f'''$ has two roots: 
\begin{equation}\label{240}
u_{\pm}= \pm\sqrt{d}.  
\end{equation}
In our case, the function~$f'''$ is positive on $(-\infty, u_-) \cup(u_+,+\infty)$ 
and negative on $(u_-,u_+)$. Thus, $c_0=u_-$, $v_1=u_+$, and $c_1=+\infty$.

Consider the force originated at $+\infty$:
\begin{equation}
\label{ml}
F_1(u)=e^{u/\eps}\int\limits_{u}^{+\infty}f'''(t)e^{-t/\eps}dt
=u^2\eps+2u\eps^2+2\eps^3-d\eps.
\end{equation}
By Theorem~\ref{T1b}, the whole domain $\Omega_{\eps}$ is foliated by the 
left tangents provided this value is non-positive everywhere, i.e. the left
tail of $F_1$ covers the entire real axis. It is clear that
\begin{equation}
F_1(u)\geq 0\qquad\Longleftrightarrow\qquad \eps\geq\sqrt{d}
\label{negeometria}.
\end{equation}
The last inequality can be easily reformulated in terms of the distance 
between the roots:

\begin{equation}
u_{+}-u_{-}\leq 2\eps. \label{geometria}
\end{equation}
So, if condition~(\ref{negeometria}) is fulfilled, then the whole parabolic 
strip is foliated by the left tangents, i.e. $\Bell=\Blt$. Now, we 
suppose
\begin{equation}\label{smalleps}
\eps < \sqrt{d}.
\end{equation}
It is worth mentioning that this condition is equivalent to $u_{+}-u_{-}> 2\eps$. 
In other words, $f \in \mathfrak{W}^{1}_{\eps}$ for every $\eps$. We can expect either
$\Bell=\Bltlt$ or $\Bell=B^{\mathrm{LRL}}$, because the case $\Bell=\Blt$ has
already been excluded. Therefore, we have to balance two forces: $F_0$ originated at 
$c_0=-\sqrt{d}$ and $F_1$ originated at $c_{1}=+\infty$. There are two options:
the balance point is either inside the screen of $F_0$ or outside it. In the first case
we have to choose $\ell$, $\ell\le2\eps$, (the size of the screen) such that the balance 
point $v$ coincides with the right end $b(\ell)$ of the screen, and $\Bell=\Bltlt$.
In this situation, the balance equation~(\ref{balance_eq}) is
\begin{equation}
\label{inside}
F_0(v,\ell(v))+F_1(v)=0\,.
\end{equation}
In the second case, the function~$F_0$ has $2\eps$-screen and the balance equation is
\label{outside}
\begin{equation}
F_0(v,2\eps)+F_1(v)=0\,.
\end{equation}

The function~$F_1$ is given by~\eqref{ml}. In order to avoid unnecessary computation, we do not write down 
the general expression for $F_0$. We have to calculate this function at the end of the right screen
for $\ell\in(0,2\eps]$ and in the right tail off the screen for $\ell=2\eps$. In any case, we need 
the expression for the right differential $\Drt$.

For this purpose, we write down equation~(\ref{urlun}) for the cup with origin $c_0=-\sqrt{d}$. 
After that, we express the left end of the cup in terms of the right one, i.e. find the function
$\tilde a$, and after that find the relation between all the parameters of the cup: $a$, $b$,
and $\ell$.

In our case, the cup equation
$$
\frac{f(a)-f(b)}{a-b} - \frac{f'(a)+f'(b)}{2}=0
$$
turns into 
$$
\frac{1}{120}(a-b)^{2} (3a^2+4ab+3b^2-10d)=0.
$$
Since $a \neq b$, we have two possible solutions:
$$
a_{\pm}(b) = \frac{-2b\pm\sqrt{30d-5b^2}}3.
$$
To satisfy the initial condition $\tilde a (u_{-})=u_{-}$, we have to choose $\tilde a = a_-$, i.e. 
\begin{equation}\label{levykonec}
\tilde a (b)\df \frac{-2b-\sqrt{30d-5b^2}}3.
\end{equation}
Solving the equation $b-\tilde a (b) =\ell$, we get
\begin{equation}\label{rtend}
b(\ell) = \frac12\ell-\sqrt{d-\frac1{20}\ell^2}.
\end{equation}
We have chosen the minus sign for the square root, because the chord must shrink to the 
origin of the cup  as $\ell\to 0$.
Let $b_*$ denote the right end of the $2\eps$-screen, then
\begin{equation}\label{zvezda}
b_* = \eps-\sqrt{d-\frac15\eps^2}.
\end{equation}
So we have to look for a solution of the equation
$$
F(v)=0\,,
$$
where
\begin{equation}\label{F}
F(v) =
\begin{cases}
F_0(v,\ell(v))+F_1(v)\,,\quad& -\sqrt d\le v\le b_*\,;
\\
F_0(v,2\eps)+F_1(v)\,,\quad& \phantom{-}b_* \le v<+\infty\,.
\end{cases}
\end{equation}

Calculating the expression for $\Drt(a,b)$ by definition~(\ref{e334}), we get
\begin{equation}
\label{Drt}
\Drt(a,b)=\frac1{30}(b-a)^2(3b+2a).
\end{equation}
We substitute $b=v$ and $a=\tilde a(v)$ from~(\ref{levykonec}) and see that
$$
F_0(v)=\frac1{405}\left[100v^3-225vd-(30d-5v^2)^{3/2}\right]
$$
for $v\le b_*$.

Thus, we have 
\begin{equation}
\label{phiminus}
F(v) =\frac1{405}\left[100v^3-225vd-(30d-5v^2)^{3/2}\right]+\eps\left[(v+\eps)^2+\eps^2-d\right]
\end{equation}
for $v\in[u_-,b_*]$.

Now, we suppose ${v\in [b_*,+\infty)}$.
Substituting $a_*=b_*-2\eps$ into~\eqref{Drt}, we get
$$
\Drt(a_*,b_*)=\frac{2\eps^{2}}{15}\bigg(\eps-5\sqrt{d-\frac15\eps^2}\;\bigg).
$$
The integral term in $F_0$ is
\begin{align*}
e^{-v/\eps}&\int\limits_{b^{*}}^v f'''(t)e^{t/\eps}\,dt
\\
&=\eps\big[(v-\eps)^2+\eps^2-d\big]-
\eps\big[(b^{*}-\eps)^2+\eps^2-d\big]e^{-(v-b^{*})/\eps}.
\end{align*}
As a result, we have the following formula for $v\in [b_*,+\infty)$:
\begin{equation}
\label{phiplus}
F(v) =-\frac{2\eps^2}3\bigg(\eps+\sqrt{d-\frac15\eps^2}\;\bigg)e^{-(v-b_*)/\eps}
+\left[2v^2+4\eps^2-2d\right]\eps.
\end{equation}

We are also interested in the critical value~$\eps_{*}$ that separates these two cases. 
The function $F(v)$ has a root on $[b^{*},+\infty)$ for $\eps < \eps_{*}$, and this case 
corresponds to~$B^{\mathrm{LRL}}$. For $\eps >\eps_{*}$, there is
a root on $[u_{-},b^{*}]$, and the desired solution is $\Bltlt$. In the boundary case 
$\eps = \eps_{*}$, we get the function $\Bltlt$ with the full cup and an angle adjacent to it. 
We obtain that critical value $\eps_*$ from the equation 
$$
F(b_*)=0.
$$

Although formulas~(\ref{phiplus}) and~(\ref{phiminus}) give the same value at $b_*$, 
it is more convenient to use the first one, because we have already used the identity 
$b_*-a_*=2\eps$. So,
$$
F(b_*) =-\frac{2\eps^2}3\bigg(\eps+\sqrt{d-\frac15\eps^2}\;\bigg)+\left[2b_*^2+4\eps^2-2d\right]\eps.
$$
Substituting~(\ref{zvezda}) for~$b_*$, we get
\begin{equation}
\label{eps-star}
F(b_*) =\frac{2\eps^2}{15}\bigg(37\eps-35\sqrt{d-\frac15\eps^2}\;\bigg).
\end{equation}
Consequently, the desired critical value $\eps_*$ is  
$$
\eps_{*} =\frac{35}{\sqrt{1614}}\sqrt d.
$$

First, we note that $\lim_{v \to +\infty}F(v)= +\infty$. Second, ${F(u_-)=F_1(u_{-})<0}$.
The first claim follows from~(\ref{phiplus}), the second is a consequence of
inequality~(\ref{smalleps}). Indeed, once it is fulfilled,  
the inequality $F_1(-\sqrt d)<0$ is satisfied by virtue of~(\ref{negeometria}).

Thus, the existence of the root on $(u_{-},+\infty)$ is clear. In order to get a more 
precise localization of the root, we transform $F(b_*)$ expressing $d$ in terms of 
$\eps_*$ in~(\ref{eps-star}):
$$
F(b_{*})=\frac{2\eps^2}{15}\bigg(37\eps-35\sqrt{\frac{1614}{35^2}\eps_*^2-\frac15\eps^2}\bigg)
= \frac{1076}5\cdot\frac{\eps^2(\eps^2-\eps_*^2)}{37\eps+\sqrt{1614\eps_*^2-245\eps^2}}.
$$
We see that the last expression is negative for $\eps < \eps_{*}$. Thus, the solution 
$B^{\mathrm{LRL}}$ is realized (i.e. the angle lies on the right of $b_{*}$), because 
the continuous function $F$ must have a root on $(b_{*},+\infty)$. In the case $\eps> \eps_{*}$, 
the solution is $\Bltlt$, because $F$ has a root on $(u_{-},b_{*})$.

We sum up our results in a proposition.

\begin{Prop}
Suppose $f(t) = t^{5} + pt^{4}+qt^{3}+P_{2}(t)$ is  a fifth-degree polynomial\textup, where 
$P_{2}$ is an arbitrary quadratic polynomial.
Then an analytic expression for the Bellman function~\textup{(\ref{e11})} is determined by 
the value $\frac{d}{\eps^2}$\textup, $d=\frac{p^2}{25}-\frac{q}{10}$\textup, as follows.
\begin{itemize}
 \item If  $\frac{d}{\eps^2} \leq 1$\textup, then Theorem~\textup{\ref{T1b}} 
 works\textup: $\Omega_{\eps}$ is fully foliated by left tangents\textup,  
 $\Bell(x) = \Blt(x;\,-\infty, +\infty)$.

\item If $1< \frac{d}{\eps^{2}}\leq \frac{1614}{1225}$\textup, then
$\Bell(x)=\Bltlt(x;\,-\infty,[\tilde a (v),v],+\infty)$ \textup(see Proposition~\textup{\ref{S16})}\textup, 
where $v$ is the root of~\textup{(\ref{phiminus})} and $a(v)$ is defined
 by~\textup{(\ref{levykonec})}. A single left trolleybus is realized in $\Omega_{\eps}$.   

\item  If $\frac{1614}{1225}<\frac{d}{\eps^{2}}$\textup, then $\Bell(x)=B^{\mathrm{LRL}}(x)$.
There is a full cup with the origin at~$u_{-}$ \textup{(}where $u_{-}$ is the left root of~$f'''$\textup{)}  
and a separated angle with the vertex at the root of~\textup{(\ref{phiplus})}.
\end{itemize}
\end{Prop}

\begin{Prop}
Suppose $f(t) = -t^{5} + pt^{4}-qt^{3}+P_{2}(t)$ is a fifth-degree polynomial\textup, where 
$P_{2}$ is an arbitrary quadratic polynomial.
Then an analytic expression for the Bellman function~\textup{(\ref{e11})} is determined by the value 
$\frac{d}{\eps^2}$\textup, where $d=\frac{p^2}{25}-\frac{q}{10}$.
In this case\textup, the extremals of $\Bell$ are symmetric to the extremals for 
$t^{5} + pt^{4}+qt^{3}$ with respect to the $y$-axis.
\begin{itemize}
 \item If $\frac{d}{\eps^{2}} \leq 1$\textup, then $\Bell(x) = \Brt(x;\,-\infty, +\infty)$.

\item  If  $1< \frac{d}{\eps^{2}}\leq \frac{1614}{1225}$\textup, then
 $\Bell(x)=\Brtrt(x;\,-\infty,[v,\tilde b (v)],+\infty)$.

\item If $\frac{1614}{1225}<\frac{d}{\eps^{2}}$\textup, then $\Bell(x)=B^{\mathrm{RLR}}(x)$.

\end{itemize}
\end{Prop}

We set $\rho = 0$ if the equation $f'''=0$ has no solutions. Otherwise, we set ${\rho = |u_{+}-u_{-}|}$, where
$u_{-}$ and $u_{+}$ are the roots of $f'''(t)=0$.
Now we can treat the statement about the fifth-degree polynomial in terms of $\rho$. If  
$\rho$ satisfies the inequality $\rho<2\eps$, then the domain is foliated by the tangents entirely
(their direction is determined by the leading coefficient of the polynomial). Next, if 
$\rho=2\eps$, there appears a zero-length trolleybus in the point where $f'''$ changes its 
sign from $+$ to $-$ (notice that the condition $\rho=2\eps$ makes the segment, connecting 
the roots of the third derivative of the polynomial, to touch the upper parabola). 
Further, in the case  $2\eps<\rho<2\frac{\sqrt{1614}}{35}\eps$,  the trolleybus occurs. 
In the case  $\rho =2\frac{\sqrt{1614}}{35}\eps$, the trolleybus sits on the full cup. If 
$\rho> 2\frac{\sqrt{1614}}{35}\eps$, there is a full cup and an angle separated from it. 

The reader is welcome to watch a series of pictures (see Figure~\ref{Series}), where we fix one root of $f'''$ (e.g.~$u_{-}$), 
while the second one $u_{+}$ runs away from it. The red line denotes the segment connecting 
$U_{-}$ and $U_{+}$.

\begin{figure}[H]
\begin{center}
\includegraphics[scale=0.75]{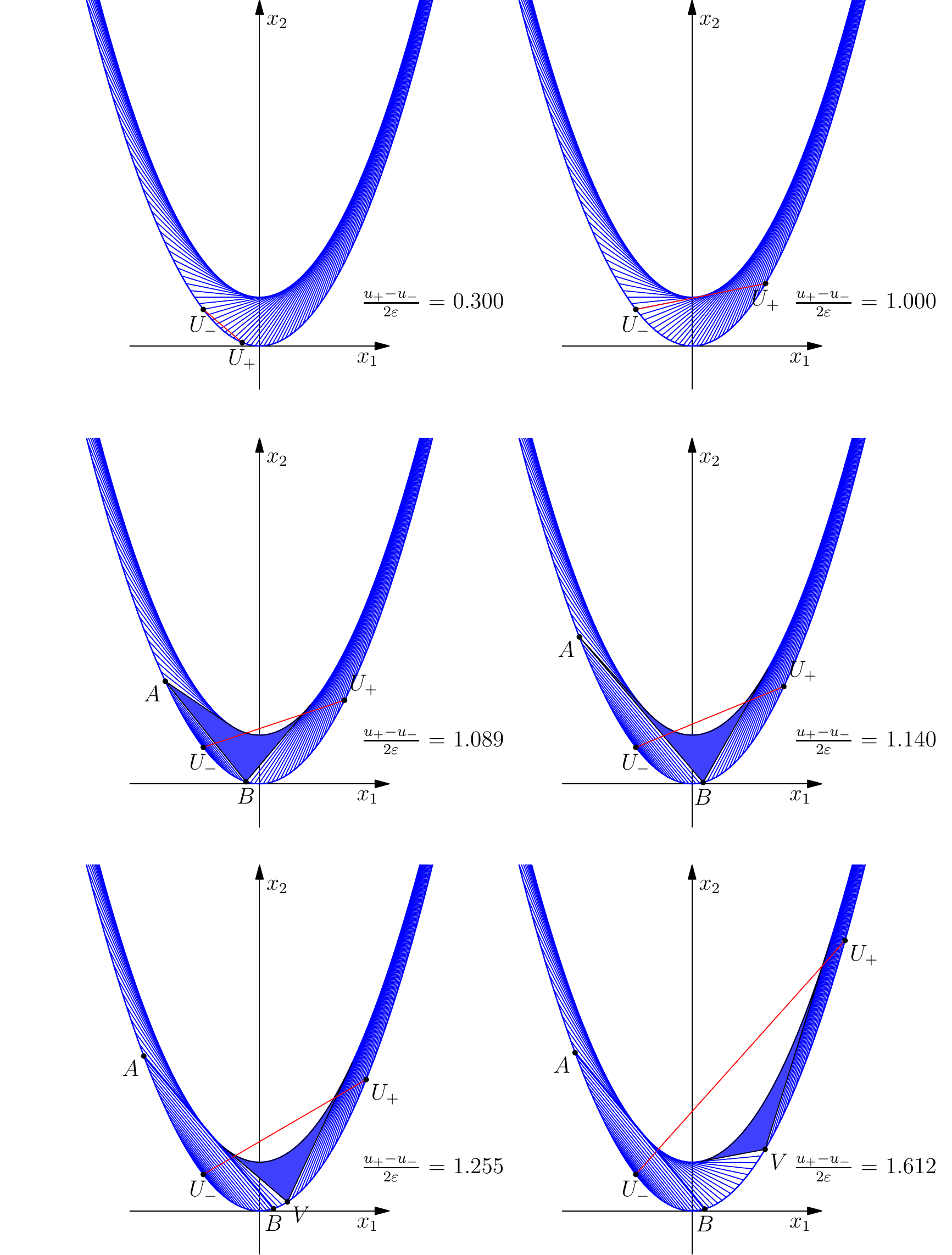}
\caption{Pictures for Example~8.}
\label{Series}
\end{center}
\end{figure}
\section{Conclusion}\label{s7}

We conclude the paper with a brief description of results we achieved and knowledge we acquired while writing it. But first we describe the things we understood from the beginning but did not write down in order to simplify our arguments.

We have not described all the geometric structures of extremals that can occur. To avoid the mixing of the cups, we have assumed that the roots of~$f'''$ are well separated. A figure that arises if two cups meet together is called a multicup. Another new figure is a birdie. Though, it occurs in this text implicitly: it is the union of a non-full cup and two angles adjacent to it from both sides (if there is only one angle, we have a trolleybus). Formally speaking, we can say that we have already considered this figure. Indeed, we can treat it as a union of a trolleybus and an angle; the tangent domain between them has reduced to a single extremal. However, it is more convenient to think of this construction as a figure of some new type. The reason is that it has its own dynamical properties: it can either be stable or break into a trolleybus and an angle.

Now we discuss the dynamics in $\eps$, i.e. we fix a boundary function~$f$ and observe the evolution of the foliation. For $\eps$ small enough, the picture is relatively simple: we have a sequence of alternating cups and angles that sit near the roots of $f'''$. When $\eps$ increases, the cups grow and the angles move from side to side. We know that when some angle meets a cup, they form a trolleybus. But there can be more difficult constructions when several figures meet together (e.g., a birdie or a multicup mentioned above), and we will study their evolutional properties in the forthcoming paper.
What is more, we will provide another algorithm for the calculation of our Bellman function, based on the evolutional approach. It will allow us to calculate the Bellman function for all $\eps$ simultaneously and to find critical values of $\eps$ (we already saw such a value $\eps_{*}$ in the last example) in which the structure of the picture changes.  

Surely, we will abandon the absurd condition $f'''\ne 0$ a.e. and, therefore, ``thick'' roots (intervals where 
$f'''\equiv 0$) will appear.
This will bring a multicup built on a continuum of zero cups. We did not consider this case here in order to avoid multicups.

Also we are going to make the function $f$ less smooth. Some Bellman functions with boundary functions that have jumps are already known. In general, we understand the nature of the subject, though not all the formal proofs are still written down. We have a hypothesis that the geometric picture for non-smooth $f$ can be obtained by passing to the limit in an appropriate sequence of smooth functions. For example, this approach explains the fact that if 
the function~$f$ has a jump at some point, then for all $\eps$ there is a singular cup. Indeed, we have 
$f''' = \delta''$ at that point, so at least two changes of sign are ``compressed'' there. 

Finally, we should mention the study of the Bellman function behavior at the limit value $\eps_0$ of the parameter $\eps$. By this we mean that for all $\eps$ greater than $\eps_0$, the Bellman function is infinite. Surprisingly, the set of $\eps$ for which the Bellman function is finite, can be both closed and open. The case of an open set happens, for example, for the integral John--Nirenberg inequality. We want to achieve the Bellman function not only for all values~$\eps$, $\eps < \eps_0$, but also for $\eps = \eps_0$. 

To end up the conclusion, we say a few words about what lies beyond $\BMO$. It is widely known that for the John--Nirenberg inequality and for the reverse H\"older inequality for the $A_p$-weights, the Bellman functions can be
constructed similarly (compare \cite{Va,Va2,SlVa} and \cite{Va3,Va4,rez}). So we will employ the technique designed in
this paper not only for the parabolic strip, but in a much more general setting.


\begin{thebibliography}{9999}
  
  \bibitem{Ivo} I.~Klemes, \emph{A mean oscillation inequality}, Proceedings of the AMS, vol.~93, no.~3, 1985.
  \bibitem{Koosis} P.~Koosis, \emph{Introduction to $H^p$ spaces}, Cambridge University press, 1998.
  \bibitem{Kor} A.~A.~Korenovski\u \i, \emph{On the connection between mean oscillation and exact integrability classes of functions},  
    Math. of the USSR-Sbornik, vol.~71, no.~2 (1992), 561--567.
	\bibitem{NaTr} F.~Nazarov, S.~Treil,
		\emph{The hunt for a Bellman function\textup: applications to estimates of 
		singular integral operators and to other classical problems of harmonic analysis},
		St.~Petersburg Math. J., vol.~8 (1997), 721--824.

  \bibitem{NTV} F.~Nazarov, S.~Treil, A.~Volberg, \emph{
   Bellman function in stochastic optimal control and harmonic analysis
   \textup(how our Bellman function got its name\textup)}, Oper. Theory: Advances and Appl.
   vol.~129 (2001), 393--424, Birkhauser Verlag.
  \bibitem{Pog} A.~V.~Pogorelov, Differential geometry, ``Noordhoff'' 1959.

	\bibitem{rez}A.~Reznikov, \emph{
   Sharp weak type estimates for weights in the class $A_{p_1,p_2}$},
  submitted to Revista Matematica Iberoamericana. 
	\bibitem{SlVa2} L.~Slavin and V.~Vasyunin,
		\emph{Sharp $L^p$ estimates on $\BMO$},
		to appear in Indiana University Mathematics Journal,\\
\href{http://www.iumj.indiana.edu/IMJU/Preprints/4651.pdf}{http://www.iumj.indiana.edu/IMJU/Preprints/4651.pdf}.
	\bibitem{SlVa} L.~Slavin, V.~Vasyunin,
		\emph{Sharp results in the integral-form John--Ni\-ren\-berg inequality},
		Trans. Amer. Math. Soc., vol.~363, no.~8 (2011), 4135--4169;
		preprint, 2007, 
		\href{http://arxiv.org/abs/0709.4332}{http://arxiv.org/abs/0709.4332}
  \bibitem{Stein} I.~M.~Stein, \emph{Harmonic analysis, real-variable methods, orthogonality and oscillatory integrals}, 
    Princeton University press, 1993.
  \bibitem{Va4} V.~I.~Vasyunin, \emph{
  	Mutual estimates of $L^p$-norms and the Bellman function}, 
  	J. of Math. Sci., vol.~156, no.~5 (2009), 766--798.
	\bibitem{Va} V.~Vasyunin, \emph{
    Sharp constants in the classical weak form of the John--Ni\-ren\-berg inequality.}
    preprint POMI, no.~10, 1--9, 2011,\\
    \href{http://www.pdmi.ras.ru/preprint/2011/eng-2011.html}{http://www.pdmi.ras.ru/preprint/2011/eng-2011.html}    
	\bibitem{Va2}	V.~Vasyunin,
		\emph{The sharp constant in the John--Nirenberg inequality},
		preprint POMI no.~20, 2003.
  \bibitem{Va3} V.~Vasyunin, \emph{
		The sharp constant in the reverse H\"older inequality for Muckenhoupt weights}, 
		St.~Petersburg Math. J., vol.~15 (2004), 49--79.
	\bibitem{VaVo} Vasily Vasyunin and Alexander Volberg,
		\emph{Monge--Amp\`ere Equation and Bellman Optimization of Carleson Embedding Theorems},
		Amer. Math. Soc. Transl. Ser.~2, vol.~226 (2009), 195--238.
\end{thebibliography}
\end{document}